\documentclass{article}
\usepackage{graphicx} 
\usepackage[a4paper, margin=1in]{geometry} 
\RequirePackage{etex}
\usepackage[english]{babel}
\usepackage{subcaption}
\usepackage{mathrsfs}
\usepackage[all]{xy}
\usepackage{color}
\usepackage{url}
\usepackage{indent first}
\usepackage[labelfont=bf,labelsep=period,justification=raggedright]{caption}
\usepackage{quantikz}
\usepackage{amsthm}
\usepackage{amsmath}
\usepackage{amsfonts}
\usepackage{amssymb}
\usepackage{tikz}
\usetikzlibrary{arrows.meta, positioning, bending}
\usepackage{float}
\usepackage{xcolor}
\usepackage{subcaption}
\theoremstyle{plain}
\newtheorem{Theorem}{Theorem}[section]
\newtheorem{cor}[Theorem]{Corollary}

\newtheorem{prop}[Theorem]{Proposition}
\newtheorem{lemma}[Theorem]{Lemma}

\newtheorem{defn}[Theorem]{Definition}

\newtheorem{remark}[Theorem]{Remark}
\newtheorem{ex}[Theorem]{Example}

\makeatletter
\renewcommand{\theTheorem}{%
  \thesection
  \ifnum\value{subsection}>0
    .\arabic{subsection}%
  \fi
  .\arabic{Theorem}}
\makeatother
\usepackage{graphicx} 
\usepackage[utf8]{inputenc}
\usepackage[ backend=biber]{biblatex}
\usepackage{abstract} 
\usepackage[colorlinks=true, linkcolor=blue, citecolor=blue, urlcolor=blue]{hyperref}
\usepackage{authblk} 
\renewbibmacro{in:}{}
\title{Perfect State Transfer on Quotient Graphs 
in Shunt Decomposition-Based Quantum Walks}
                  
\author[1]{Banita Katuwal\thanks{Corresponding author. Email: banitakatuwal@sssihl.edu.in}}
\author[1]{Srinath M~S\thanks{Email: srinathms@sssihl.edu.in}}
\author[1]{Y~Lakshmi Naidu\thanks{Email: ylakshminaidu@sssihl.edu.in}}
\author[2]{Supriyo Dutta\thanks{Email:dosupriyo@gmail.com}}

\affil[1]{Department of Mathematics and Computer Science\\
Sri Sathya Sai Institute of Higher Learning\\
Andhra Pradesh, India -- 515001}

\affil[2]{Department of Mathematics\\
National Institute of Technology Agartala\\
Jirania, West Tripura, India -- 799046}
\date{}
\begin{document}

\maketitle
\begin{abstract}
This paper investigates perfect state transfer (PST) in 
discrete-time quantum walks constructed via the shunt 
decomposition method. The walks are defined on a graph $G$ 
and its associated quotient graph $G/\pi$, induced by an 
equitable partition $\pi$. 
Through the shunt decomposition of $G$, we derive an explicit 
relation between the shift operator of the parent graph $G$ 
and that of its quotient graph $G/\pi$. We construct a reflection operator based on the characteristic matrix, 
which establishes a connection between the transition operator of the 
parent graph and that of its lower-dimensional quotient graph. 
We then prove that PST occurs on $G$ if and only if it occurs on 
$G/\pi$. Furthermore, we express the unitary evolution 
operator of the quotient graph in terms of Chebyshev polynomials of 
the first kind, from which we derive explicit criteria for PST. As an 
application, we establish PST on the cycle graph $C_{n}$ at time 
$k = n/2$, and lift the result to the parent graph $C_{2n}$ via the 
equitable partition \(\pi\). We further show that if an equitable partition 
$\pi$ of $G$ induces a quotient isomorphic to $K_n^{\circlearrowleft}$, 
the complete digraph on $n$ vertices with a loop at every vertex, then 
PST occurs at step $k = n$, and the walk is periodic at $k = 2n$. This framework is applied to two families 
of graphs, which are the complete bipartite digraph $K_{n,n}^{\rightleftharpoons}$ 
and the circulant graph $\operatorname{Circ}(2n, S)$, where $S$ 
consists of all odd residues modulo $2n$ and $n = 2^s$ for some 
$s \geq 1$, establishing PST in their respective line digraphs. Collectively, these results also answer the  question posed by 
Godsil and Zhan concerning which shunt decompositions or 
embeddings of a graph admit PST. \\~\\


\end{abstract}

\section{Introduction}
Quantum walks are the quantum analogue of classical random walks~\cite{portugal2013quantum} and serve as a 
primary framework for designing quantum algorithms~\cite{ambainis2003quantum}. 
Beyond their algorithmic utility, they have been proven universal for quantum 
computation, implying that any quantum circuit can be encoded into a quantum walk 
evolution~\cite{childs2009universal}. Notably, quantum walk-based search algorithms 
provide a quadratic speedup over classical search, analogous to the speedup achieved 
by Grover's algorithm~\cite{grover1996fast}. In fact, both the continuous- and discrete-time quantum 
walk models can be used to derive and justify Grover's search algorithm as a 
special case of quantum walk dynamics on graphs~\cite{childs2009universal, godsil2023discrete}. Quantum walks are broadly categorized 
into two primary frameworks: {continuous-time quantum walks (CTQW)}~\cite{farhi1998quantum, childs2009universal,coutinho2016graph} and {discrete-time 
quantum walks (DTQW)~\cite{Aharonov2001QuantumWO,ambainis2003quantum}. In a CTQW, the evolution is driven by a Hamiltonian typically 
derived from the adjacency or Laplacian matrix of the underlying graph. In contrast, 
a DTQW evolves through the repeated application of a unitary coin operator followed 
by a shift operator.

Within the discrete-time setting, several models have been developed to describe the walker's dynamics, including the arc-reversal, two-reflection, sedentary, and vertex-face models~\cite{godsil2019discrete, zhan2021quantum}. Among these, the shunt-decomposition model stands out as one of the most computationally efficient~\cite{wing2023circuit,wing2025circuit} frameworks in which the evolution alternates between a coin unitary and a shift
that moves the walker along an outgoing arc inside a given shunt or arc-class~\cite{godsil2019discrete, godsil2023discrete}. Originally introduced by Aharonov \emph{et al.}~\cite{Aharonov2001QuantumWO}, this model was later reformulated in a combinatorial context~\cite{godsil2019discrete,godsil2023discrete}. They explored its spectral properties and established conditions for uniform average mixing using the Grover coin. Due to its structural simplicity, the shunt-decomposition model is particularly well-suited for quantum circuit implementations~\cite{wing2023circuit,11497753} and has been applied to diverse problems, such as the development of quantum search complement algorithm~\cite{wing2025circuit}, quantum channel~\cite{11497753}.

A central phenomenon in quantum-walk-based protocols is PST~\cite{godsil2012state}. PST occurs when a quantum state localized at a vertex $a$ is transferred to a vertex $b$ with unit fidelity at a specific time $t\in \mathbb{R}$. In CTQW, this is expressed as:
\[ | \langle e_b \mid e^{-iAt} \mid e_a \rangle | = 1, \]
while in DTQW, it requires:
\[ | \langle e_b \mid U^k \mid e_a \rangle | = 1, \]
for some step $k \in \mathbb{Z}$. A substantial body of work has identified families of graphs that exhibit PST in both continuous and discrete settings~\cite{zhan2019infinite, kubota2022perfect, ge2011perfect, chan2023pretty, guo2024perfect,coutinho2015perfect, dutta2026perfect}. One of the method for analyzing PST is the theory of {quotient graphs}~\cite{godsil2001algebraic, godsil1993algebraic}, where a large, highly symmetric graph can be reduced (collapsed) into a smaller graph while preserving quantum walk behavior a technique known as {path collapsing}. 

This argument was used by Christandl \emph{et al.}~\cite{christandl2004perfect} to show that weighted paths exhibit PST, derived from the fact that the unweighted $n$-dimensional hypercube $Q_n$ has PST and can be collapsed into a weighted path. This follows earlier work by Childs \emph{et al.}~\cite{childs2003exponential} in the context of exponential algorithmic speedup for graph search problems, observing that CTQW on unweighted layered graphs has polynomial hitting times due to their behavior on corresponding weighted paths.
In the continuous-time framework, Bachman \emph{et al.}~\cite{bachman2011perfect} established a fundamental equivalence: a graph $G$ admits perfect state transfer (PST) if and only if its quotient graph $G/\pi$, modulo an equitable partition $\pi$, also admits PST. This theorem provided a systematic methodology for constructing graphs where PST occurs between vertices that lack a ``swapping'' automorphism. This was instrumental in addressing a significant open question posed by {Godsil}~\cite{godsil2012state}, who had demonstrated that while PST implies the equality of stabilizer subgroups ($\text{Aut}(X)_u = \text{Aut}(X)_v$) and the identity of distance partitions ($\Delta_u = \Delta_v$), it remained unclear whether a global swapping automorphism was strictly necessary for the phenomenon to occur. 

{Bachman} \emph{et al.} provided the definitive counter-example to this conjecture by constructing a graph using two non-isomorphic regular graphs of the same valency and size. By joining these graphs such that the distance partition of a vertex $a$ remained equitable, they proved that PST can occur between vertices $a$ and $b$ even when no automorphism maps one to the other. Their work further showed that Feder’s graphs are quotients of a $k$-fold Cartesian product of PST graphs and provided an extensive treatment of graphs whose quotients are weighted $P_4$ paths. This confirmed that state transfer is governed by spectral properties and equitable structures rather than global symmetries. Subsequently, Coutinho and Godsil~\cite{coutinho2016graph} reinforced these findings by identifying PST timings for specific quotient structures, such as the $d$-cube at $t = \pi/2$ and the complete bipartite graph $K_{2,n}$ at $t = \pi/\sqrt{2n}$. Similarly, Yang \emph{et al.}~\cite{ge2011perfect} used generalized path-collapsing to revisit graphs of diameter three and compare them to weighted paths $P_4$, providing PST conditions for paths $P_4(\gamma, \kappa)$ with middle edge weight $\gamma$ and internal self-loops with weight $\kappa$. Furthermore, Kim \emph{et al.}~\cite{kim2024generalization} introduced the notion of $s$-pair state transfer, providing a systematic method for identifying PST by analyzing lower-dimensional quotient matrices.

While the relationship between quotient graphs and PST is 
well-established in the continuous-time framework, the discrete-time 
counterpart remains a significant research gap. This motivates our 
central research question \textit{``Is it possible to define a 
discrete-time quantum walk on a graph $G$ such that PST occurs if 
and only if it occurs on the quotient graph $G/\pi$?''} This approach 
offers two key advantages. First, it allows PST to be studied on 
smaller quotient graphs, which in turn facilitates the identification 
of PST in larger graphs; that is, establishing PST on a quotient graph 
$G/\pi$ is equivalent to demonstrating PST on the corresponding larger 
graph $G$. Second, it addresses the question raised by Godsil and Zhan 
in~\cite[\textit{Sec.}~12]{godsil2019discrete} \textit{``Which 
shunt-decompositions or embeddings give perfect state transfer?''} To this end, we employ the shift operator arising from the shunt decomposition of a graph and define a reflection operator for both $G$ and 
$G/\pi$.

Moreover, Zhan~\cite{Zhan2025uniform} previously 
considered $Q^*UQ$ as the transition matrix relative to $\pi$ and 
provided conditions for pretty good state transfer in this context. 
Crucially, they showed that if $\pi$ is an equitable partition of 
$G$, then $\pi'$ is an induced equitable partition of the line digraph 
$LD(G)$. In the discrete framework, Krovi and Brun~\cite{krovi2007quantum}, studied quantum walks on quotient graphs using graph automorphisms. They demonstrated that symmetries of a subgroup $H$ induce invariant subspaces, and restricting the walk to such subspaces yields an equivalent walk on a quotient graph obtained by identifying vertex and edge orbits. They observed that quotient graphs with significantly fewer vertices, such as those reducing to paths in hypercubes and glued trees, can lead to exponentially faster hitting times. This highlights the fundamental role of symmetry and quotient graph structure in optimizing quantum walk dynamics.

The remainder of this paper is organized as follows. 
Section~\ref{section 2} presents the preliminaries required for the 
subsequent developments. Section~\ref{section 3} establishes the 
relationship between the shift matrix of a graph $G$ and its quotient 
graph $G/\pi$ via the shunt-decomposition model. 
Subsection~\ref{subsection 3.1} facilitates the transition to arc 
partitions by demonstrating that $G$ is isomorphic to the arc 
partition of its line digraph $LD(G)$, establishing that
\[
LD(G)/\tau \cong G, \qquad LD(G/\pi)/\sigma \cong G/\pi,
\]
where $\tau$ is the full arc equitable partition of $LD(G)$, $\pi$ is the 
vertex equitable partition of $G$, and $\sigma$ is the arc equitable partition of 
$LD(G/\pi)$. Section~\ref{section 4} defines the transition matrices 
of $G$ and its quotient $G/\pi$. Subsection~\ref{subsection 4.1} 
derives the relationship between these transition matrices and 
establishes the condition under which PST occurs in $G$ if and only 
if it occurs in $G/\pi$. Section~\ref{section 5} defines the 
transition matrix on the line digraph $LD(G)$ and on $LD(G/\pi)$, 
and provides the condition under which $LD(G/\pi)/\sigma$ exhibits PST if 
and only if $LD(G)$ exhibits PST. Section~\ref{section 6} expresses 
the unitary evolution operator in terms of Chebyshev polynomials of 
the first kind, yielding a representation analogous to the Grover 
walk and providing criteria for PST in quotient graphs. 
Section~\ref{section 7} establishes PST on the quotient of the cycle 
graph $C_{2n}$ with even $n$ at time $k = n/2$, and uses the results 
of Subsection~\ref{subsection 4.1} to derive the PST condition on the 
parent graph $C_{2n}$ at $k = n/2$. Section~\ref{section 8} shows 
that if a graph reduces to the quotient graph $K_n^{\circlearrowleft}$, 
that is, the complete digraph with loops at every vertex, then under a 
specific characteristic matrix $\widetilde{Q}$, PST occurs at step 
$k = n$ and the walk is periodic at $k = 2n$. 
Sections~\ref{section 9} and~\ref{section 10} reduce the complete 
bipartite digraph $K_{n,n}^{\rightleftharpoons}$ and the circulant 
graph $\operatorname{Circ}(2n, S)$ to quotient graphs isomorphic to 
$K_n^{\circlearrowleft}$, and establish that the line digraphs 
$LD(K_{n,n}^{\rightleftharpoons})$ and $LD(\operatorname{Circ}(2n,S))$ 
exhibit PST at $k = n$ and are periodic at $k = 2n$. Finally, 
Section~\ref{section 11} summarizes the conclusions and outlines 
directions for future research. T
\section{Preliminaries} \label{section 2}
\subsection{Equitable Partitions}
The concept of equitable partitions in algebraic graph theory has been 
studied for a long time; foundational results on this topic can be found 
in~\cite{godsil2001algebraic, godsil1993algebraic}. More precisely, let $G = (V, E)$ be 
a connected graph, let $\pi = \{C_1, C_2, \dots, C_r\}$ be a partition 
of the vertex set $V(G)$ into $r$ cells, and let $Q \in \{0,1\}^{|V| 
\times r}$ be the \emph{characteristic matrix} associated with $\pi$, 
which is defined as
\begin{equation}
    Q_{uj} =
    \begin{cases}
        1, & \text{if } u \in C_j, \\
        0, & \text{otherwise.}
    \end{cases}
\end{equation}
A partition $\pi$ is called \emph{equitable} if, for all 
$i, j \in \{1, \dots, r\}$, every vertex $u \in C_i$ has the same 
number of neighbours in $C_j$, denoted by $b_{ij}$, independent of 
the choice of $u$. Equivalently, if $\pi$ is equitable, then
\begin{equation}\label{eq:equitable}
    AQ = Q\widetilde{A},
\end{equation}
where $A$ is the adjacency matrix of $G$ and 
$\widetilde{A} \in \mathbb{R}^{r \times r}$ is the \emph{quotient matrix} 
of $G$ with respect to $\pi$~\cite[Chapter~5, Lemma~2.2]{godsil1993algebraic}. The matrix 
$\widetilde{A} = (b_{ij})$ records the number of edges from each 
vertex in cell $C_i$ to cell $C_j$. In particular, the 
diagonal entry $b_{ii}$ counts the number of neighbours of any 
vertex $u \in C_i$ that lie within the same cell $C_i$. Note that 
when $i = j$, each edge inside $C_i$ connects two vertices of the 
same cell and contributes exactly $1$ to the neighbour count of 
each of its endpoints; if $b_{ii} = 0$, there are no edges within 
$C_i$, and the cells are said to be \emph{independent sets}. The 
graph $\widetilde{G}$ with adjacency matrix $\widetilde{A}$ is 
called the \emph{quotient graph} of $G$ with respect to $\pi$, 
denoted $G/\pi$.

The following lemma gives three equivalent characterisations of 
equitable partitions that will be used throughout.

\begin{lemma}[Lemma 3.1.1 {~\cite{coutinho2016graph}}]
\label{lem:equitable_equiv}
Let $\pi$ be a partition of $V(G)$ with normalized characteristic matrix $Q$. 
Then the following statements are equivalent:
\begin{enumerate}
    \item The partition $\pi$ is equitable.
    \item The column space of $Q$, denoted $\operatorname{col}(Q)$, 
          is invariant under $A$.
    \item The matrices $A$ and ${Q}{Q}^\top$ commute.
\end{enumerate}
\end{lemma}

\subsection{Shunt Decomposition Walks and Perfect State Transfer (PST)}
The shunt decomposition provides a systematic method for partitioning the arc set of a \(d\)-regular
directed graph (where each vertex has $d$ in-neighbors and $d$ out-neighbors) into disjoint permutations of the vertex set. This decomposition plays a fundamental role
in the construction of discrete-time quantum walks, as it allows the adjacency matrix to be expressed as
a sum of permutation matrices. These matrices directly form the basis of the shift operator \(S\) and the
transition matrix \(U\).

A \emph{shunt} on a directed graph $G$ is a permutation of $V(G)$ where each vertex maps to one of its out-neighbors. For a $d$-regular directed graph, it comprises exactly $d$ shunts. A \emph{shunt decomposition} of $G$ partitions its arcs into such shunts. This yields permutation matrices $P_1, \dots, P_d \in \mathbb{C}^{n \times n}$ such that
\begin{equation}
A = \sum_{j=1}^{d} P_j,
\end{equation}
 where $A$ is the adjacency matrix of $G$.  The existence of such a decomposition is guaranteed by Lemma~7.1.1 in~\cite{godsil2023discrete}.

In the discrete-time quantum walk, the state space is the Hilbert space
\(
\mathcal{H} = \mathbb{C}^n \otimes \mathbb{C}^d,
\)
where $n = |V(G)|$ and $d$ is the degree of the graph $G$ (or equivalently, the coin dimension). The unitary evolution operator is given by
\[
U = SC,
\]
where $S$ is the shift operator and $C$ is the coin operator. Using the natural isomorphism between $\mathbb{C}^d\otimes \mathbb{C}^n$ and $\mathbb{C}^n\otimes \mathbb{C}^d$, the operator $U$ can be expressed as
\begin{equation}
U = \left( \sum_{j=1}^{d}  E_{jj} \otimes  P_j \right)
    \left( \sum_{v \in V(G)}  C_v \otimes E_{vv} \right),
\end{equation}
where $P_j$ are permutation matrices acting on $\mathbb{C}^n$, $E_{jj}$ and $E_{vv}$ are matrix units, and each $C_v$ is a coin operator acting on $\mathbb{C}^d$. The shift operator $S$ is defined by
\begin{equation}\label{shift_matrix_}
S = \sum_{j=1}^{d}E_{jj}\otimes  P_j ,
\end{equation}
which can be written in block diagonal form as
\begin{equation}
S =
\begin{pmatrix}
P_{1} & 0      & \cdots & 0 \\
0      & P_{2} & \cdots & 0 \\
\vdots & \vdots & \ddots & \vdots \\
0      & 0      & \cdots & P_{d}
\end{pmatrix},
\end{equation}
where $E_{jj}$ is the $d \times d$ matrix with a $1$ in the $(j,j)$-entry and zeros elsewhere. The operator $S$ shifts the amplitude associated with label $j$ according to the permutation $P_j$, as described in~\cite[\textit{ Sec.} 2.2 ]{godsil2019discrete} and  ~\cite[\textit{Sec.} 7.1]{godsil2023discrete}.
This framework is particularly useful for studying equitable partitions: if the permutation matrices $P_j$ respect a partition $\pi$, then the operator $U$ reduces to a smaller operator $\widetilde{U}$ acting on the quotient graph, thereby confining the quantum dynamics to a lower-dimensional invariant subspace.
The state evolves in time according to
\[
\psi_k = U^k x.
\]
Following~\cite{chan2023pretty}, we say that there is PSTfrom a unit state $x \in \mathcal{H}$ to $y \in \mathcal{H}$ if there exists an integer $k$ such that
\begin{equation}
U^k x = \gamma y,
\end{equation}
for some $\gamma \in \mathbb{C}$ with $|\gamma| = 1$. Equivalently,
\begin{equation}
|\langle U^k x, y \rangle| = 1.
\end{equation}

\subsection{Chebyshev Polynomials of the First Kind}\label{Chebyshev Polynomials}

The Chebyshev polynomials of the first kind, denoted by $T_n(p)$, form a sequence of orthogonal polynomials that play a significant role in the spectral analysis of quantum walks. They are defined recursively by
\[
T_0(p) = 1, \quad T_1(p) = p,
\]
and for $n \geq 2$,
\begin{equation}
T_n(p) = 2pT_{n-1}(p) - T_{n-2}(p).
\end{equation}
A fundamental property of these polynomials is their trigonometric representation. For $p \in [-1,1]$, let $p = \cos \theta$. Then,
\begin{equation}
T_n(\cos \theta) = \cos(n\theta).
\end{equation}
This identity immediately implies that
\begin{equation}
|T_n(p)| \leq 1 \quad \text{for all } p \in [-1,1],
\end{equation}
a property that ensures the stability of amplitudes in quantum walk dynamics.

In the context of state transfer, Kubota and Segawa~\cite{kubota2022perfect} employed Chebyshev polynomials of the first kind to analyze perfect state transfer (PST) between vertex-type states, that is, states associated with the vertices of a graph. By expressing powers of the transition operator in terms of these polynomials, they derived necessary conditions on the eigenvalues of the underlying graph. Their results show that PST between vertex-type states can occur only if the spectral structure of the graph satisfies certain algebraic constraints determined by the roots and values of the Chebyshev polynomials $T_n$.

\section{ Quantum Walk on Quotient Graphs via Shunt Decomposition}\label{section 3}
In this section, we develop a framework for the symmetrized quotient graph, denoted $G/\pi$, using a vertex equitable partition. This approach allows us to relate the shift matrix of the parent graph and its quotient graph within the context of the shunt decomposition model.

Let $G=(V,E)$ be a finite $d$-regular directed graph with adjacency matrix $A$. Let $\pi$ be an equitable partition of $V$ into $|\pi|$ cells, and let $Q \in \mathbb{C}^{|V| \times |\pi|}$ be its normalized characteristic matrix satisfying $Q^\top Q = I_{|\pi|}$. By definition of an equitable partition, we have the intertwining relation $AQ = Q\widetilde{A}$, where $\widetilde{A}$ is the adjacency matrix of the quotient graph $G/\pi$.
\begin{cor}\label{cor:3.1}
Let $G$ be a $d$-regular directed graph with equitable partition
$\pi = \{C_1, \ldots, C_r\}$ and quotient matrix $\widetilde{A}$.
If $|C_1| = |C_2| = \cdots = |C_r| $,
then $G/\pi$ is $d$-regular.
\end{cor}
\begin{proof}
Let $u \in C_i$. Since $G$ is $d$-regular, $u$ has exactly $d$
out-neighbours distributed across the cells, with $b_{ij}$ of them
falling in $C_j$, so $\sum_{j=1}^{r} b_{ij} = d$ for all $i$. Now
fix $C_j$. Let $|C_1| = |C_2| = \cdots = |C_r|=\alpha$, each of the $\alpha$ vertices in $C_j$ has in-degree $d$,
so the total number of directed edges entering $C_j$ is $\alpha d$.
Each cell $C_i$ contributes exactly $\alpha\,b_{ij}$ directed edges
into $C_j$, so
\(
\sum_{i=1}^{r} \alpha\,b_{ij} = \alpha d,
\)
we get $\sum_{i=1}^{r} b_{ij} = d$ for all $j$.
Hence both row and column sums equal $d$, and $G/\pi$ is $d$-regular.
\end{proof}

Assume that $G$ and $G/\pi$ are both $d$-regular directed graphs, 
so that both $A$ and $\widetilde{A}$ admit a shunt decomposition 
of the form
\begin{equation}
A = \sum_{j=1}^{d} P_j, \qquad \widetilde{A} = \sum_{j=1}^{d} 
\widetilde{P}_j,
\end{equation}
where $P_j \in \{0,1\}^{|V|\times|V|}$ are permutation matrices on $V$, and $\widetilde{P}_j \in \mathbb{C}^{|\pi|\times|\pi|}$ are matrices acting on the quotient cells. Suppose these shunts are consistent with the partition such that for each $j \in \{1, \dots, d\}$, the following holds:
\begin{equation}
    P_j Q = Q \widetilde{P}_j.
    \label{eq:intertwine}
\end{equation}
Using the definition of the shift operator~\eqref{shift_matrix_}, together with \eqref{eq:intertwine}, we obtain a corresponding reduction of this operator to the quotient space.

\begin{lemma}\label{lemma 3.1}
Let $G$ be a $d$-regular directed graph with equitable partition
$\pi = \{C_1, \ldots, C_r\}$ of equal cell size,
and normalized characteristic matrix $Q$. Then
\[
(I_d \otimes Q^\top)\, S\, (I_d \otimes Q) = \widetilde{S}
\]
if and only if $P_j Q = Q\widetilde{P}_j$ for each $j = 1, \ldots, d$, 
where
\(
S = \sum_{j=1}^d E_{jj} \otimes P_j
\)
is the shift matrix of $G$ and
\(
\widetilde{S} = \sum_{j=1}^d E_{jj} \otimes \widetilde{P}_j
\)
is the quotient shift matrix of \(G/\pi\).
\end{lemma}

\begin{proof}
$(\Rightarrow)$ By the equitability of $\pi$, we have \(AQ=Q\tilde A\). Since by Corollary~\ref{cor:3.1}, the quotient graph $G/\pi$ is
$d$-regular. We have 
$\left(\sum_{j=1}^d P_j\right) Q = Q \left(\sum_{j=1}^d \widetilde{P}_j\right)$. 
Assume $P_j Q = Q\widetilde{P}_j$ for each $j = 1,\ldots,d$. 
Then
\begin{equation}\label{permuatation}
\sum_{j=1}^{d} E_{jj} \otimes P_j Q 
= \sum_{j=1}^{d} E_{jj} \otimes Q\widetilde{P}_j.
\end{equation}
Applying $(I_d \otimes Q^\top)$ on the left of ~\eqref{permuatation} and using $Q^\top Q = I_{|\pi|}$, 
we obtain
\begin{equation}\label{quotient_permutation}
\sum_{j=1}^{d} E_{jj} \otimes Q^\top P_j Q 
= \sum_{j=1}^{d} E_{jj} \otimes \widetilde{P}_j.
\end{equation}
Hence by ~\eqref{quotient_permutation} the quotient shift matrix satisfies
\begin{equation}
\widetilde{S} 
= \sum_{j=1}^d E_{jj} \otimes (Q^\top P_j Q) 
= \sum_{j=1}^d (E_{jj} \otimes Q^\top)\, S\, (E_{jj} \otimes Q) 
= (I_d \otimes Q^\top)\, S\, (I_d \otimes Q).
\end{equation}

$(\Leftarrow)$ Assume $(I_d \otimes Q^\top)\,S\,(I_d \otimes Q) = \widetilde{S}$. 
Expanding via the mixed-product property and using the linear independence 
of $\{E_{jj}\}$, we obtain
\[
Q^\top P_j Q = \widetilde{P}_j \qquad \text{for each } j.
\]
Left-multiplying by $Q$ gives $QQ^\top P_j Q = Q\widetilde{P}_j$. Since 
$\pi$ is equitable, by Lemma~\ref{lem:equitable_equiv} the projector 
$QQ^\top$ commutes with each $P_j$, that is, $QQ^\top P_j = P_j QQ^\top$. 
Therefore
\[
P_j Q = Q\widetilde{P}_j, \qquad j = 1,\ldots,d. \qquad 
\]
\end{proof}

\begin{cor}\label{involutary}
Under the assumptions of Lemma~\ref{lemma 3.1}, the quotient shift
operator satisfies $\widetilde{S}^2 = I$ if and only if
$\widetilde{P}_j^2 = I$ for all $j = 1, \ldots, d$.
\end{cor}

\begin{proof}
Recall that $\widetilde{S} = \sum_{j=1}^d E_{jj} \otimes \widetilde{P}_j$.
Using the orthogonality of the coin-basis projectors,
$E_{jj}E_{kk} = \delta_{jk}E_{jj}$, we compute
\begin{equation}
\widetilde {S}^2
= \sum_{j=1}^d \sum_{k=1}^d (E_{jj}E_{kk}) \otimes
(\widetilde {P}_j\widetilde {P}_k)=  \sum_{j=1}^d E_{jj} \otimes \widetilde{P}_j^2.
\end{equation}
$(\Rightarrow)$ Suppose $\widetilde{S}^2 = I_{d|\pi|}$, that is,
\[
\sum_{j=1}^d E_{jj} \otimes \widetilde{P}_j^2
= \sum_{j=1}^d E_{jj} \otimes I_{|\pi|}.
\]
Since the matrices $\{E_{jj}\}_{j=1}^d$ are linearly independent,
we may equate the $j$-th block on each side to conclude
$\widetilde{P}_j^2 = I_{|\pi|}$ for all $j = 1, \ldots, d$.

$(\Leftarrow)$ Suppose $\widetilde{P}_j^2 = I_{|\pi|}$ for all
$j = 1, \ldots, d$. Then
\begin{equation}
\widetilde{S}^2
= \sum_{j=1}^d E_{jj} \otimes \widetilde{P}_j^2
= \sum_{j=1}^d E_{jj} \otimes I_{|\pi|}
= I_d \otimes I_{|\pi|}
= I_{d|\pi|}.
\end{equation}
This completes the proof. 
\end{proof}
\subsection{Transition to  Arc Partitions}\label{subsection 3.1}
\begin{defn}
Let $G = (V,E)$ be a directed graph, where each element of $E$ is an ordered pair $(u,v)$, called an \emph{arc} (that is, a directed edge from $u$ to $v$). The \emph{line digraph} of $G$, denoted by $LD(G)$, is the directed graph whose vertex set is $E(G)$, so that each vertex corresponds to an arc $(u,v)$ of $G$. Two vertices $(u,v)$ and $(v,w)$ in $LD(G)$ are adjacent if and only if $(u,v), (v,w) \in E(G)$.
\end{defn}
Thus, the direction in $LD(G)$ encodes the natural progression of a walk along consecutive arcs of $G$, where the terminal vertex of one arc coincides with the initial vertex of the next. To redefine the quantum walk model on arc partition, we utilize the structural relationship between a $d$-regular directed graph $G$ and its line digraph $LD(G)$. In this framework, the state space is
\[
\mathcal{H} = \mathbb{C}^{|E(G)|},
\]
where each basis state corresponds to an arc of $G$. The evolution of the quantum walk is governed by the adjacency matrix of the line digraph, denoted by $A_{LD(G)}$, which encodes transitions between arcs. Specifically, the walker evolves from an arc $(u,v)$ to a succeeding arc $(v,w)$, reflecting the adjacency condition in $LD(G)$.

Let
\[
\pi=\{C_0,C_1,\dots,C_r\}
\]
be an equitable partition of \(V(G)\), and let \(Q\) be its normalized characteristic matrix of \(G\). For each pair \(0\le i,j\le r\), define
\[
C'_{ij}=\{(u,v)\in E(G):u\in C_i,\ v\in C_j\}.
\]
The nonempty sets \(C'_{ij}\) form a partition of \(E(G)\). Let \(Q'\) be the normalized characteristic matrix of this induced arc partition, defined by
\[
Q'_{(u,v),(i,j)}=
\begin{cases}
\dfrac{1}{\sqrt{|C'_{ij}|}}, & \text{if }(u,v)\in C'_{ij},\\[1ex]
0, & \text{otherwise.}
\end{cases}
\]
Since the nonempty arc-cells are disjoint, the columns of \(Q'\) are orthonormal, and hence
\[
{Q'}^\top Q'=I.
\]

We recall the following lemma, which holds for any graph $G$.
\begin{lemma}[{~\cite[Lemma 7.1]{Zhan2025uniform}}]\label{lemma 7.1}
Let \(\pi = \{C_0, C_1, \ldots, C_r\}\) be an equitable partition of \(G\). Define arc-cell
\[
C'_{ij} \;=\; \{(v_i, v_j)\in E(G) : v_i\in C_i,\ v_j\in C_j\}.
\]
Then \(\pi'=\{C'_{ij}\}_{0\le i,j\le r}\) is an equitable partition of \(LD(G)\).
\end{lemma}

\begin{cor}\label{cor:shunt-intertwine}
Let $G$ be a $d$-regular digraph  with 
equitable partition $\pi = \{C_1, \ldots, C_r\}$ of equal 
cell size, and let $Q$ be the normalized characteristic 
matrix of $\pi$ . For each 
$\ell = 1, \ldots, d$, let $G_\ell = (V, E_\ell)$ be the 
directed graph induced by the $\ell$-th arc class. Define 
arc cells
\[
C'_{ij} = \{(v_i, v_j) \in E(G_\ell) : v_i \in C_i,\ 
v_j \in C_j\}.
\]
Then $\pi'_\ell = \{C'_{ij}\}$ is an equitable arc partition of 
$\mathrm{LD}(G_\ell)$. Moreover, for each $\ell = 1, \ldots, d$,
\(
P_\ell\, Q = Q\, \widetilde{P}_\ell,
\)
where $P_\ell$ is the shunt of $G$ associated to the 
$\ell$-th arc class, and $\widetilde{P}_\ell$ is the 
corresponding shunt of the quotient graph $G/\pi$.
\end{cor}

\begin{proof}
Fix $\ell \in \{1, \ldots, d\}$. Since $P_\ell$ is a 
permutation matrix, $G_\ell$ is $1$-regular, each vertex 
$v \in V$ has exactly one outgoing arc $(v, P_\ell(v))$, 
so $A(\mathrm{LD}(G_\ell)) = P_\ell$. By 
Lemma~\ref{lemma 7.1}, the arc cells $C'_{ij}$ form an 
equitable partition of $\mathrm{LD}(G_\ell)$. Any arc 
$(v_i, P_\ell(v_i)) \in C'_{ik}$ has its unique 
out-neighbor $(P_\ell(v_i), P_\ell^2(v_i))$ in 
$C'_{ks}$, so the out-neighbor count into $C'_{\ell s}$ 
equals $1$ if $\ell = k$ and $0$ otherwise, independent 
of the arc chosen. Since $\pi$ is equitable and $P_\ell$ is consistent with 
the arc structure, $P_\ell$ maps each cell $C_k$ 
$(k \in \{1,\ldots,r\})$ entirely into a single cell 
$C_{\sigma_\ell(k)}$, defining a well-defined permutation 
$\sigma_\ell$ on $\{1, \ldots, r\}$, and $\widetilde{P}_\ell$ 
is the corresponding $r \times r$ permutation matrix.
For the $k$-th column $Q_k$ of $Q$, since $P_\ell$ maps 
$C_k$ bijectively onto $C_{\sigma_\ell(k)}$ with 
$|C_k| = |C_{\sigma_\ell(k)}| = m$, we have
\(
P_\ell\, Q_k = Q_{\sigma_\ell(k)} = 
Q\,\widetilde{P}_\ell\, e_k.
\)
Since this holds for every $k \in \{1, \ldots, r\}$, 
we conclude
\begin{equation}\label{eq:shunt-intertwine}
P_\ell Q = Q\widetilde{P}_\ell.
\end{equation}
\end{proof}
\begin{remark}\label{rem:shunt-intertwine-general}
The intertwining relation~\eqref{eq:shunt-intertwine} holds for
any $d$-regular digraph admitting a shunt decomposition and an
equitable partition. Equitability forces each shunt to map every
cell entirely into a single cell, so
Corollary~\ref{cor:shunt-intertwine} applies to each shunt
individually. Summing over all $d$ shunts then recovers the
full quotient intertwining.
\end{remark}

We now record a particularly useful special case, in which the arc partition 
of $LD(G)$ is indexed in a structured way by the vertex partition of $G$.

Let $G=(V,E)$ be a finite $d$-regular directed graph with $2n$ vertices and 
let $\pi=\{C_1,\dots,C_n\}$ be an equitable partition of $V$ with normalized 
characteristic matrix $Q\in\mathbb{R}^{2n\times n}$. Suppose the arc set 
$E$ is written as a disjoint union of $r$ arc-classes
\[
E \;=\; \mathcal{A}^{(1)} \;\dot{\cup}\; \mathcal{A}^{(2)} \;\dot{\cup}\; 
\cdots \;\dot{\cup}\; \mathcal{A}^{(r)}.
\]
For each $s=1,\dots,r$, define the induced arc-partition of $\mathcal{A}^{(s)}$
coming from $\pi$ by
\[
\pi'^{(s)} \;=\; \bigl\{\,C_{ij}^{\prime\,(s)} : 1\le i,j\le n\,\bigr\}, 
\qquad
C_{ij}^{\prime\,(s)} \;=\; 
\{(u,v)\in\mathcal{A}^{(s)} : u\in C_i,\ v\in C_j\}.
\]
The nonempty sets $C_{ij}^{\prime\,(s)}$ form a partition of 
$\mathcal{A}^{(s)}$. Let $m_s = |\mathcal{A}^{(s)}|$. The normalized 
characteristic matrix $Q'^{(s)} \in \mathbb{R}^{m_s \times k_s}$ (where 
$k_s$ is the number of nonempty cells in $\pi'^{(s)}$) is defined by
\[
Q'^{(s)}_{(u,v),(i,j)} =
\begin{cases}
\dfrac{1}{\sqrt{|C_{ij}^{\prime\,(s)}|}}, 
& \text{if } (u,v)\in C_{ij}^{\prime\,(s)},\\[1ex]
0, & \text{otherwise.}
\end{cases}
\]
Thus each column is the normalized indicator vector of a cell, and the 
columns of $Q'^{(s)}$ are orthonormal.

The combined arc partition of $E$ into atomic cells is
\[
\tau_{\mathrm{atom}} \;=\; \bigcup_{s=1}^{r} \pi'^{(s)},
\]
that is, $\tau_{\mathrm{atom}}$ consists of all nonempty cells 
$C_{ij}^{\prime\,(s)}$ for $1\le i,j\le n$ and $1\le s\le r$.

A \emph{coarsening} of $\tau_{\mathrm{atom}}$ is any partition obtained by 
grouping together some of these atomic cells. In particular, for each vertex 
$v_p \in V(G)$ with $1\le p\le 2n$, we define
\[
T_p \;=\; \bigcup_{s=1}^{r} C_{i_s j_s}^{\prime\,(s)},
\qquad
\text{where each } C_{i_s j_s}^{\prime\,(s)}\in \pi'^{(s)} 
\text{ is chosen so that all arcs in } T_p 
\text{ have terminal vertex } v_p.
\]
The resulting coarsening
\[
\tau \;=\; \{T_1,\, T_2,\, \dots,\, T_{2n}\}
\]
has exactly $t = 2n = |V(G)|$ cells, one for each vertex of $G$.

\begin{figure}[H]
\centering
\begin{tikzpicture}[>=stealth,
    every node/.style={circle, draw, minimum size=0.85cm, font=\small},
    t1/.style={->, thick, green!60!black},
    t2/.style={->, thick, orange!90!black},
    t3/.style={->, thick, blue!75!black},
    t4/.style={->, thick, violet!80!black}]
\node (1) at (-2,  1) {$1$};
\node (2) at (-2, -1) {$2$};
\node (3) at ( 2,  1) {$3$};
\node (4) at ( 2, -1) {$4$};
\draw[t1] (3) to[bend right=20] (1);
\draw[t1] (4) to[bend left=20]  (1);
\draw[t2] (4) to[bend right=20] (2);
\draw[t2] (3) to[bend left=20]  (2);
\draw[t3] (2) to[bend left=20]  (3);
\draw[t3] (1) to[bend right=20] (3);
\draw[t4] (1) to[bend left=20]  (4);
\draw[t4] (2) to[bend right=20] (4);
\end{tikzpicture}
\qquad
\begin{tikzpicture}[>=stealth,
    every node/.style={circle, draw, minimum size=1.0cm, font=\small},
    qarr/.style={->, thick, gray!55!black}]
\node[fill=green!20]  (T1) at (-2,  1) {$T_1$};
\node[fill=orange!25] (T2) at (-2, -1) {$T_2$};
\node[fill=blue!15]   (T3) at ( 2,  1) {$T_3$};
\node[fill=violet!20] (T4) at ( 2, -1) {$T_4$};
\draw[qarr] (T1) -- (T3);
\draw[qarr] (T1) -- (T4);
\draw[qarr] (T2) -- (T3);
\draw[qarr] (T2) -- (T4);
\draw[qarr] (T3) -- (T1);
\draw[qarr] (T3) -- (T2);
\draw[qarr] (T4) -- (T1);
\draw[qarr] (T4) -- (T2);
\end{tikzpicture}

\begin{tabular}{ll}
\textcolor{green!60!black}{$\bullet$}  $T_1 = \{(3,1),(4,1)\}$ &
\textcolor{orange!90!black}{$\bullet$} $T_2 = \{(4,2),(3,2)\}$ \\[2pt]
\textcolor{blue!75!black}{$\bullet$}   $T_3 = \{(2,3),(1,3)\}$ &
\textcolor{violet!80!black}{$\bullet$} $T_4 = \{(1,4),(2,4)\}$
\end{tabular}
\caption{Left: $K_{2,2}$ with arcs coloured by cell $T_p$ (terminal vertex 
$p$). Right: quotient graph $LD(G)/\tau$; arrow $T_p\!\to\!T_q$ exists when 
the head-vertex of arcs in $T_p$ has an out-arc belonging to $T_q$.}
\label{fig:K22-partition}
\end{figure}
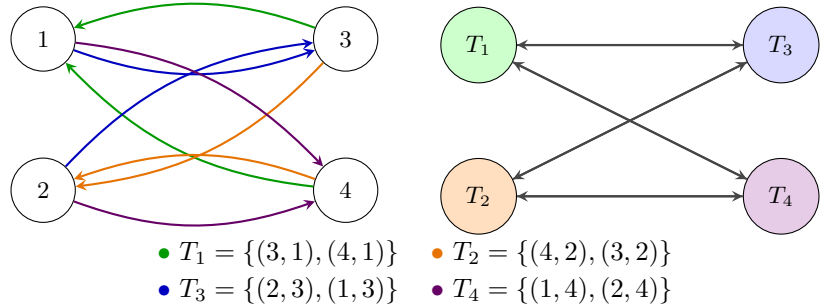

\begin{ex}\label{ex 1}
Let $G = K_{2,2}$ with vertex set $\{1,2,3,4\}$, bipartition 
$C_1=\{1,2\}$ and $C_2=\{3,4\}$, and arc classes consisting of all 
directed arcs between the two parts:
\[
\mathcal{A}^{(1)}=\{(3,1),(4,2),(2,3),(1,4)\},\qquad
\mathcal{A}^{(2)}=\{(1,3),(2,4),(4,1),(3,2)\}.
\]
The atomic cells of $\pi'^{(1)}$ and $\pi'^{(2)}$ are:
\[
\begin{array}{ll}
C_{21}^{\prime(1)} = \{(3,1),(4,2)\}, & 
C_{21}^{\prime(2)} = \{(4,1),(3,2)\}, \\[4pt]
C_{12}^{\prime(1)} = \{(2,3),(1,4)\}, & 
C_{12}^{\prime(2)} = \{(1,3),(2,4)\},
\end{array}
\]
where the index $ij$ records tail-cell $C_i$ and head-cell $C_j$ with 
$i,j\in\{1,2\}$. The coarsening by terminal vertex gives 
$\tau=\{T_1,T_2,T_3,T_4\}$ with $t = 2n = 4$:
\[
T_1=\{(3,1),(4,1)\},\quad T_2=\{(4,2),(3,2)\},\quad 
T_3=\{(2,3),(1,3)\},\quad T_4=\{(1,4),(2,4)\}.
\]
Explicitly, each $T_p$ is formed by taking from each atomic cell the 
unique arc whose head is vertex $p$:
\[
T_1 = \{(3,1)\}\cup\{(4,1)\}, \quad
T_2 = \{(4,2)\}\cup\{(3,2)\}, \quad
T_3 = \{(2,3)\}\cup\{(1,3)\}, \quad
T_4 = \{(1,4)\}\cup\{(2,4)\}.
\]
The arcs within each class are ordered so that the $k$-th arc of 
$\mathcal{A}^{(1)}$ and the $k$-th arc of $\mathcal{A}^{(2)}$ share 
the same tail vertex for every $k$:
\[
\begin{array}{c|cc}
k & \mathcal{A}^{(1)} & \mathcal{A}^{(2)} \\
\hline
1 & (3,1) & (4,1) \\
2 & (4,2) & (3,2) \\
3 & (2,3) & (1,3) \\
4 & (1,4) & (2,4)
\end{array}
\]
This alignment ensures the Kronecker factorization $Q_\tau = I_2\otimes  Q$.
The adjacency matrix of $G$ is
\[
A = P_{\mathcal{A}^{(1)}} + P_{\mathcal{A}^{(2)}} =
\begin{pmatrix}0&0&1&1\\0&0&1&1\\1&1&0&0\\1&1&0&0\end{pmatrix}.
\]
Each cell $T_p$ has size $2$, so the normalized characteristic matrix, 
with rows ordered as $\mathcal{A}^{(1)}$ then $\mathcal{A}^{(2)}$, that is
$(3,1),(4,2),(2,3),(1,4),(4,1),(3,2),(1,3),(2,4)$, is
\[
Q_\tau = \frac{1}{\sqrt{2}}\begin{pmatrix}
1&0&0&0\\
0&1&0&0\\
0&0&1&0\\
0&0&0&1\\
1&0&0&0\\
0&1&0&0\\
0&0&1&0\\
0&0&0&1\
\end{pmatrix} = I_2\otimes Q,
\]
where columns correspond to $T_1, T_2, T_3, T_4$ respectively.
The adjacency matrix of $LD(G)$, with arcs ordered as
$(3,1),(4,2),(2,3),(1,4),(4,1),(3,2),(1,3),(2,4)$, is
\[
A_{LD}=\begin{pmatrix}
0&0&0&1&0&0&1&0\\
0&0&1&0&0&0&0&1\\
1&0&0&0&0&1&0&0\\
0&1&0&0&1&0&0&0\\
0&0&0&1&0&0&1&0\\
0&0&1&0&0&0&0&1\\
1&0&0&0&0&1&0&0\\
0&1&0&0&1&0&0&0
\end{pmatrix}.
\]
Computing $Q_\tau^\top A_{LD(G)}\,Q_\tau$ gives
\[
B = Q_\tau^\top A_{LD(G)}\,Q_\tau
=\begin{pmatrix}0&0&1&1\\0&0&1&1\\1&1&0&0\\1&1&0&0\end{pmatrix} = A (\text{see Figure}~\ref{fig:K22-partition}),
\]
confirming
\[
Q_\tau^\top A_{LD(G)}\,Q_\tau = A \qquad \text{and hence} \qquad
A_{LD(G)}\,Q_\tau = Q_\tau\,A.
\] 
\end{ex}

\begin{lemma}\label{line digraph 1}
Let $\pi = \{C_1, \dots, C_n\}$ be an equitable partition of $V(G)$, 
where $|V(G)| = 2n$ and $G$ is a $d$-regular directed graph. 
For each vertex $v_p \in V(G)$ with $1 \le p \le 2n$, define
\[
T_p \;=\; \bigcup_{s=1}^{r} C_{i_s j_s}^{\prime\,(s)},
\qquad 1 \le i_s,\, j_s \le n,
\]
where for each $s$, the cell $C_{i_s j_s}^{\prime\,(s)} \in \pi'^{(s)}$ 
is chosen so that every arc in $T_p$ has terminal vertex $v_p$. 
Then the collection
\[
\tau = \{T_1,\, T_2,\, \dots,\, T_{2n}\}
\]
is an equitable arc partition of the line digraph $LD(G)$. Moreover, if 
$Q_\tau$ denotes the normalized characteristic matrix of $\tau$, then
\[
A_{LD(G)}\, Q_\tau = Q_\tau\, A,
\]
where $A_{LD(G)}$ is the adjacency matrix of $LD(G)$ and $A$ is the 
adjacency matrix of $G$.
\end{lemma}

\begin{proof}
Let $(v_i, v_p)$ be any arc of $G$ lying in cell $T_p$, and let $T_q$ 
be any cell of $\tau$, where $1\le p,q\le 2n$. We count the neighbors 
of $(v_i, v_p)$ in $T_q$ inside $LD(G)$. By definition of the line 
digraph, an arc $(v_\ell, v_m)$ is adjacent to $(v_i, v_p)$ in $LD(G)$ 
if and only if $v_\ell = v_p$. Therefore the neighbors of $(v_i,v_p)$ 
lying in $T_q$ are exactly
\(
S \;=\; \{(v_\ell, v_m) \in T_q : v_\ell = v_p\}.
\)
By construction, every arc in $T_q$ has terminal vertex $v_q$, so every 
arc in $T_q$ has the form $(v_\ell, v_q)$ for some $v_\ell$. Imposing 
$v_\ell = v_p$ reduces $S$ to
\(
S \;=\; \{(v_p, v_q) \in T_q\} \;=\; \{(v_p, v_q) \in E(G)\}.
\)
Hence
\begin{equation}\label{eq_20}
|S| \;=\; A_{pq},
\end{equation}
where $A_{pq}$ is the $(p,q)$-entry of the adjacency matrix $A$ of $G$. 
Since $|S| = A_{pq}$ depends only on $p$ and $q$, and not on the 
particular arc $(v_i, v_p)$ chosen within $T_p$, every arc in $T_p$ 
has exactly $A_{pq}$ neighbors in $T_q$. This holds for all 
$1\le p,q\le 2n$, so $\tau$ is an equitable partition of $LD(G)$.

Now let $Q_\tau$ be the normalized characteristic matrix of $\tau$. 
For any equitable partition the standard identity gives
\begin{equation}\label{eq_21}
A_{LD(G)}\,Q_\tau \;=\; Q_\tau\,B,
\end{equation}
where $B$ is the quotient matrix with $(p,q)$-entry equal to the number 
of neighbors in $T_q$ of any arc in $T_p$. Thus from~\eqref{eq_20}, we have
\[
B_{pq} \;=\; |S| \;=\; A_{pq} \qquad \text{for all } 1\le p,q\le 2n,
\]
so $B = A$. Substituting in~\eqref{eq_21} gives
\begin{equation}
A_{LD(G)}\,Q_\tau \;=\; Q_\tau\,A.
\end{equation}
\end{proof}

Now we extend our analysis by performing an arc partition on the quotient graph $G/\pi$. This allows us to further reduce the dimensionality of the quantum walk by exploiting symmetries remaining in the quotient structure.

\begin{lemma}\label{line digraph 2}
Let $\pi = \{C_1, \dots, C_n\}$  be an equitable partition of $V(G)$, where 
$|V(G)| = 2n$ and $G$ is $d$-regular directed graph,
and let $ G/\pi$ be the quotient digraph with adjacency
matrix $\widetilde{A}$. Define
\[
\sigma_j \;=\;
\{\,(C_i, C_j)\in E({G/\pi}) : 1 \le i \le n\,\},
\qquad j = 1,\dots,n,
\]
that is, $\sigma_j$ consists of all arcs of ${G/\pi}$ with terminal
vertex $C_j$. Then
\[
\sigma = \{\sigma_1,\dots,\sigma_n\}
\]
is an equitable arc partition of the line digraph
$LD({G/\pi})$. Moreover, if $\widetilde{Q}$ is the
normalized characteristic matrix of $\sigma$, then
\[
A_{LD(G/\pi)}\,\widetilde{Q} \;=\; \widetilde{Q}\,\widetilde{A}.
\]
\end{lemma}

\begin{proof}

The proof follows the same argument as the proof of Lemma~\ref{line digraph 1}.

\end{proof}

\begin{ex}\label{C_4 grah and its quotient}
Let $G = C_4$ and consider the vertex partition
\(
\pi = \{\{1,3\},\{2,4\}\} = \{C_1, C_2\}, \quad d = 2.
\)
The normalized characteristic matrix of this partition is
\[
Q = \frac{1}{\sqrt{2}}
\begin{pmatrix}
1 & 0\\ 0 & 1\\ 1 & 0\\ 0 & 1
\end{pmatrix},
\]
so the quotient adjacency matrix is
\[
A(G/\pi) = Q^\top A\, Q =
\begin{pmatrix}0 & 2\\ 2 & 0\end{pmatrix}.
\]
The quotient graph $G/\pi$ has two vertices $C_1, C_2$ with two arcs
in each direction. Label the four arcs sequentially as
\[
b_1 = (C_1, C_2)^{(1)}, \quad
b_2 = (C_2, C_1)^{(1)}, \quad
b_3 = (C_1, C_2)^{(2)}, \quad
b_4 = (C_2, C_1)^{(2)},
\]
ordered as $b_1, b_2, b_3, b_4$ ( see Figure~\ref{fig:quotient-arc-partition}). The arcs alternate terminal vertex
in this ordering:
\[
\begin{array}{c|cc}
\text{arc} & \text{direction} & \text{terminal vertex}\\
\hline
b_1 & C_1 \to C_2 & C_2\\
b_2 & C_2 \to C_1 & C_1\\
b_3 & C_1 \to C_2 & C_2\\
b_4 & C_2 \to C_1 & C_1
\end{array}
\]
Coarsen these four arcs into two cells by terminal vertex:
\[
\sigma_1 = \{b_1, b_3\}
\quad\text{(terminal vertex }C_2\text{)},\qquad
\sigma_2 = \{b_2, b_4\}
\quad\text{(terminal vertex }C_1\text{)}.
\]
Each cell has size $2$, so the normalized characteristic matrix of
$\sigma = \{\sigma_1, \sigma_2\}$, with rows ordered as
$b_1, b_2, b_3, b_4$ and columns as $\sigma_1, \sigma_2$, is
\[
\widetilde{Q} = \frac{1}{\sqrt{2}}
\begin{pmatrix}
1 & 0\\
0 & 1\\
1 & 0\\
0 & 1
\end{pmatrix}.
\]
In $LD(G/\pi)$, an arc with head $C_j$ points to all arcs with tail
$C_j$, so every vertex has out-degree $2$. Under the ordering $b_1, b_2, b_3, b_4$, the adjacency matrix of
$LD(G/\pi)$ is therefore
\[
A_{LD(G/\pi)} =
\begin{pmatrix}
0 & 1 & 0 & 1\\
1 & 0 & 1 & 0\\
0 & 1 & 0 & 1\\
1 & 0 & 1 & 0
\end{pmatrix}.
\]
We verify the intertwining relation directly:
\[
A_{LD(G/\pi)}\,\widetilde{Q}
=
\frac{1}{\sqrt{2}}
\begin{pmatrix}
0&1&0&1\\ 1&0&1&0\\ 0&1&0&1\\ 1&0&1&0
\end{pmatrix}
\begin{pmatrix}
1&0\\ 0&1\\ 1&0\\ 0&1
\end{pmatrix}
=
\frac{1}{\sqrt{2}}
\begin{pmatrix}
0&2\\ 2&0\\ 0&2\\ 2&0
\end{pmatrix}.
\]
On the other hand,
\[
\widetilde{Q}\,A(G/\pi)
=
\frac{1}{\sqrt{2}}
\begin{pmatrix}
1&0\\ 0&1\\ 1&0\\ 0&1
\end{pmatrix}
\begin{pmatrix}0&2\\ 2&0\end{pmatrix}
=
\frac{1}{\sqrt{2}}
\begin{pmatrix}
0&2\\ 2&0\\ 0&2\\ 2&0
\end{pmatrix}.
\]
Since both sides agree, we confirm
\[
A_{LD(G/\pi)}\,\widetilde{Q} = \widetilde{Q}\,A(G/\pi).
\]
We also recover the quotient adjacency matrix:
\[
\widetilde{Q}^\top A_{LD(G/\pi)}\,\widetilde{Q}
=
\frac{1}{2}
\begin{pmatrix}1&0&1&0\\ 0&1&0&1\end{pmatrix}
\begin{pmatrix}
0&2\\ 2&0\\ 0&2\\ 2&0
\end{pmatrix}
=
\begin{pmatrix}0&2\\ 2&0\end{pmatrix}
= A(G/\pi),
\]
confirming that the nested quotient recovers the two-vertex multigraph.
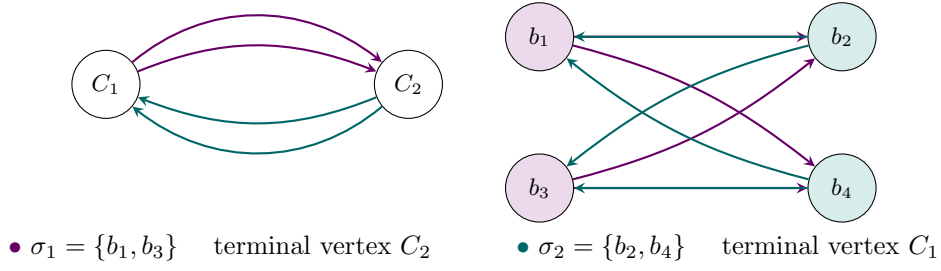
\begin{figure}[ht]
\centering
\begin{tikzpicture}[>=stealth,
    every node/.style={circle, draw, minimum size=0.9cm, font=\small},
    s1/.style={->, thick, violet!80!black},
    s2/.style={->, thick, teal!80!black}]
  \node (C1) at (-2, 0) {$C_1$};
  \node (C2) at ( 2, 0) {$C_2$};
  \draw[s1] (C1) to[bend left=40]
      node[draw=none, above, font=\tiny, violet!80!black]{} (C2);
  \draw[s1] (C1) to[bend left=22]
      node[draw=none, above, font=\tiny, violet!80!black]{} (C2);
  \draw[s2] (C2) to[bend left=40]
      node[draw=none, below, font=\tiny, teal!80!black]{} (C1);
  \draw[s2] (C2) to[bend left=22]
      node[draw=none, below, font=\tiny, teal!80!black]{} (C1);
\end{tikzpicture}
\qquad
\begin{tikzpicture}[>=stealth,
    every node/.style={circle, draw, minimum size=0.9cm, font=\small},
    s1arr/.style={->, thick, violet!80!black},
    s2arr/.style={->, thick, teal!80!black}]
  \node[fill=violet!15] (b1) at (-2,  1) {$b_1$};
  \node[fill=teal!15]   (b2) at ( 2,  1) {$b_2$};
  \node[fill=violet!15] (b3) at (-2, -1) {$b_3$};
  \node[fill=teal!15]   (b4) at ( 2, -1) {$b_4$};
  \draw[s1arr] (b1) -- (b2);
  \draw[s1arr] (b1) to[bend left=12]  (b4);
  \draw[s1arr] (b3) to[bend right=12] (b2);
  \draw[s1arr] (b3) -- (b4);
  \draw[s2arr] (b2) -- (b1);
  \draw[s2arr] (b2) to[bend right=12] (b3);
  \draw[s2arr] (b4) to[bend left=12]  (b1);
  \draw[s2arr] (b4) -- (b3);
\end{tikzpicture}

\begin{tabular}{ll}
  \textcolor{violet!80!black}{$\bullet$}\ $\sigma_1 = \{b_1, b_3\}$
  \quad terminal vertex $C_2$ &\qquad
  \textcolor{teal!80!black}{$\bullet$}\ $\sigma_2 = \{b_2, b_4\}$
  \quad terminal vertex $C_1$
\end{tabular}

\caption{Left: quotient graph $G/\pi$ with arcs coloured by cell;
$\sigma_1$ (violet) groups arcs $b_1, b_3$ with terminal vertex
$C_2$, and $\sigma_2$ (teal) groups arcs $b_2, b_4$ with terminal
vertex $C_1$.
Right: line digraph $LD(G/\pi)$; every arc in $\sigma_j$ has the
same out-neighbours.}
\label{fig:quotient-arc-partition}
\end{figure}
\end{ex}
\begin{Theorem}\label{Isomorphism}
Let $G = (V, E)$ be a $d$-regular directed graph with $|V(G)| = 2n$, and let
\[
\pi = \{C_1, \dots, C_n\}
\] 
be an equitable partition of $V(G)$ with quotient digraph $G/\pi$. Then there 
exist equitable arc partitions $\tau$ of $LD(G)$ and $\sigma$ of $LD(G/\pi)$, 
obtained by grouping arcs according to their terminal vertices, such that
\[
LD(G)/\tau \;\cong\; G
\qquad \text{and} \qquad
LD(G/\pi)/\sigma \;\cong\; G/\pi.
\]
\end{Theorem}               
\begin{proof}

The result follows directly from the preceding Lemma~\ref{line digraph 1} and Lemma~\ref{line digraph 2}. The first isomorphism \begin{equation}LD(G)/\tau \cong G\end{equation} is established by the construction of the arc partition \(\tau\), where arcs are grouped according to their terminal vertices. Similarly, the second isomorphism \begin{equation}LD(G/\pi)/\sigma \cong G/\pi\end{equation} follows from the corresponding induced arc partition \(\sigma\) on the quotient graph.
This construction shows that forming the line digraph and then taking the corresponding arc partition yields a graph isomorphic to the original graph (or its quotient), completing the proof.
\end{proof}
\section{Case~I: Unitary Evolution on \texorpdfstring{$G$}{G} and 
\texorpdfstring{$G/\pi$}{G/pi} under the Assumption 
\texorpdfstring{$Q = \widetilde{Q}$}{Q = Q-tilde} and 
\texorpdfstring{$Q_\tau = I_d \otimes \widetilde{Q}$}{Q-tau = Id otimes Q-tilde}}
\label{section 4}
\label{sec:unitary-evolution-case1}
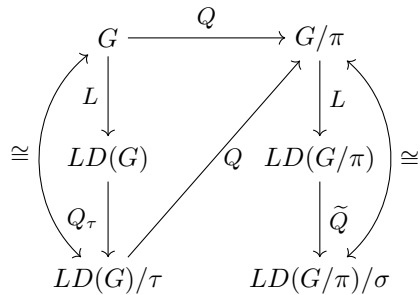
\begin{figure}[h]
\centering
\begin{tikzpicture}[scale=0.8]
\node (G)     at (0,4)  {$G$};
\node (Gpi)   at (3.5,4)  {$G/\pi$};
\node (LDG)   at (0,2)  {$LD(G)$};
\node (LDGpi) at (3.5,2) {$LD(G/\pi)$};
\node (LDGt)  at (0,0)  {$LD(G)/\tau$};
\node (LDGps) at (3.5,0) {$LD(G/\pi)/\sigma$};

\draw[->] (G)     -- node[above]{\small$Q$}             (Gpi);
\draw[->] (G)     -- node[left] {\small$L$}              (LDG);
\draw[->] (Gpi)   -- node[right]{\small$L$}              (LDGpi);
\draw[->] (LDG)   -- node[left] {\small$Q_\tau$}         (LDGt);
\draw[->] (LDGpi) -- node[right]{\small$\widetilde{Q}$}  (LDGps);
\draw[->] (LDGt)  -- node[right]{\small$Q$}              (Gpi);
\draw[<->] (Gpi) to[bend left=50]  node[right]{\small$\cong$} (LDGps);
\draw[<->] (G)   to[bend right=50] node[left] {\small$\cong$} (LDGt);
\end{tikzpicture}
\caption{Commutative diagram illustrating the structural relationships between 
$G$, $LD(G)$, and their equitable quotients, as established in 
Theorem~\ref{Isomorphism}.}
\label{fig:commutative-diagram}
\end{figure}
 In this section, we formalize the transition operator arising from the 
arc-partitioned structure induced by a shunt decomposition. Using the 
partitions introduced above, we describe the unitary evolution on the 
line digraph $LD(G)$ and its reduced form on the quotient space. 
By Theorem~\ref{Isomorphism}, the arc partitions $\tau$ and $\sigma$ 
yield the identifications
\[
LD(G)/\tau \;\cong\; G,
\qquad
LD(G/\pi)/\sigma \;\cong\; G/\pi,
\]
showing that the arc-based dynamics on $LD(G)$ reduce naturally to 
vertex dynamics on $G$, and similarly for the quotient graph $G/\pi$. 
The interplay between the line digraph construction, the equitable 
partitions, and the resulting isomorphisms is captured in the commutative 
diagram of Figure~\ref{fig:commutative-diagram}.
Thus, the arc-based dynamics on $LD(G)$ reduce to vertex dynamics on
$G$, and similarly for the quotient graph.

Let $\widetilde{A} \in \mathbb{C}^{n \times n}$, where $n = |\pi|$ 
denotes the number of cells in the partition $\pi = \{C_1, \ldots, C_r\}$, 
each cell of equal size. Its entries are given by
\begin{equation}
\widetilde{A}_{ij} =
\begin{cases}
1, & \text{if there exists an arc from } C_i \text{ to } C_j \text{ in } G/\pi,\\
0, & \text{otherwise}
\end{cases}
\end{equation}
This reduced matrix governs the effective evolution on the quotient
space and provides a lower-dimensional representation of the dynamics.. Since $G$ is $d$-regular directed graph with $|\pi|$ cells of equal 
size, each vertex $C_i$ of $G/\pi$ has out-degree $d$, and hence
\begin{equation}
|\mathcal{A}(G/\pi)| \;=\; d \cdot |\pi|,
\end{equation}
where $\mathcal{A}(G/\pi)$ denotes the arc set of $G/\pi$. This 
reflects Corollary~\ref{cor:3.1} that the quotient graph inherits the 
$d$-regularity of $G$, with each cell $C_i$ sending exactly $d$ arcs 
to other cells, one per shunt class. 

The \emph{unitary evolution operator or the transition operator} for the quotient graph $G/\pi$
is defined as
\begin{equation}
\widetilde{U} \;=\; \mathcal{U}(G/\pi)
\;=\; \widetilde{S}\,(2\widetilde{Q}\widetilde{Q}^\top - I)
\;\in\; \mathbb{C}^{|\mathcal{A}(G/\pi)| \times |\mathcal{A}(G/\pi)|},
\end{equation}
where $\widetilde{S}$ is the shift operator on the arcs of $G/\pi$
and $\widetilde{Q}$ is the normalized characteristic matrix of the arc
partition $\sigma$ of $LD(G/\pi)$. Since
$|\mathcal{A}(G/\pi)| = d\,|\pi|$, this is a square matrix of size
$d|\pi| \times d|\pi|$ acting on the quotient arc space
$\mathbb{C}^{\mathcal{A}(G/\pi)}$, with one coordinate per arc of
$G/\pi$. We refer to the discrete quantum walks governed by
$\widetilde{U}$ as \emph{shunt decomposition walks}.
To verify the unitarity of $\widetilde {U}$, we examine its components. Since each shunt $P_j$ is a permutation matrix, it is unitary and satisfies the intertwining relation $P_j Q = Q \widetilde{P}_j$, where $Q$ is the normalized vertex characteristic matrix. Given ${Q}^\top {Q} = I_n$, it follows that for every $j$:
\begin{equation}
(Q^\top P_j Q)(Q^\top P_j Q)^\top = (Q^\top Q \widetilde{P}_j)(Q^\top Q \widetilde{P}_j)^\top = \widetilde{P}_j \widetilde{P}_j^\top = I_n.
\end{equation}
Thus, each block $Q^\top P_j Q$ is unitary. Using the properties of the Kronecker product and the identity $E_{jj}E_{ii} = \delta_{ji}E_{jj}$, the reduced shift operator $\widetilde {S}$ satisfies:
\begin{equation}\label{shift matrix unitary opertaor}
\widetilde {S}\widetilde {S}^\top = \sum_{j,i}E_{jj}E_{ii} \otimes (Q^\top P_j Q)(Q^\top P_i Q)^\top  = \sum_{j}E_{jj} \otimes I_n = I_{|\pi|} \otimes I_n .
\end{equation}
Hence, $\widetilde {S}$ is unitary. Similarly, the reduced reflection operator \begin{equation}\label{quotient reflection}
\widetilde {R} = 2\widetilde{Q}\widetilde{Q}^\top - I\end{equation} is Hermitian. Since $\widetilde{Q}\widetilde{Q}^\top$ is an orthogonal projector, we have:
\begin{equation}
\widetilde {R}^2 = (2\widetilde{Q}\widetilde{Q}^\top - I)^2 = 4(\widetilde{Q}\widetilde{Q}^\top)^2 - 4\widetilde{Q}\widetilde{Q}^\top + I = I,
\end{equation}
implying that $\widetilde {R}$ is an involution and thus unitary. Consequently, $\widetilde{U} = \widetilde{S}\widetilde{R}$ is the 
product of two unitary matrices, and hence $\widetilde{U}$ is unitary.
The reflection operator on the full arc space $\mathcal{A}(G)$ is 
defined as
\begin{equation}\label{reflection operator 2}
R = 2\,Q_\tau(\widetilde{Q}\widetilde{Q}^\top)Q_\tau^\top - I,
\end{equation}
where $Q_\tau$ is the normalized characteristic matrix of the arc 
partition $\tau$ induced by the shunt decomposition, and 
$\widetilde{Q}$ is the normalized characteristic matrix of the arc 
partition $\sigma$ of $LD(G/\pi)$. The matrix 
$P = Q_\tau(\widetilde{Q}\widetilde{Q}^\top)Q_\tau^\top$ is an 
orthogonal projection onto the subspace spanned by the columns of 
$Q_\tau\widetilde{Q}$, satisfying $P^2 = P$ and $P^\top = P$. The 
operator $R = 2P - I$ is then the reflection through this projection 
subspace, which maps any vector $v$ to $2Pv - v$, reversing the 
component of $v$ orthogonal to the subspace while preserving the 
component within it. In particular, $R$ satisfies $R^2 = I$ and 
$R^\top = R$, confirming that $R$ is both unitary and self-adjoint. The operator $R$ acts as a 
reflection about the subspace spanned by the columns of $Q_\tau$, 
and is unitary by construction.

Using this, we define the transition matrix $U$ for the parent graph 
$G$ on the full arc space $\mathcal{A}(G)$ as
\begin{equation}
U = SR = S\bigl(2\,Q_\tau(\widetilde{Q}\widetilde{Q}^\top)Q_\tau^\top
- I\bigr)
\;\in\; \mathbb{C}^{|\mathcal{A}(G)|\times|\mathcal{A}(G)|},
\end{equation}
where $S = \displaystyle\sum_{j=1}^{d} P_j \otimes E_{jj}$ is the 
shift operator on $\mathcal{A}(G)$, with $P_j$ the permutation matrix 
of the $j$-th arc class and $E_{jj}$ the $j$-th standard basis matrix, 
and $|\mathcal{A}(G)|$ denotes the total number of arcs in $G$. Thus 
$U$ is a square unitary matrix acting on the arc space 
$\mathbb{C}^{|\mathcal{A}(G)|}$, with one coordinate per arc of $G$. 
This framework allows the  analysis of high-dimensional walks 
to be conducted efficiently through the reduced operator $\widetilde{U}$. 
In what follows, we illustrate this construction with examples and 
establish the precise relationship between $\widetilde{U}$ and $U$.

\subsection{Relation between \texorpdfstring{$U$}{U} and 
\texorpdfstring{$\widetilde{U}$}{\~U}}\label{subsection 4.1}
Ordering the arcs of $\mathcal{A}(G)$ such that all elements of the
arc-subsets $\mathcal{A}^{(1)}, \mathcal{A}^{(2)}, \dots,
\mathcal{A}^{(d)}$ are grouped sequentially, the normalized
characteristic matrix $Q_{\tau}$ of the induced arc-partition $\tau$
factorizes as:
\begin{equation}
Q_{\tau} = I_d \otimes \widetilde{Q} ,
\end{equation}
where $\widetilde{Q}$ is the normalized characteristic matrix of the
quotient arc-partition of $LD(G/\pi)$, and $I_d$ is the identity
matrix on the arc-class index. Similarly, by ordering the arcs of the nested quotient
$LD(G/\pi)/\sigma$, such that the normalized characteristic
matrix of the induced arc-partition $\sigma$ satisfies
$\widetilde{Q} = Q$, where $Q$ is the characteristic matrix of the
original vertex partition $\pi$. From Equation~\ref{eq:equitable} and Lemma~\ref{line digraph 2}, 
using $\widetilde{Q} = Q$, we have:
\begin{equation}\label{eq:spectral-equiv}
\widetilde{Q}^\top \!\left(A - A_{LD(G/\pi)}\right)\widetilde{Q} = O.
\end{equation}
where $O$ denotes the zero matrix. This identity states that the
matrix $M = A - A_{LD(G/\pi)}$ satisfies
$\widetilde{Q}^\top M\,\widetilde{Q} = O$, meaning that for every
pair of cells $C_i, C_j \in \pi$, the total number of arcs from
$C_i$ to $C_j$ counted in $A$ equals the total number counted in
$A_{LD(G/\pi)}$:
\[
\sum_{u \in C_i}\sum_{v \in C_j} A_{uv}
\;=\;
\sum_{u \in C_i}\sum_{v \in C_j} (A_{LD(G/\pi)})_{uv}.
\]
In other words, even if $A$ and $A_{LD(G/\pi)}$ differ entry by
entry, their cell-to-cell arc counts are identical under $\pi$.
From Equation~\eqref{eq:spectral-equiv} we identify two structural
cases:

\begin{enumerate}
\item  $A = A_{LD(G/\pi)}$, that is, the
adjacency matrix of the parent graph $G$ and the adjacency matrix of
the line digraph quotient $LD(G/\pi)$ coincide entry by entry. The
example of $C_4$ with partition $\pi = \{\{1,3\},\{2,4\}\}$
illustrates this case in Example~\ref{C_4 grah and its quotient}: we computed
\[
A_{LD(C_4/\pi)} =
\begin{pmatrix}0&1&0&1\\1&0&1&0\\0&1&0&1\\1&0&1&0\end{pmatrix}
= A(C_4),
\]
so $M = O$ identically, and Equation~\eqref{eq:spectral-equiv} holds
trivially.

\item 
$A \neq A_{LD(G/\pi)}$ but
$\mathrm{Range}(\widetilde{Q}) \subseteq \mathrm{Null}(M)$, where
$M = A - A_{LD(G/\pi)}$. In this case the two matrices differ in
individual entries, but for every pair of cells $C_i, C_j$ the number
of arcs from $C_i$ to $C_j$ is the same in both $A$ and
$A_{LD(G/\pi)}$. The difference $M$ annihilates every column of
$\widetilde{Q}$, so $M\widetilde{Q} = O$ and hence
$\widetilde{Q}^\top M \widetilde{Q} = O$.
\end{enumerate}
\begin{prop}\label{rem:cases}
Let $G$ be a $d$-regular digraph on $2n$ vertices. Let $\pi$ an equitable
partition of $V(G)$ into $n$ cells, and $\sigma$ an equitable partition
of the arcs of $LD(G/\pi)$ into $n$ cells. If $LD(G/\pi) \cong G$, then
$Q = \widetilde{Q}$; however, the converse does not hold in general.
\end{prop}
\begin{proof}
Since $\pi$ is an equitable vertex partition of $G$, there exists a normalized
characteristic matrix $Q$ such that
\begin{equation}\label{eq:vertex}
AQ = Q\widetilde{A},
\end{equation}
where $A$ is the adjacency matrix of $G$ and $\widetilde{A}$ is the quotient
matrix of $G/\pi$. Since $\sigma$ is an equitable arc partition of $L(G/\pi)$, there exists a
normalized characteristic matrix $\widetilde{Q}$ such that
\begin{equation}\label{eq:arc}
A_{LD(G/\pi)}\,\widetilde{Q} = \widetilde{Q}\,\widetilde{A},
\end{equation}
where $\widetilde{A}$ is the quotient matrix of \(G/\pi\) by Theorem~\ref{Isomorphism}. Now suppose $LD(G/\pi)\cong G$, so that $A_{LD(G/\pi)}=A$. Substituting into
Equation~\eqref{eq:arc} gives
\(
A\widetilde{Q} = \widetilde{Q}\,\widetilde{A}.
\)
Comparing with Equation~\eqref{eq:vertex}, both $Q$ and $\widetilde{Q}$
satisfy the same Equation, and since the normalized
characteristic matrix of an equitable partition is unique, we conclude that
\[
Q = \widetilde{Q}.
\]
For the converse, $Q = \widetilde{Q}$ does not imply $LD(G/\pi) \cong G$ in general. As a counterexample, see Example~\ref{ex:C6_quotient}.
\end{proof}
\begin{prop}\label{prop}
Let $G$ be a $d$-regular directed graph on $2n$ vertices, and let $\pi$ be an
equitable partition of $V(G)$ into cells of equal size. Let $\tau$ and $\sigma$ be the arc
partitions of $LD(G)$ and $LD(G/\pi)$ into cells 
of equal size, respectively, with normalized
characteristic matrices $Q_\tau$ and $\widetilde{Q}$ satisfying
\[
Q_\tau=I_d\otimes \widetilde Q, \qquad \widetilde{Q}={Q} .
\]
Let $U = SR$ and $\widetilde{U} = \widetilde{S}\,\widetilde{R}$ be the
unitary evolution operators on $G$ and
$G/\pi$, respectively. Then the following hold:
\begin{enumerate}
    \item $R\, Q_\tau = Q_\tau\, \widetilde{R}$.

    \item For all integers $k \geq 0$,
    \(
        \widetilde{U}^k = Q_\tau^\top\, U^k\, Q_\tau.
    \)
\end{enumerate}
\end{prop}

\begin{proof}

\textbf{(a)} Since \(Q_{\tau}^{\top}Q_{\tau} = I\), by Equations~\ref{rem:cases} and~\ref{quotient reflection}, we have
\begin{equation}
R Q_{\tau}
= (2Q_{\tau}(\widetilde  Q \widetilde  Q^\top)Q_{\tau}^{\top} - I)Q_{\tau}
= Q_{\tau}(2\widetilde Q \widetilde Q^\top - I)
= Q_{\tau}\widetilde{R}.
\end{equation}
\textbf{(b)} By Lemma~\ref{lemma 3.1} and Corollary~\ref{cor:shunt-intertwine}, and using \(Q_{\tau} = I_d \otimes \widetilde Q,\quad Q=\widetilde{Q}\), we have

\[
\widetilde U
= \widetilde S \widetilde R
= (I_d \otimes \widetilde Q^\top)\, S\, (I_d \otimes \widetilde Q)\, \widetilde R=Q_\tau SQ_\tau \widetilde  R.
\]
Using part~(a), we have
\begin{equation}
    \widetilde{U} = Q_\tau^{\top} U Q_\tau.
\end{equation}
By Lemma~\ref{lem:equitable_equiv}, $Q_\tau Q_\tau^{\top}$ commutes 
with $U$, and hence $\operatorname{Im}(Q_\tau)$ is invariant under $U$. 
This gives \(
    U Q_\tau = Q_\tau \widetilde{U}\).
Applying this relation inductively yields
\begin{equation}\label{eq:Uk-intertwine}
    U^k Q_\tau = Q_\tau \widetilde{U}^k \quad \text{for all } k \geq 0.
\end{equation}
Multiplying both sides of \eqref{eq:Uk-intertwine} on the left by 
$Q_\tau^{\top}$ and using the isometry property $Q_\tau^{\top}Q_\tau = I$, 
we obtain
\begin{equation}
    Q_\tau^{\top}\, U^k\, Q_\tau = \widetilde{U}^k 
    \quad \text{for all } k \geq 0.
\end{equation}

\end{proof}
The assumptions and propositions established above lead to the 
following equivalence between PST in the parent 
graph $G$ and its quotient graph $G/\pi$, as made precise in 
Theorem~\ref{PST theorem 1} below. 
\begin{Theorem}\label{PST theorem 1}
Let $G$ be a $d$-regular directed graph on $2n$ vertices, and let $\pi$ be an
equitable partition of $V(G)$ into cells of equal size. Let $\tau$ and $\sigma$ be the arc
partitions of $LD(G)$ and $LD(G/\pi)$ into cells 
of equal size, respectively, with normalized
characteristic matrices $Q_\tau$ and $\widetilde{Q}$ satisfying
\[
Q_\tau=I_d\otimes \widetilde Q, \qquad \widetilde{Q}={Q} .
\] Then $U$ exhibits PST from $Q_\tau x$ to $Q_\tau y$ at time $k$ if and only if $\widetilde U$ exhibits PST from $x$ to $y$ at time $k$, for all $x,y \in \operatorname{Im}(\widetilde Q)$.
\end{Theorem}

\begin{proof}
Since $\operatorname{Im}(Q_\tau)$ is $U$--invariant and $Q_\tau$ has orthonormal
columns, we have the basic relation from Proposition~\ref{prop} 
 for all $k\ge0$,
\begin{equation}\label{eq.4.20}
\widetilde U^k = Q_\tau^\top U^k Q_\tau,\qquad U^k Q_\tau = Q_\tau\widetilde U^k.
\end{equation} 
\noindent
($\Rightarrow$)
Assume that $\widetilde U$ exhibits PST at time $k$, that is,
\begin{equation}
\widetilde U^k x = y,
\quad \text{for some } x,y \in \operatorname{Im}(\widetilde Q).
\end{equation}
From Equation~\ref{eq.4.20}, we obtain
\begin{equation}
U^k Q_\tau x
= Q_\tau \widetilde U^k x
= Q_\tau y.
\end{equation}
Hence, $U$ exhibits PST from $Q_\tau x$ to $Q_\tau y$ at time $k$.

\noindent
($\Leftarrow$)
Conversely, suppose that $U$ exhibits PST at time $k$, that is,
\[
U^k Q_\tau x = Q_\tau y,
\quad \text{for some } x,y \in \operatorname{Im}(\widetilde Q).
\]
Multiplying on the left by $Q_\tau^{\top}$ and using $Q_\tau^{\top} Q_\tau = I$, together with~\eqref{eq.4.20}, we obtain
\[
\widetilde U^k x
= Q_\tau^{\top} U^k Q_\tau x
= Q_\tau^{\top} Q_\tau y
= y.
\]
Thus $\widetilde U$ exhibits PST from $x$ to $y$ at time $k$.
Therefore, PST in the quotient system is equivalent to PST in the original system restricted to the invariant subspace $\operatorname{Im}(Q_\tau)$.
\end{proof}
\begin{cor}
  Under the assumptions of Theorem~\ref{PST theorem 1}, the quotient graph
  $LD(G/\pi)/\sigma$ exhibits PST if and only if $LD(G)\tau$ exhibits PST.
    \end{cor} 
\begin{proof}
  By Theorem~\ref{Isomorphism}, we have the isomorphisms
  \(
    LD(G)/\tau \;\cong\; G
    \quad \text{and} \quad
    LD(G/\pi)/\sigma \;\cong\; G/\pi,
  \)
  from which the result follows.
\end{proof}
\begin{lemma}\label{lem:2regular}
Let $G$ be a $d$-regular directed graph on $2n$ vertices. Let
\[
\tau = \{T_1, \dots, T_{2n}\},
\qquad
\pi = \{C_1, \dots, C_{n}\},
\qquad
\sigma = \{\sigma_1, \dots, \sigma_{n}\}
\]
be, respectively, an equitable arc partition of $LD(G)$, an equitable
vertex partition of $G$, and an equitable arc partition of $LD(G/\pi)$, with
normalized characteristic matrices $Q_\tau$, $Q$, and $\widetilde{Q}$.
If $A(G) = A(LD(G/\pi))$ and
\(
Q_\tau = \widetilde{Q} \otimes I_d,
\)
then $G$ is $2$-regular. The converse does not hold in general.
    \end{lemma}

\begin{proof}
Since $A(G) = A(LD(G/\pi))$, the graphs $G$ and $L(G/\pi)$ are identical, and
in particular have the same vertex set. Therefore
\[
|V(G)| = |V(L(G/\pi))| = |\mathcal{A}(G/\pi)|,
\]
which gives $|\mathcal{A}(G/\pi)| = nd$. Now consider the dimensions of each side of $Q_\tau = \widetilde{Q}\otimes I_d$.
Since $|\mathcal{A}(G)|=2nd$ and $\tau$ partitions $\mathcal{A}(G)$ into $n$
cells,
\(
Q_\tau \in \mathbb{R}^{2nd \times 2n}.
\)
Since $|\mathcal{A}(G/\pi)|=nd$ and $\sigma$ partitions $\mathcal{A}(G/\pi)$
into $n$ cells,
\(
\widetilde{Q} \in \mathbb{R}^{nd \times n},
\quad Q_\tau=
I_d\otimes \widetilde{Q}  \in \mathbb{R}^{nd^2 \times nd}.
\)
For the two sides to be equal, their column dimensions must agree, thus \(
2n = nd \implies d = 2.
\)
Therefore $G$ is $2$-regular. For the converse, $G$ being $2$-regular and satisfying
$Q_\tau =  I_2 \otimes \widetilde{Q}$ does not imply $A(G)=A(LD(G/\pi))$ in
general; a counterexample is provided in
Example~\ref{ex:counterexample}.
\end{proof}

\begin{ex}\label{ex:counterexample}
Let $G$ be a directed graph with vertex set
\(
V(G)=\{1,2,3,4,5,6\}
\)
and arc set
\(
\mathcal{A}(G)=\{1\to2,1\to5,\;2\to3,2\to6,\;3\to1,3\to4,\;4\to5,4\to2,\;5\to6,5\to3,\;6\to4,6\to1\}.
\)
Then $G$ is $2$-regular, since every vertex has in-degree and out-degree equal to $2$. The adjacency matrix of $G$ is
\[
A(G)=
\begin{pmatrix}
0&1&0&0&1&0\\
0&0&1&0&0&1\\
1&0&0&1&0&0\\
0&1&0&0&1&0\\
0&0&1&0&0&1\\
1&0&0&1&0&0
\end{pmatrix}.
\]
Consider the partition
\(
\pi=\{\{1,4\},\{2,5\},\{3,6\}\}.
\)
Each vertex in a cell has the same number of out-neighbors in every other cell. Hence $\pi$ is equitable. Its normalized characteristic matrix is \[
Q=
\frac{1}{\sqrt{2}}
\begin{pmatrix}
I_3\\
I_3
\end{pmatrix}.
\]
The quotient graph $G/\pi$ has three nodes $C_1,C_2,C_3$ with two parallel arcs between consecutive nodes:
\(
C_1 \to C_2,\quad C_2 \to C_3,\quad C_3 \to C_1.
\)
Label the arcs as
\(
b_1,b_2: C_1 \to C_2,\quad
b_3,b_4: C_2 \to C_3,\quad
b_5,b_6: C_3 \to C_1.
\)
The line digraph $LD(G/\pi)$ has adjacency matrix (in the order $b_1,\dots,b_6$)
\[
A_{LD(G/\pi)}=
\begin{pmatrix}
0&0&1&1&0&0\\
0&0&1&1&0&0\\
0&0&0&0&1&1\\
0&0&0&0&1&1\\
1&1&0&0&0&0\\
1&1&0&0&0&0
\end{pmatrix}.
\] Define the arc partition
\(
\sigma=\{\{b_1,b_2\},\{b_3,b_4\},\{b_5,b_6\}\}.
\)
Its normalized characteristic matrix is
\[
\widetilde{Q}=
\frac{1}{\sqrt{2}}
\begin{pmatrix}
1&0&0\\
1&0&0\\
0&1&0\\
0&1&0\\
0&0&1\\
0&0&1
\end{pmatrix}.
\]

Then
\[
\widetilde{Q}_\sigma^\top A_{LD(G/\pi)} \widetilde{Q}_\sigma
=
\begin{pmatrix}
0&2&0\\
0&0&2\\
2&0&0
\end{pmatrix}
= \widetilde{A}.
\]

Although $A(G)$ and $A_{LD(G/\pi)}$ are both $6\times6$, they are not equal. Let
\[
M = A(G) - A_{LD(G/\pi)} \neq 0.
\]
However,
\[
\widetilde{Q}^{\top}A(G)\widetilde{Q}
=
\widetilde{Q}^{\top}A_{LD(G/\pi)}\widetilde{Q}
=
\begin{pmatrix}0&2&0\\0&0&2\\2&0&0\end{pmatrix}.
\]
Hence,
\[
\widetilde{Q}^{\top} M \widetilde{Q} = 0 \quad \text{while} \quad M \neq 0,
\]
showing that the two matrices agree at the quotient level but differ entry-wise.
\end{ex}
\section{Case~II: Unitary Evolution on \texorpdfstring{$LD(G)$}{LD(G)} 
and \texorpdfstring{$LD(G/\pi)/\sigma$}{LD(G/pi)/sigma} under the 
Assumption \texorpdfstring{$Q \neq \widetilde{Q}$}{Q neq Q-tilde} and 
\texorpdfstring{$Q_\tau \neq I_d \otimes \widetilde{Q}$}{Q-tau neq Id 
otimes Q-tilde}}\label{section 5}

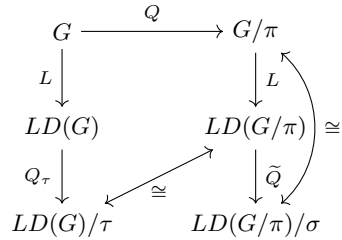
\begin{figure}[h]
  \centering
  \begin{tikzpicture}[scale=0.8, every node/.style={font=\small}]
    \node (G)      at (0,3.2)  {$G$};
    \node (Gpi)    at (3.2,3.2)  {$G/\pi$};
    \node (LDG)    at (0,1.6)  {$LD(G)$};
    \node (LDGpi)  at (3.2,1.6)  {$LD(G/\pi)$};
    \node (LDGt)   at (0,0)  {$LD(G)/\tau$};
    \node (LDGps)  at (3.2,0)  {$LD(G/\pi)/\sigma$};
    \draw[->]  (G)     -- node[above]{\scriptsize$Q$}            (Gpi);
    \draw[->]  (G)     -- node[left] {\scriptsize$L$}             (LDG);
    \draw[->]  (Gpi)   -- node[right]{\scriptsize$L$}             (LDGpi);
    \draw[->]  (LDG)   -- node[left] {\scriptsize$Q_\tau$}        (LDGt);
    \draw[->]  (LDGpi) -- node[right]{\scriptsize$\widetilde{Q}$} (LDGps);
    \draw[<->] (LDGt)  -- node[below]{\scriptsize$\cong$}         (LDGpi);
    \draw[<->] (Gpi) to[bend left=50] node[right]{\scriptsize$\cong$} (LDGps);
  \end{tikzpicture}
  \caption{Commutative diagram relating $G$, its quotient $G/\pi$, their
           associated $LD$-structures, and the induced quotients
           $LD(G)/\tau$ and $LD(G/\pi)$}
  \label{fig:LD-commutative-diagram}
\end{figure}
In this section, we establish PST equivalence between $LD(G)$ and 
$LD(G/\pi)/\sigma$ under the condition $LD(G)/\tau \cong LD(G/\pi)$, 
as shown in Figure~\ref{fig:LD-commutative-diagram}. In 
Sections~\ref{section 9} and~\ref{section 10}, we verify this 
isomorphism explicitly for specific graphs.
In the special case of Theorem~\ref{Isomorphism}, the arc partition
$\tau = \{T_1, \dots, T_{2n}\}$ is indexed by the arcs of $G$,
and the quotient matrix coincides with the adjacency matrix of $G$,
so $LD(G)/\tau \cong G$. 
We consider a $d$-regular
directed graph $G$ with $d \geq 3$ and $m = 2n$ vertices. In general, let
\[
  \tau = \{T_1, T_2, \dots, T_{2m}\}
\]
be an equitable arc partition of $LD(G)$ into $2m$ cells, where $2m$
need not equal $|V(G)|$. Since $|\mathcal{A}(G)| = 2nd$ and $LD(G)$
is $d$-regular (as $G$ is $d$-regular), we have
\[
  |\mathcal{A}(LD(G))| = 2nd \cdot d = 2nd^2.
\]
Assuming further that $LD(G)/\tau$ is $d$-regular, it follows that
\[
  |\mathcal{A}(LD(G)/\tau)| = 2md = 4nd.
\]
The equitable partition condition yields
\[
  A_{LD(G)}\, Q_\tau = Q_\tau\, A',
  \qquad
  A' = Q_\tau^\top A_{LD(G)}\, \quad Q_\tau \in \mathbb{R}^{2nd \times 4n},
\]
where $A'$ is the adjacency matrix of the quotient graph $LD(G)/\tau$.
In general, $A'$ does not coincide with $A_G$. Now let
\(
  \pi = \{C_1, C_2, \dots, C_n\}
\)
be an equitable vertex partition of $V(G)$ into $n$ cells, so that
the quotient graph $G/\pi$ satisfies
\[
  |\mathcal{A}(G/\pi)| = nd.
\]
Let $\sigma = \{\sigma_1, \dots, \sigma_n\}$ be an equitable arc
partition of $G/\pi$ into $n$ cells, with normalized
characteristic matrix $\widetilde{Q}$. Since $\sigma$ has $n$ parts
and $|\mathcal{A}(G/\pi)| = nd$, the matrix $\widetilde{Q}$ has
dimensions
\[
  \widetilde{Q} \in \mathbb{R}^{nd \times n}.
\]
By Lemma~\ref{line digraph 2},
\[
  A_{LD(G/\pi)}\, \widetilde{Q} = \widetilde{Q}\, \widetilde{A},
  \qquad
  \widetilde{A}
  = \widetilde{Q}^\top A_{LD(G/\pi)}\, \widetilde{Q}
  \in \mathbb{R}^{n \times n}.
\]
Since $Q_\tau \neq \widetilde{Q} \otimes I_d$ in general, the
hypotheses of Theorem~\ref{PST theorem 1} are not satisfied, and one cannot
directly conclude that PST occurs in $G$
if and only if it occurs in $G/\pi$. Instead, we relate PST between
$LD(G)$ and $LD(G/\pi)/\sigma$ by imposing additional structural
assumptions ( see Figure~\ref{fig:LD-commutative-diagram}) on the partition $\tau$ together with compatibility
conditions on the associated reflection operators.
\subsection{Relations Between the Transition Operators of 
\texorpdfstring{$LD(G)$}{LD(G)}, 
\texorpdfstring{$LD(G/\pi)$}{LD(G/pi)}, 
\texorpdfstring{$LD(G)/\tau$}{LD(G)/tau}, and 
\texorpdfstring{$LD(G/\pi)/\sigma$}{LD(G/pi)/sigma}}
Since $Q \neq \widetilde{Q}$ and $Q_\tau \neq I_d \otimes \widetilde{Q}$, 
we define reflection operators for the graphs $LD(G/\pi)$, $LD(G)$, and 
$LD(G)/\tau$ analogously to $R$ and $\widetilde{R}$ for $G$ and $G/\pi$, 
as given in Equations~\eqref{quotient reflection} and~\eqref{reflection operator 2}.
Each reflection is of the form $2PP^{\top} - I$ for an appropriate 
isometry $P$, hence unitary, Hermitian, and involutive (i.e.\ $R^2 = I$). Let the transition operators for the three coined quantum walks 
acting on $LD(G/\pi)$, $LD(G)$, and $LD(G)/\tau$ respectively, 
all of degree $d$, be defined as
\[
U_{\sigma} = S_{\sigma}R_{\sigma}, \qquad
U_{\tau} = S_{\tau}R_{\tau}, \qquad
\widetilde{U}_{\tau} = \widetilde{S}_{\tau}\widetilde{R}_{\tau}.
\]
In each case, the shift matrix 
is defined according to the shunt decomposition model as given in 
Equation~\eqref{shift_matrix_}, and the reflection operator $R = 2PP^{\top} - I$ 
serves as the coin operator for the respective graph.
The three reflection operators,
with explicit dimensions, are defined as follows.
\noindent\textit{Reflection operator on $LD(G/\pi)$}
(of size $|\mathcal{A}(G/\pi)|\cdot d \times |\mathcal{A}(G/\pi)|\cdot d
= nd^{2}\times nd^{2}$):
\begin{equation}
  R_{\sigma}
  \;=\;
  2\,
  \underbrace{
    (I_d \otimes \widetilde{Q})
  }_{\displaystyle nd^{2}\,\times\, nd}
  \;
  \underbrace{
    (\widetilde{Q}\,\widetilde{Q}^{\top})
  }_{\displaystyle nd\,\times\, nd}
  \;
  \underbrace{
    (I_d \otimes \widetilde{Q})^{\top}
  }_{\displaystyle nd\,\times\, nd^{2}}
  \;-\; I
  \;\in\;
  \mathbb{C}^{\,nd^{2}\times nd^{2}}.
\end{equation}
Here the ambient space decomposes as 
$\mathbb{C}^{nd^2} \cong \mathbb{C}^d \otimes \mathbb{C}^{nd}$, 
where $\mathbb{C}^d$ is the coin space (the $d$ outgoing arcs per vertex) 
and $\mathbb{C}^{nd}$ is the arc space of $G/\pi$ with 
$|A(G/\pi)| = nd$. The isometry 
$P_{\sigma} = I_d \otimes \widetilde{Q}$ 
(of size $nd^2 \times nd$) embeds the coin tensored with the arc-space 
projector $\widetilde{Q}\widetilde{Q}^{\top}$, so that
$R_{\sigma} = 2P_{\sigma}(\widetilde{Q}\widetilde{Q}^{\top})P_{\sigma}^{\top} - I$
is unitary and Hermitian.

\noindent\textit{Reflection operator on $LD(G)$}
(of size $|\mathcal{A}(G)|\cdot d \times |\mathcal{A}(G)|\cdot d
= 2nd^{2}\times 2nd^{2}$):
\begin{equation}
  R_{\tau}
  \;=\;
  2\,
  \underbrace{
    (I_d \otimes Q_{\tau})
  }_{\displaystyle 2nd^{2}\,\times\, 4nd}
  \;
  \underbrace{
    (I_2 \otimes Q_{\tau}Q_{\tau}^{\top})
  }_{\displaystyle 4nd\,\times\, 4nd}
  \;
  \underbrace{
    (I_d \otimes Q_{\tau})^{\top}
  }_{\displaystyle 4nd\,\times\, 2nd^{2}}
  \;-\; I
  \;\in\;
  \mathbb{C}^{2nd^{2}\times 2nd^{2}}.
\end{equation}
The ambient space decomposes as
$\mathbb{C}^{2nd^2} \cong \mathbb{C}^d \otimes \mathbb{C}^{2nd} 
\cong \mathbb{C}^d \otimes \mathbb{C}^2 \otimes \mathbb{C}^{nd}$,
where the $\mathbb{C}^2$ factor encodes the two-block structure induced 
by the equitable arc partition $\tau$. 
The isometry $P_{\tau} = I_d \otimes Q_{\tau}$ 
(of size $2nd^2 \times 4nd$) embeds into 
$\mathbb{C}^{4nd} \cong \mathbb{C}^2 \otimes \mathbb{C}^{2nd}$. 
The block projector $I_2 \otimes Q_{\tau}Q_{\tau}^{\top}$ acts 
identically on the two copies of $\mathbb{C}^{2nd}$, so that
$R_{\tau} = 2P_{\tau}(I_2 \otimes Q_{\tau}Q_{\tau}^{\top})P_{\tau}^{\top} - I$
is unitary and Hermitian.

\noindent\textit{Reflection operator on $LD(G)/\tau$}
(of size $|\mathcal{A}(LD(G)/\tau)| \times |\mathcal{A}(LD(G)/\tau)|
= 4nd\times 4nd$):
\begin{equation}
  \widetilde{R}_{\tau}
  \;=\;
  2\,(I_2 \otimes Q_{\tau}Q_{\tau}^{\top})
  \;-\; I
  \;\in\;
  \mathbb{C}^{4nd\times 4nd}.
\end{equation}
This operator acts directly on the reduced arc space
$\mathbb{C}^{4nd} \cong \mathbb{C}^2 \otimes \mathbb{C}^{2nd}$, 
where the $\mathbb{C}^2$ factor again denotes the $\tau$-block structure. 
Since $Q_{\tau}Q_{\tau}^{\top}$ is an orthogonal projector, 
$\widetilde{R}_{\tau}$ is unitary, Hermitian, and involutive.

Since each reflection operator is of the form $R = 2PP^{\top} - I$ 
for an appropriate isometry $P$, it satisfies $R^2 = I$ and 
$R = R^{\dagger}$. That is, each reflection is both involutive and 
Hermitian unitary. Furthermore, each is constructed from an isometric 
embedding and the appropriate orthogonal projector. Therefore, 
$R_{\sigma}$, $R_{\tau}$, and $\widetilde{R}_{\tau}$ are valid 
coined quantum-walk reflection operators on the arc spaces of 
$LD(G/\pi)$, $LD(G)$, and $LD(G)/\tau$, respectively.

Using the reflection operators defined above, we establish a relation 
between the transition operators of $LD(G)$, $LD(G/\pi)$, and 
$LD(G)/\tau$ in Proposition~\ref{relation_operators}, and subsequently 
show the equivalence of perfect state transfer in their respective 
line digraphs in Theorem~\ref{lemma_auxiliary_pst}.
\begin{prop}\label{relation_operators}
Let $G$ be a $d$-regular directed graph on $2n$ vertices. Let
$\tau$, $\sigma$, and $\pi$ be the equitable arc partitions of
$LD(G)$, $LD(G/\pi)$, and the equitable vertex partition of $G$ into cells of equal size,
respectively. Let $Q_{\tau}$ and $\widetilde{Q}$ be their normalized
characteristic matrices, satisfying
\[
  Q_{\tau}^{\top} Q_{\tau} = I,
  \qquad
  \widetilde{Q}^{\top}\widetilde{Q} = I.
\]
Let
\[
  U_{\tau} = S_{\tau} R_{\tau},
  \qquad
  U_{\sigma} = S_{\sigma} R_{\sigma},
  \qquad
  \widetilde{U}_{\tau} = \widetilde{S}_{\tau}\widetilde{R}_{\tau},
  \qquad
  \widetilde{U} = \widetilde{S}\,\widetilde{R}
\]
be the unitary evolution operators on $LD(G)$, $LD(G/\pi)$,
$LD(G)/\tau$, and $L(G/\pi)/\sigma,$ respectively, each
of degree $d$. Then the following hold.

\begin{enumerate}
  \item[\textup{(a)}]
    $R_{\tau}\,(I_d \otimes Q_{\tau})
    = (I_d \otimes Q_{\tau})\,\widetilde{R}_{\tau}$
    \quad and \quad
    $R_{\sigma}\,(I_d \otimes \widetilde{Q})
    = (I_d \otimes \widetilde{Q})\,\widetilde{R}.$
\item[\textup{(b)}] For all integers $k \geq 0$,
    \(
      \widetilde{U}_{\tau}^{\,k}
      = (I_d \otimes Q_{\tau})^{\top}\,U_{\tau}^{k}\,
        (I_d \otimes Q_{\tau}),
      \quad
      \widetilde{U}^{\,k}
      = (I_d \otimes \widetilde{Q})^{\top}\,U_{\sigma}^{k}\,
        (I_d \otimes \widetilde{Q}).
    \)
\end{enumerate}
\end{prop}

\begin{proof}
\textbf{(a).}
Set $P = I_d \otimes Q_{\tau}$. Since $P^{\top}P = I$, we compute
\begin{align*}
  R_{\tau}\,P
  &= \Bigl(2\,P\,(I_2 \otimes Q_{\tau}Q_{\tau}^{\top})\,P^{\top}
     - I\Bigr)P = P\,\Bigl(2\,(I_2 \otimes Q_{\tau}Q_{\tau}^{\top}) - I\Bigr) = P\,\widetilde{R}_{\tau}.
\end{align*}
Hence \begin{equation}
R_{\tau}\,(I_d \otimes Q_{\tau})
= (I_d \otimes Q_{\tau})\,\widetilde{R}_{\tau}.\end{equation}
The second identity follows by the identical argument with $Q_{\tau}$
replaced by $\widetilde{Q}$ and $R_{\tau}$ replaced by $R_{\sigma}$.

\textbf{(b).}
Since $\tau$ is an equitable arc partition of $LD(G)$, we have
\(
  A_{LD(G)/\tau} = Q_{\tau}^{\top}\,A_{LD(G)}\,Q_{\tau}.
\)
Since $LD(G)/\tau$ and $LD(G)$ are both $d$-regular, by same argument as 
Lemma~\ref{lemma 3.1} and Remark~\ref{rem:shunt-intertwine-general}, we have 
\(
  \widetilde{S}_{\tau}
  = (I_d \otimes Q_{\tau})^{\top}\,S_{\tau}\,(I_d \otimes Q_{\tau}).
\)
Therefore,
\begin{align*}
  \widetilde{U}_{\tau}
  = \widetilde{S}_{\tau}\,\widetilde{R}_{\tau}
  &= (I_d \otimes Q_{\tau})^{\top}\,S_{\tau}\,(I_d \otimes Q_{\tau})\,
     \widetilde{R}_{\tau} \\
  &= (I_d \otimes Q_{\tau})^{\top}\,S_{\tau}\,R_{\tau}\,
     (I_d \otimes Q_{\tau})
     \qquad\text{(by part~(a))} \\
  &= (I_d \otimes Q_{\tau})^{\top}\,U_{\tau}\,(I_d \otimes Q_{\tau}).
\end{align*}
For the second identity, since $\sigma$ is an equitable arc partition
of $LD(G/\pi)$ and $LD(G/\pi)/\sigma \cong G/\pi$ by
Theorem~\ref{Isomorphism}, both $LD(G/\pi)/\sigma$ and $LD(G/\pi)$ are
$d$-regular, by same argument as 
Lemma~\ref{lemma 3.1} and Remark~\ref{rem:shunt-intertwine-general}, we have 
\(
  \widetilde{S}
  = (I_d \otimes \widetilde{Q})^{\top}\,S_{\sigma}\,
    (I_d \otimes \widetilde{Q}).
\)
Therefore,
\begin{align*}
  \widetilde{U}
  = \widetilde{S}\,\widetilde{R}
  &= (I_d \otimes \widetilde{Q})^{\top}\,S_{\sigma}\,
     (I_d \otimes \widetilde{Q})\,\widetilde{R} \\
  &= (I_d \otimes \widetilde{Q})^{\top}\,S_{\sigma}\,R_{\sigma}\,
     (I_d \otimes \widetilde{Q})
     \qquad\text{(by part~(a))} \\
  &= (I_d \otimes \widetilde{Q})^{\top}\,U_{\sigma}\,
     (I_d \otimes \widetilde{Q}).
\end{align*}
By Lemma~\ref{lem:equitable_equiv}, $(I_d \otimes Q_{\tau})(I_d \otimes Q_{\tau})^\top$ commutes with $U$, and hence 
$\operatorname{Im}(I_d \otimes Q_{\tau})$ is invariant under
$U_{\tau}$.This gives $U_{\tau}\,(I_d \otimes Q_{\tau})
= (I_d \otimes Q_{\tau})\,\widetilde{U}_{\tau}$. Applying this relation inductively yields
\begin{equation}\label{induction}
  U_{\tau}^{k}(I_d \otimes Q_{\tau})
  = (I_d \otimes Q_{\tau})\,\widetilde{U}_{\tau}^{\,k}
\quad \forall k \geq 0\end{equation}
 
Multiplying both sides of (\ref{induction}) on the left by $(I_d \otimes Q_{\tau})^{\top}$ and using
$(I_d \otimes Q_{\tau})^{\top}(I_d \otimes Q_{\tau}) = I$, we obtain
\begin{equation}
  \widetilde{U}_{\tau}^{\,k}
  = (I_d \otimes Q_{\tau})^{\top}\,U_{\tau}^{k}\,
    (I_d \otimes Q_{\tau})
  \quad\text{for all } k \geq 0.
\end{equation}
The identity for $\widetilde{U}^{\,k}$ follows by the identical
argument with $Q_{\tau}$ replaced by $\widetilde{Q}$ and $U_{\tau}$
replaced by $U_{\sigma}$.
\end{proof}

\begin{Theorem}\label{lemma_auxiliary_pst}
Let $G$ be a $d$-regular directed graph on \(2n\) vertices.  Let $\tau = \{T_1, \ldots, T_{2m}, m=2n\}$, 
$\sigma = \{\sigma_1, \ldots, \sigma_{n}\}$, and $\pi = \{C_1, \ldots, C_{n}\}$ 
be equitable arc partitions of $LD(G)$, $LD(G/\pi)$, and vertex partition of $G$ into cells 
of equal size, with normalized characteristic matrices $Q_{\tau}$, 
$\widetilde{Q}$, and $Q$, respectively.let
$\tau = \{T_1, \ldots, T_{2m}, m=2n\}$ be an equitable arc partition of
$LD(G)$ with normalized characteristic matrix $Q_{\tau}$, and let
$\pi = (C_1, \ldots, C_{n})$ be an equitable vertex partition of
$G$.  Assume that
\[
  LD(G)/\tau \;\cong\; LD(G/\pi).
\]
Then the transition operator $U_{\tau}$ of $LD(G)$ exhibits PST from
\[
  (I_d \otimes Q_{\tau})(I_d \otimes \widetilde{Q})\,x
  \quad\text{to}\quad
  (I_d \otimes Q_{\tau})(I_d \otimes \widetilde{Q})\,y
\]
at time $k$ if and only if the transition operator $\widetilde{U}$ of
$LD(G/\pi)/\sigma$ exhibits PST from $x$ to $y$
at time $k$, for all
$x, y \in \operatorname{Im}(I_d \otimes \widetilde{Q})$.
\end{Theorem}

\begin{proof}
Let $U_{\tau}$, $\widetilde{U}_{\tau}$, $U_{\sigma}$, and
$\widetilde{U}$ denote the unitary evolution operators on $LD(G)$,
$LD(G)/\tau$, $LD(G/\pi)$, and $LD(G/\pi)/\sigma$, respectively.
Since $\tau$ is an equitable arc partition of $LD(G)$ and $\sigma$ is
an equitable arc partition of $LD(G/\pi)$, Proposition~\ref{relation_operators}
gives, for every integer $k \geq 0$,
\begin{align}
  U_{\tau}^{k}\,(I_d \otimes Q_{\tau})
  &= (I_d \otimes Q_{\tau})\,\widetilde{U}_{\tau}^{\,k},
  \label{eq:tau-power}\\
  U_{\sigma}^{k}\,(I_d \otimes \widetilde{Q})
  &= (I_d \otimes \widetilde{Q})\,\widetilde{U}^{\,k}.
  \label{eq:sigma-power}
\end{align}

\noindent\textbf{($\Rightarrow$).}
Suppose $\widetilde{U}$ exhibits PST at time $k$,
so that
\[
  \widetilde{U}^{\,k}\,x = y
  \qquad
  \text{for some } x, y \in \operatorname{Im}(I_d \otimes \widetilde{Q}).
\]
Applying \eqref{eq:sigma-power} yields
\[
  U_{\sigma}^{k}\,(I_d \otimes \widetilde{Q})\,x
  = (I_d \otimes \widetilde{Q})\,y.
\]
Since $LD(G)/\tau \cong LD(G/\pi)$, we have $U_{\sigma} =
\widetilde{U}_{\tau}$, and \eqref{eq:tau-power} then gives
\begin{equation}
  U_{\tau}^{k}\,(I_d \otimes Q_{\tau})(I_d \otimes \widetilde{Q})\,x
  = (I_d \otimes Q_{\tau})(I_d \otimes \widetilde{Q})\,y.
\end{equation}
Hence $U_{\tau}$ exhibits PST between the
corresponding lifted states in $LD(G)$.

\noindent\textbf{($\Leftarrow$).}
Conversely, suppose $U_{\tau}$ exhibits PST at
time $k$, so that
\[
  U_{\tau}^{k}\,(I_d \otimes Q_{\tau})(I_d \otimes \widetilde{Q})\,x
  = (I_d \otimes Q_{\tau})(I_d \otimes \widetilde{Q})\,y.
\]
Multiplying on the left by $(I_d \otimes Q_{\tau})^{\top}$ and using
\eqref{eq:tau-power} together with
$(I_d \otimes Q_{\tau})^{\top}(I_d \otimes Q_{\tau}) = I$, we obtain
\[
  U_{\sigma}^{k}\,(I_d \otimes \widetilde{Q})\,x
  = (I_d \otimes \widetilde{Q})\,y.
\]
Multiplying on the left by $(I_d \otimes \widetilde{Q})^{\top}$ and
using \eqref{eq:sigma-power} together with
$(I_d \otimes \widetilde{Q})^{\top}(I_d \otimes \widetilde{Q}) = I$,
we obtain
\begin{equation}
  \widetilde{U}^{\,k}\,x = y.
\end{equation}
Hence $\widetilde{U}$ exhibits PST from $x$ to $y$. Therefore, PST occurs in $LD(G/\pi)/\sigma$
between states $x$ and $y$ if and only if it occurs in $LD(G)$
between the lifted states
$(I_d \otimes Q_{\tau})(I_d \otimes \widetilde{Q})\,x$ and
$(I_d \otimes Q_{\tau})(I_d \otimes \widetilde{Q})\,y$.
\end{proof}
\section{Chebyshev Representation of the Unitary Evolution on Quotient Graphs}\label{section 6}

The powers of the evolution operator $\widetilde{U}$ can be expressed 
in terms of Chebyshev polynomials of the first kind, yielding a 
representation analogous to that of the Grover walk~\cite{kubota2022perfect}. 
This formulation provides explicit criteria for PST
on quotient graphs, which can subsequently be lifted to the original 
graph $G$. We will apply this framework in Section~\ref{section 7} 
to establish PST in the cycle graph $C_{2n}$.

The reduced shift operator \(\widetilde{S}\) of quotient graph \(G/\pi\) is involutory, that is,
\begin{equation}\label{shift matrix_Involutary matrix}
    \widetilde{S}^{2}=I,
\end{equation}
whenever the induced permutation matrix satisfies
\(
\widetilde{P_j}^{2}=I_d
\) by Corollary~\ref{involutary}. We define the \emph{discriminant matrix} of the quotient graph by
\begin{equation}
    \widetilde{D}
    =
    \widetilde{D}(G/\pi)
    =
\widetilde{Q}^{\top}\widetilde{S}\widetilde{Q}
    \in \mathbb{C}^{d\times d}.
\end{equation}
 Taking transpose gives
\[
\widetilde{D}^{\top}
=
(\widetilde{Q}^{\top}\widetilde{S}\widetilde{Q})^{\top}
=
\widetilde{Q}^{\top}\widetilde{S}^{\top}\widetilde{Q}.
\]
Since \(\widetilde{S}\) is unitary and involutory, by
Equations~\eqref{shift matrix unitary opertaor} and
\eqref{shift matrix_Involutary matrix},
\[
\widetilde{S}\widetilde{S}^{\top}=I
\quad \Rightarrow \quad
\widetilde{S}^{\top}=\widetilde{S}^{-1}=\widetilde{S}.
\]
Hence,
\begin{equation}\label{symmetric}
\widetilde{D}^{\top}
=
\widetilde{Q}^{\top}\widetilde{S}\widetilde{Q}
=
\widetilde{D}.
\end{equation}
Therefore, \(\widetilde{D}\) is symmetric.
The matrix \(\widetilde  D\) encodes the effective transition structure of the quotient system. In particular, the dynamics of the walk can be reduced to the action of \(\widetilde  D\), and the \(n\)-th power of the evolution operator can be expressed through Chebyshev polynomials \(T_n(\widetilde  D)\). This establishes a direct connection between the adjacency structure \(\widetilde {A}\) and the unitary dynamics of the shunt decomposition walk. We first show how the discriminant $\widetilde{D}$ relates to the 
adjacency matrix $\widetilde{A}$ of the quotient graph $G/\pi$ under 
a specific characteristic matrix, as given in Lemma~\ref{lem:discriminant} below.
\begin{lemma}\label{lem:discriminant}
Let $G$ be a $d$-regular directed graph on $2n$ vertices with no
internal edges within any cell, and let $\pi=\{C_1,\dots,C_n\}$ be an
equitable partition of $V$.  Let $G/\pi$ be the quotient graph with
adjacency matrix $\widetilde{A}$, arc set
$\mathcal{A}(G/\pi)$, and permutation decomposition
$\widetilde{A}=\widetilde P_1+\cdots+\widetilde P_d$.  Let $\widetilde{S}=\mathrm{diag}(\widetilde P_1,\dots,\widetilde P_d)$
for the reduced shift operator on $\mathcal{A}(G/\pi)$. Suppose the normalized characteristic matrix
$\widetilde{Q}\in\mathbb{R}^{dn\times n}$ of the arc partition of
    $\mathrm{LD}(G/\pi)$ takes one of the following forms:
\begin{enumerate}
    \item
  \[
        \widetilde{Q}
        \;=\;
        \frac{1}{\sqrt{d}}
        \begin{pmatrix}I_n\\I_n\\\vdots\\I_n\end{pmatrix}
        \;\in\mathbb{R}^{dn\times n},
    \]
    corresponding to all $d$ arc-groups from each cell being
        ordered identically; or

    \item 
   \[
        \widetilde{Q}
        \;=\;
        \frac{1}{\sqrt{d}}
        \begin{pmatrix}\widetilde P_1\\\widetilde P_2\\\vdots\\\widetilde P_d\end{pmatrix}
        \;\in\mathbb{R}^{dn\times n},
    \]
   where $\widetilde P_1,\dots,\widetilde P_d$ are the same permutation matrices appearing
    in the decomposition $\widetilde{A}=\widetilde P_1+\cdots+\widetilde P_d$.
\end{enumerate}
Then the {discriminant matrix
\[
    \widetilde{D}
    \;=\;
    \widetilde{Q}^{\!\top}\,\widetilde{S}\,\widetilde{Q}.
\]
satisfies 
\[
    \widetilde{D}
    \;=\;
    \frac{1}{d}\,\widetilde{A},
\] in both cases,
where $\widetilde{A}$ has zero diagonal (since there are no internal
edges within cells).  Moreover, $\widetilde{D}$ is row-stochastic and
every eigenvalue satisfies $|\lambda(\widetilde{D})|\le 1$.}
\end{lemma}
\begin{proof}

\small
\noindent\textit{Case (i): }
With $\widetilde{Q}=\frac{1}{\sqrt{d}}(I_n^{\top}\cdots I_n^{\top})^{\top}$
and $\widetilde{S}=\mathrm{diag}(\widetilde P_1,\dots,\widetilde P_d)$, block multiplication gives
\[
    \widetilde{S}\,\widetilde{Q}
    \;=\;
    \begin{pmatrix}\widetilde P_1&&\\&\ddots&\\&&\widetilde P_d\end{pmatrix}
    \frac{1}{\sqrt{d}}
    \begin{pmatrix}I_n\\\vdots\\I_n\end{pmatrix}
    \;=\;
    \frac{1}{\sqrt{d}}
    \begin{pmatrix}\widetilde P_1\\\vdots\\\widetilde P_d\end{pmatrix},
\]
and therefore
\begin{equation}
    \widetilde{D}
    \;=\;
    \widetilde{Q}^{\!\top}\widetilde{S}\,\widetilde{Q}
    \;=\;
    \frac{1}{d}
    \begin{pmatrix}I_n&\cdots&I_n\end{pmatrix}
    \begin{pmatrix}\widetilde P_1\\\vdots\\\widetilde P_d\end{pmatrix}
    \;=\;
    \frac{1}{d}(\widetilde P_1+\cdots+\widetilde P_d)
    \;=\;
    \frac{1}{d}\,\widetilde{A}.
\end{equation}

\small
\noindent\textit{Case (ii): }
With $\widetilde{Q}=\frac{1}{\sqrt{d}}(\widetilde P_1^{\top}\cdots \widetilde P_d^{\top})^{\top}$
and the same $\widetilde{S}$, block multiplication gives
\[
    \widetilde{S}\,\widetilde{Q}
    \;=\;
    \frac{1}{\sqrt{d}}
    \begin{pmatrix} \widetilde P_1^2\\\vdots\\\widetilde P_d^2\end{pmatrix},
\]therefore
\begin{align*}
    \widetilde{D}
    &\;=\;
    \widetilde{Q}^{\!\top}\widetilde{S}\,\widetilde{Q}
    \;=\;
    \frac{1}{d}
    \begin{pmatrix}\widetilde P_1^{\top}&\cdots&P_d^{\top}\end{pmatrix}
    \begin{pmatrix} \widetilde P_1^2\\\vdots\\\widetilde  P_d^2 \end{pmatrix}
    \;=\;
    \frac{1}{d}\sum_{k=1}^d\widetilde  P_k^{\top}\widetilde P_k^2.
\end{align*}
Since each $\widetilde P_k$ is a permutation matrix we have $\widetilde P_k^{\top}\widetilde P_k=I_n$,
hence $\widetilde P_k^{\top}\widetilde P_k^2=P_k$, and so
\begin{equation}
    \widetilde{D}
    \;=\;
    \frac{1}{d}\sum_{k=1}^d \widetilde P_k
    \;=\;
    \frac{1}{d}\,\widetilde{A}.
\end{equation}

Since each $\widetilde P $ is a permutation matrix, every row of $\widetilde P_k$ sums to
$1$.  Hence every row of $\widetilde{A}=\sum_{k=1}^d \widetilde P_k$ sums to $d$,
and therefore every row of $\widetilde{D}=\frac{1}{d}\widetilde{A}$
sums to $1$, confirming that $\widetilde{D}$ is row-stochastic with
non-negative entries. By the Perron--Frobenius theorem~\cite{horn2013matrix}, the spectral
radius of a row-stochastic matrix equals $1$, so every eigenvalue of
$\widetilde{D}$ satisfies $|\lambda(\widetilde{D})|\le 1$.

\end{proof}

\begin{defn}\label{def:quotient-vertex-type}
Let $ G/\pi$ be the quotient graph of a $d$-regular directed graph
$G$ under an equitable partition $\pi=\{C_1,\dots,C_r\}$. Let
$\widetilde{Q}\in\mathbb{R}^{|\mathcal A(G/\pi)|\times r}$ be the
normalized characteristic matrix of the arc partition of $LD(G/\pi)$,
whose $(a,i)$-entry is
\(
(d|C_i|)^{-1/2}
\)
if the arc $a\in\mathcal A(G/\pi)$ has terminal vertex in $C_i$, and
$0$ otherwise. For each cell $C_i\in\pi$, let
$e_{C_i}\in\mathbb C^r$
denote the standard basis vector with $1$ in the $i$-th position.
The \emph{quotient vertex-type state} associated with $C_i$ is
\begin{equation}
x_{C_i}
=
\widetilde{Q}e_{C_i}
=
\frac{1}{\sqrt{d|C_i|}}
\sum_{\substack{a\in\mathcal A(G/\pi)\\ t(a)\in C_i}} e_a,
\end{equation}
where $e_a$ is the standard basis vector indexed by arc $a$. Thus, $x_{C_i}$ is the normalized uniform superposition over all arcs
of $G/\pi$ whose terminal vertex lies in $C_i$. The set of quotient vertex-type states is
\begin{equation}
\widetilde{\chi}
=
\{\,\widetilde{Q}e_{C_i}: C_i\in\pi\,\}.
\end{equation}
\end{defn}

\begin{defn}\label{def:pst-Q}
Let $G$ be a $d$-regular directed graph with evolution operator $U$,
and let $ G/\pi$ be its quotient with reduced operator
$\widetilde{U}$. Given the lift operator $Q_\tau = I_d\otimes \widetilde{Q}
$, we say that PST occurs
from cell $C_i$ to cell $C_j$ at time $k$ if
\begin{equation}
    \bigl|\langle x_{C_j},\,\widetilde{U}^k x_{C_i}\rangle\bigr|
    \;=\; 1,
\end{equation}
where $x_{C_i} = \widetilde{Q}\,e_{C_i}$ is the quotient vertex-type
state of Definition~\ref{def:quotient-vertex-type}.
\end{defn}
Under the assumption that $\operatorname{Im}(Q_{\tau})$ is
$U$-invariant, Theorem~\ref{PST theorem 1} gives the equivalent
evolution in the parent graph:
\begin{equation}
  U^{k}(Q_{\tau}\,x_{C_i})
  \;=\; Q_{\tau}\,(\widetilde{U}^{k}\,x_{C_i})
  \;=\; Q_{\tau}\,x_{C_j}.
\end{equation}Thus PST between the quotient states $x_{C_i}$ and $x_{C_j}$ implies
PST between the lifted vertex-type states $Q_{\tau}\,x_{C_i}$ and
$Q_{\tau}\,x_{C_j}$ in $G$, and such transfer is fundamentally
governed by the spectral decomposition of the discriminant matrix
\(
  \widetilde{D} = \widetilde{Q}^{\top}\widetilde{S}\,\widetilde{Q}.
\)
As is standard in the analysis of Grover-type quantum walks, the
discrete-time evolution naturally gives rise to Chebyshev polynomials
of the first kind~\cite{kubota2022perfect}.
\begin{defn}\label{def:eigenvalue-support}
For a cell $C_i \in \pi$, the \emph{eigenvalue support} of $C_i$
with respect to the discriminant matrix $\widetilde{D}$ is
\begin{equation}
    \Sigma_{C_i}
    \;=\;
    \bigl\{\,\mu_r \in \mathrm{Spec}(\widetilde{D})
    : \widetilde{E}_r\,e_{C_i} \neq 0\,\bigr\},
\end{equation}
where $\widetilde{E}_r$ is the orthogonal projection onto the
eigenspace of $\widetilde{D}$ corresponding to eigenvalue $\mu_r$.
That is, $\Sigma_{C_i}$ consists of those eigenvalues of
$\widetilde{D}$ whose eigenspace has a non-trivial component in the
direction of $e_{C_i}$.
\end{defn}
A key property used in our analysis is that
\(
|T_n(p)| \leq 1
\quad \text{whenever } |p| \leq 1,
\)
where \(T_n(p)\) denotes the Chebyshev polynomial of the first kind, as discussed in Section~\ref{Chebyshev Polynomials}, this bound plays an important role in evaluating the powers of the discriminant matrix, since the eigenvalues of \(\widetilde{D}\) lie in the interval \([-1,1]\).
\begin{lemma}\label{quotient_Chebyshev}
Let $G/\pi$ be a quotient graph with time-evolution matrix $\widetilde{U}$ 
and discriminant 
\(
\widetilde{D} = \widetilde{Q}^{\top}\widetilde{S}\widetilde{Q}.
\)
Define the time-dependent discriminant for $k \in \mathbb{N} \cup \{0\}$ as
\(
\widetilde{D}_k = \widetilde{Q}^{\top}\widetilde{U}^{k}\widetilde{Q}.
\)
Then $\widetilde{D}_k = T_k(\widetilde{D})$, where $T_k$ is the Chebyshev 
polynomial of the first kind.
\end{lemma}

\begin{proof}
We first observe that for every $k \geq 1$:
\begin{equation}\label{eq:discriminant}
\widetilde{D}_k 
= \widetilde{Q}^{\top} \widetilde{U}^k \widetilde{Q} 
= \widetilde{Q}^{\top} \widetilde{U}^{k-1} \widetilde{S} \widetilde{Q}.
\end{equation}
Indeed, using the shunt-decomposition 
$\widetilde{U} = \widetilde{S}(2\widetilde{Q}\widetilde{Q}^{\top} - I)$:
\begin{align*}
\widetilde{Q}^{\top} \widetilde{U}^k \widetilde{Q}
&= \widetilde{Q}^{\top} \widetilde{U}^{k-1} 
   \bigl[\widetilde{S}(2\widetilde{Q}\widetilde{Q}^{\top} - I)\bigr] 
   \widetilde{Q} \\
&= 2\widetilde{Q}^{\top} \widetilde{U}^{k-1} \widetilde{S} 
   (\widetilde{Q}\widetilde{Q}^{\top}\widetilde{Q}) 
   - \widetilde{Q}^{\top} \widetilde{U}^{k-1} \widetilde{S} \widetilde{Q} \\
&= 2\widetilde{Q}^{\top} \widetilde{U}^{k-1} \widetilde{S} \widetilde{Q} 
   - \widetilde{Q}^{\top} \widetilde{U}^{k-1} \widetilde{S} \widetilde{Q} 
   \qquad (\text{since } \widetilde{Q}^{\top}\widetilde{Q} = I) \\
&= \widetilde{Q}^{\top} \widetilde{U}^{k-1} \widetilde{S} \widetilde{Q}.
\end{align*}
We now proceed by induction on $k$. For the base cases:
\begin{itemize}
    \item $k = 0$: 
    $\widetilde{D}_0 = \widetilde{Q}^{\top} I \widetilde{Q} = I = T_0(\widetilde{D})$.
    \item $k = 1$: 
    $\widetilde{D}_1 = \widetilde{Q}^{\top} \widetilde{S} \widetilde{Q} 
    = \widetilde{D} = T_1(\widetilde{D})$.
\end{itemize}
Assume the identity holds for $k-1$ and $k-2$. 
Using Eq.~\eqref{eq:discriminant} and the shunt-decomposition, we have
\begin{align*}
\widetilde{D}_k 
&= \widetilde{Q}^{\top} \widetilde{U}^{k-1} \widetilde{S} \widetilde{Q} 
   \tag{by \eqref{eq:discriminant}}\\
&= \widetilde{Q}^{\top} 
   \bigl[\widetilde{U}^{k-2}\widetilde{S}
   (2\widetilde{Q}\widetilde{Q}^{\top} - I)\bigr] 
   \widetilde{S}\widetilde{Q} \\
&= 2(\widetilde{Q}^{\top}\widetilde{U}^{k-2}\widetilde{S}\widetilde{Q})
    (\widetilde{Q}^{\top}\widetilde{S}\widetilde{Q}) 
   - \widetilde{Q}^{\top}\widetilde{U}^{k-2}\widetilde{Q} \\
&= 2\widetilde{D}_{k-1}\,\widetilde{D} - \widetilde{D}_{k-2} 
   \tag{by \eqref{eq:discriminant} and induction hypothesis}\\
&= T_k(\widetilde{D}).
   \tag{Chebyshev recurrence}
\end{align*}
Since $T_k(p) = 2p\,T_{k-1}(p) - T_{k-2}(p)$ and the identity holds for 
all base cases, by induction we conclude 
$\widetilde{D}_k = T_k(\widetilde{D})$ for all $k \in \mathbb{N} \cup \{0\}$.
\end{proof}
Using the spectral decomposition of the discriminant $\widetilde{D}$, let
\begin{equation}
\widetilde{D} = \sum_{r=1}^{m} \mu_r \widetilde{E}_r,
\end{equation}
where $\mu_r$ are the eigenvalues of $\widetilde{D}$ and $\widetilde{E}_r$ are the corresponding eigenprojectors. For any polynomial $f$, we have
\begin{equation}
f(\widetilde{D}) = \sum_{r=1}^{m} f(\mu_r) \widetilde{E}_r. 
\end{equation}
\begin{lemma}\label{lem:norm-bound}
Let $G/\pi$ be a quotient graph with shift matrix $\widetilde{S}$ is 
symmetric, and let $\widetilde{D} = \widetilde{Q}^{\top}\widetilde{S}\widetilde{Q}$ 
be the corresponding discriminant. For a cell $C_i \in \pi$ and 
$k \in \mathbb{N} \cup \{0\}$, we have
\[
\|T_k(\widetilde{D})\,e_{C_i}\| \le 1.
\]
The equality holds if and only if $T_k(\mu) = \pm 1$ for every 
$\mu \in \Sigma_{C_i}$.
\end{lemma}

\begin{proof}
The proof is analogous to the vertex-state case in standard Grover 
walks (see~\cite[Lemma~3.2]{kubota2022perfect}).

Since $\widetilde{D} = \widetilde{Q}^{\top}\widetilde{S}\widetilde{Q}$ 
and $\widetilde{S}$ is a symmetric matrix, $\widetilde{D}$ is symmetric, 
and hence all its eigenvalues are real and satisfy 
$\mu \in \Sigma_{C_i} \subset [-1,1]$.
This guarantees that the spectral decomposition
\(
  \widetilde{D} = \sum_{\mu \in \Sigma_{C_i}.} \mu\, \widetilde{E}_\mu
\)
holds with orthogonal projectors $\widetilde{E}_\mu$, and that 
$|T_k(\mu)| \leq 1$ for all $\mu \in \Sigma_{C_i}$.
By expressing $\|T_k(\widetilde{D})\,e_{C_i}\|^2$ through this spectral 
decomposition, we obtain:
\[
\|T_k(\widetilde{D})\,e_{C_i}\|^2 
= \sum_{\mu \in \Sigma_{C_i}} 
  |T_k(\mu)|^2\,\langle \widetilde{E}_\mu\,e_{C_i},\, e_{C_i}\rangle.
\]
Since $|T_k(\mu)| \le 1$ for all $\mu \in  [-1,1]$ and 
$\sum_{\mu}\langle \widetilde{E}_\mu\,e_{C_i}, e_{C_i}\rangle 
= \|e_{C_i}\|^2 = 1$, 
the norm is bounded by $1$. Equality is achieved if and only if 
$|T_k(\mu)| = 1$, i.e.\ $T_k(\mu) = \pm 1$, for all $\mu$ in the 
support $\Sigma_{C_i}$.
\end{proof}
Lemma~\ref{quotient_Chebyshev} and Lemma~\ref{lem:norm-bound} together 
yield a necessary condition for PST to occur in 
the quotient graph $G/\pi$, as stated in Theorem~\ref{thm:pst-necessary} 
below.

\begin{Theorem}\label{thm:pst-necessary}
Let $G/\pi$ be a quotient graph with shift matrix $\widetilde{S}$ is 
symmetric, and let $\widetilde{D} = \widetilde{Q}^{\top}\widetilde{S}\widetilde{Q}$ 
be the corresponding discriminant. If PST occurs from the quotient vertex--type state $x_{C_i}$ to $x_{C_j}$ at time $k$, then 
\(
T_k (\mu) = \pm 1 \quad \text{for every } \mu \in \Sigma_{C_i}.
\)
\end{Theorem}

\begin{proof}
Suppose PST occurs at time $k$. Following the characterization of PST in ~\cite{kubota2022perfect}, the condition $|\langle x_{C_j}, \widetilde{U}^\tau x_{C_i} \rangle| = 1$ is equivalent to $|\langle e_{C_j}, T_k(\widetilde{D}) e_{C_i} \rangle| = 1$ by Lemma~\ref{quotient_Chebyshev}. Applying the Cauchy--Schwarz inequality:
\[
1 = |\langle e_{C_j}, T_k(\widetilde{D}) e_{C_i} \rangle| \le \|e_{C_j}\| \cdot \|T_k(\widetilde{D}) e_{C_i}\| \le 1.
\]
This implies $\|T_k(\widetilde{D}) e_{C_i}\| = 1$, and the result follows directly from Lemma~\ref{lem:norm-bound}.
\end{proof}

\section{Perfect State Transfer on the Quotient Graph of an Even 
Cycle \texorpdfstring{$C_{2n}$}{C2n}, where 
\texorpdfstring{$n$}{n} is Even}
\label{section 7}

\begin{figure}[t]
\centering

\begin{minipage}[t]{0.44\textwidth}
\centering
\begin{tikzpicture}[scale=0.9, every node/.style={scale=0.8}]
  \node[circle,draw,fill=blue!10]   (v1) at (90:1.5)  {$1$};
  \node[circle,draw,fill=green!10]  (v2) at (30:1.5)  {$2$};
  \node[circle,draw,fill=orange!10] (v3) at (-30:1.5) {$3$};
  \node[circle,draw,fill=blue!10]   (v4) at (-90:1.5) {$4$};
  \node[circle,draw,fill=green!10]  (v5) at (-150:1.5){$5$};
  \node[circle,draw,fill=orange!10] (v6) at (150:1.5) {$6$};
  \draw (v1)--(v2)--(v3)--(v4)--(v5)--(v6)--(v1);
  \node[draw=none,font=\small] at (0,2.2) {$C_6$};
\end{tikzpicture}
\subcaption{Original cycle $C_6$.}
\label{fig:c6-original}
\end{minipage}
\hfill
\begin{minipage}[t]{0.44\textwidth}
\centering
\begin{tikzpicture}[>=stealth, scale=1.1,
    every node/.style={circle, draw, font=\small,
                       inner sep=4pt, minimum size=0.8cm}]

  \node[fill=blue!10,   draw=blue!70!black]   (C1) at (90:1.6)   {$C_1$};
  \node[fill=green!10,  draw=green!60!black]  (C2) at (210:1.6)  {$C_2$};
  \node[fill=orange!10, draw=orange!80!black] (C3) at (330:1.6)  {$C_3$};

  \draw[->,thick,blue,  bend left=15] (C1) to (C2);
  \draw[->,thick,green,  bend left=15] (C2) to (C1);
  \draw[->,thick,green,bend left=15] (C2) to (C3);
  \draw[->,thick,orange,bend left=15] (C3) to (C2);
  \draw[->,thick,orange, bend left=15] (C3) to (C1);
  \draw[->,thick,blue, bend left=15] (C1) to (C3);

  \node[draw=none,font=\small] at (0, 2.3) {$K_3 = C_6/\pi$};
\end{tikzpicture}
\subcaption{Quotient graph $C_6/\pi \cong K_3$ with directed edges.}
\label{fig:c6-quotient-k3}
\end{minipage}
\hfill

\begin{minipage}[t]{\textwidth}
\centering
\begin{tikzpicture}[>=stealth, scale=0.7,
    every node/.style={circle, draw, font=\scriptsize,
                       inner sep=2pt, minimum size=0.70cm}]

  \draw[blue!60, dashed, rounded corners=6pt, thick]
        (-1.1, 1.1) rectangle (1.1, 3.3);
  \node[draw=none,font=\scriptsize,text=blue!70!black]
        at (0, 3.65) {$\sigma_1$};
  \node[fill=blue!10, draw=blue!70!black] (b1) at (-0.5, 2.4) {$b_1$};
  \node[fill=blue!10, draw=blue!70!black] (b2) at ( 0.5, 2.4) {$b_2$};

  \draw[green!60!black, dashed, rounded corners=6pt, thick]
        (-4.2,-3.1) rectangle (-2.0,-0.9);
  \node[draw=none,font=\scriptsize,text=green!60!black]
        at (-3.1,-3.5) {$\sigma_2$};
  \node[fill=green!10, draw=green!60!black] (b3) at (-3.6,-2.3) {$b_3$};
  \node[fill=green!10, draw=green!60!black] (b4) at (-2.6,-1.6) {$b_4$};

  \draw[orange!80!black, dashed, rounded corners=6pt, thick]
        ( 2.0,-3.1) rectangle ( 4.2,-0.9);
  \node[draw=none,font=\scriptsize,text=orange!80!black]
        at ( 3.0,-3.4) {$\sigma_3$};
  \node[fill=orange!10, draw=orange!80!black] (b5) at ( 2.6,-1.6) {$b_5$};
  \node[fill=orange!10, draw=orange!80!black] (b6) at ( 3.6,-2.3) {$b_6$};


  \draw[->,thick,blue!70,  bend right=12]  (b1) to (b3);
  \draw[->,thick,blue!70,  bend left=12]  (b2) to (b3);
  \draw[->,thick,blue!70,  bend left=12]  (b1) to (b5);
  \draw[->,thick,blue!70,  bend left=12]  (b2) to (b5);

  \draw[->,thick,green!70, bend right=12] (b3) to (b1);
  \draw[->,thick,green!70, bend right=12] (b4) to (b1);
  \draw[->,thick,green!90, bend right=14] (b3) to (b6);
  \draw[->,thick,green!90, bend left=12] (b4) to (b6);

  \draw[->,thick,orange!70, bend left=12]  (b5) to (b2);
  \draw[->,thick,orange!70, bend right=12]  (b6) to (b2);
  \draw[->,thick,orange!70, bend right=9] (b5) to (b4);
  \draw[->,thick,orange!70, bend left=9] (b6) to (b4);

\end{tikzpicture}
\subcaption{Arc-partition line digraph $\mathrm{LD}(K_3)$:
nodes are arcs of $K_3$ grouped by cell $\sigma_j$
arranged in triangle form; directed edges follow
the rule $(i,j)\to(j,k)$.}
\label{fig:c6-arc-quotient}
\end{minipage}

\parbox{0.92\textwidth}{\small
\textit{Vertex partition:}\;
$\pi = \{C_1=\{1,4\},\; C_2=\{2,5\},\; C_3=\{3,6\}\}$.
\textit{Arc partition:}\;
$\sigma = \{\sigma_1,\sigma_2,\sigma_3\}$ where
$\sigma_1 = \{b_1=(C_2,C_1),\; b_2=(C_3,C_1)\}$,\quad
$\sigma_2 = \{b_3=(C_1,C_2),\; b_4=(C_3,C_2)\}$,\quad
$\sigma_3 = \{b_5=(C_1,C_3),\; b_6=(C_2,C_3)\}$.}

\caption{(a) Original cycle $C_6$. (b) Grouped layout
with cells $C_1$, $C_2$, $C_3$ in triangle form;
inter-cell edges give the quotient $K_3$.
(c) Line digraph $\mathrm{LD}(K_3)$ with arc-partition
cells $\sigma_1$, $\sigma_2$, $\sigma_3$; directed edges follow the adjacency in
$\mathrm{LD}(K_3)$.}
\label{fig:c6-full}
\end{figure}

We consider the cycle graph $C_{2n}$ and construct a quotient graph 
using an equitable partition $\pi$. We partition the vertex set of 
$C_{2n}$ into cells of size $2$, so that the number of cells is \(n\). Hence \(C_{2n}/\pi\) is 2-regular quotient graph by Corollary~\ref{cor:3.1}.
Under this partition, the quotient graph $C_{2n}/\pi$ is isomorphic 
to $C_n$ (see Figure~\ref{fig:c6-full}). Let $\sigma$ be the arc 
partition of the quotient graph $LD(C_{2n}/\pi)$,  obtained by 
grouping arcs according to their terminal vertices. Using this 
structure, we establish PST on $C_{2n}$ in the 
following theorem.
\begin{lemma}\label{cycle_Quotient}
Let $C_{2n}$ be the cycle graph on $2n$ vertices with $n \geq 3$.
Partition the vertex set into $n$ cells of size $2$ defined by
\[
  C_i = \{i,\; i+n\}, \qquad i = 1, 2, \ldots, n.
\]
Then $\pi = \{C_1, C_2, \ldots, C_n\}$ is an equitable partition of
$C_{2n}$, and the quotient graph $C_{2n}/\pi$ is isomorphic to the
cycle graph $C_n$.
\end{lemma}

\begin{proof}
Since the cells $C_i = \{i, i+n\}$ for $i = 1, \ldots, n$ are 
pairwise disjoint and cover $\{1, \ldots, 2n\}$, the collection $\pi$ 
is indeed a partition. For $n \geq 3$, no two vertices within the same 
cell are adjacent in $C_{2n}$, since their difference is $n \geq 3$, 
whereas adjacency requires a difference of $1$ or $2n-1$. Hence 
$G/\pi$ has no self-loops. Now fix $i \in \{1, \ldots, n\}$ and consider the 
two vertices $v = i$ and $w = i + n$ of $C_i$. Their neighbours in 
$C_{2n}$ are
\begin{equation}
  N(v) = \{i-1,\; i+1\}, \qquad N(w) = \{i+n-1,\; i+n+1\}.
\end{equation}
Since $i - 1, \, i+n-1 \in C_{i-1}$ and $i+1,\, i+n+1 \in C_{i+1}$ 
(indices modulo $n$), each vertex of $C_i$ has exactly one neighbour 
in $C_{i-1 \bmod n}$ and exactly one neighbour in $C_{i+1 \bmod n}$, 
and no neighbours in any other cell. As this count is independent of 
the choice of vertex in $C_i$, the partition $\pi$ is equitable. Consequently, in the quotient graph $C_{2n}/\pi$, each cell $C_i$ is 
adjacent only to $C_{i-1 \bmod n}$ and $C_{i+1 \bmod n}$, which is 
precisely the adjacency structure of $C_n$. Hence
\begin{equation}
  C_{2n}/\pi \;\cong\; C_n. 
\end{equation}
\end{proof}

\begin{cor}
\label{cor:cycle_tilde_Q}
Under the hypotheses of Lemma~\ref{cycle_Quotient}, let
$\sigma = \{\sigma_1, \ldots, \sigma_n\}$ be an equitable arc partition of
$\mathrm{LD}(G/\pi)$. By ordering the $2n$ arcs of $\mathrm{LD}(G/\pi)$
so that the first arc of each class $\sigma_1, \ldots, \sigma_n$ precedes
the second arc of each class, the normalized characteristic matrix
$\widetilde{Q}$ coincides with $Q$, and both take the simple form
\[
    Q \;=\; \widetilde{Q} \;=\; \frac{1}{\sqrt{2}}
    \begin{pmatrix} I_n \\ I_n \end{pmatrix}.
\]
\end{cor}

\begin{proof}
By Lemma~\ref{cycle_Quotient}, $G/\pi \cong C_n$ has $n$ undirected
edges, so $\mathrm{LD}(G/\pi)$ has exactly $2n$ directed arcs. Hence
the number of arcs of $\mathrm{LD}(G/\pi)$ equals the number of
vertices of $G = C_{2n}$. The equitable vertex partition $\pi$ partitions the $2n$ vertices of
$G$ into $n$ cells each of size $2$, giving
\[
  Q \;=\; \frac{1}{\sqrt{2}}
  \begin{pmatrix} I_n \\ I_n \end{pmatrix}
  \;\in\; \mathbb{R}^{2n \times n},
\]
where the upper $I_n$ corresponds to the elements $\{1, 2, \ldots, n\}$
and the lower $I_n$ corresponds to the elements
$\{n+1, n+2, \ldots, 2n\}$ of each cell. The arc partition $\sigma$ partitions the $2n$ arcs of
$\mathrm{LD}(G/\pi)$ into $n$ classes each of size $2$, namely the
two arcs directed into each vertex of $C_n$. Ordering the arcs as
\[
  \underbrace{(u_1,v_1),\,(u_2,v_2),\,\ldots,\,(u_n,v_n)}_{
    \text{first arc of each class}},
  \qquad
  \underbrace{(w_1,v_1),\,(w_2,v_2),\,\ldots,\,(w_n,v_n)}_{
    \text{second arc of each class}},
\]
where $(u_i, v_i)$ and $(w_i, v_i)$ are the two arcs in class
$\sigma_i$, the normalized characteristic matrix takes the form
\[
  \widetilde{Q} \;=\; \frac{1}{\sqrt{2}}
  \begin{pmatrix} I_n \\ I_n \end{pmatrix}
  \;\in\; \mathbb{R}^{2n \times n}.
\]
Hence $Q = \widetilde{Q}$, and $\widetilde{Q}$ satisfies the
equitable partition condition
\[
  A_{\mathrm{LD}(G/\pi)}\,\widetilde{Q}
  \;=\; \widetilde{Q}\,\widetilde{A},
  \qquad
  \widetilde{A} = A(C_n),
\]
confirming that
\[
  \mathrm{LD}(G/\pi)/\sigma \;\cong\; C_n \;\cong\; G/\pi. 
\]
\end{proof}
\begin{ex}
\label{ex:C6_quotient}
Take $n = 3$, so $G = C_{2n} = C_6$ with vertices $1, 2, 3, 4, 5, 6$.
The equitable vertex partition $\pi$ is
\[
    C_1 = \{1,4\}, \quad
    C_2 = \{2,5\}, \quad
    C_3 = \{3,6\},
\]
ordering the vertices as $1,2,3,4,5,6$, the normalized characteristic
matrix is
\[
    Q \;=\; \frac{1}{\sqrt{2}}
    \begin{pmatrix} I_3 \\ I_3 \end{pmatrix}
    \;\in\; \mathbb{R}^{6 \times 3},
\]
where the upper $I_3$ corresponds to vertices $\{1,2,3\}$ and the lower
$I_3$ corresponds to vertices $\{4,5,6\}$, satisfying
$A(C_6)\,Q = Q\,A(C_3)$. The neighbour table confirms equitability:

\begin{center}
\begin{tabular}{cccc}
\hline
\rule{0pt}{10pt}%
Vertex & Cell & Neighbours in $C_6$ & Target cells \\[2pt]
\hline
\rule{0pt}{10pt}%
$1$ & $C_1$ & $2,\ 6$ & $C_2,\ C_3$ \\
$4$ & $C_1$ & $3,\ 5$ & $C_3,\ C_2$ \\[4pt]
$2$ & $C_2$ & $1,\ 3$ & $C_1,\ C_3$ \\
$5$ & $C_2$ & $4,\ 6$ & $C_1,\ C_3$ \\[4pt]
$3$ & $C_3$ & $2,\ 4$ & $C_2,\ C_1$ \\
$6$ & $C_3$ & $5,\ 1$ & $C_2,\ C_1$ \\[2pt]
\hline
\end{tabular}
\end{center}

\noindent
The quotient adjacency matrix is
\[
    \widetilde{A}(C_6/\pi)
    \;=\;
    \begin{pmatrix}
        0 & 1 & 1 \\
        1 & 0 & 1 \\
        1 & 1 & 0
    \end{pmatrix}
    \;=\; A(C_3).
\]
Now $C_6/\pi \cong C_3$ has $3$ undirected edges, so $\mathrm{LD}(C_3)$
has $6$ directed arcs. The equitable arc partition $\sigma$ groups the
two arcs directed into each vertex of $C_3$ into one class:
\[
    \sigma_1 = \{(2,1),(3,1)\}, \quad
    \sigma_2 = \{(1,2),(3,2)\}, \quad
    \sigma_3 = \{(1,3),(2,3)\}.
\]
By Corollary~\ref{cor:cycle_tilde_Q}, ordering the arcs as
\[
    \underbrace{(2,1),\,(1,2),\,(1,3)}_{\text{first arc of each class}
    \ \sigma_1,\sigma_2,\sigma_3},\;
    \underbrace{(3,1),\,(3,2),\,(2,3)}_{\text{second arc of each class}
    \ \sigma_1,\sigma_2,\sigma_3},
\]
the normalized characteristic matrix is
\[
    \widetilde{Q} \;=\; \frac{1}{\sqrt{2}}
    \begin{pmatrix} I_3 \\ I_3 \end{pmatrix}
    \;\in\; \mathbb{R}^{6 \times 3},
\]
where the upper $I_3$ corresponds to arcs $(2,1),\,(1,2),\,(1,3)$ and
the lower $I_3$ corresponds to arcs $(3,1),\,(3,2),\,(2,3)$. Hence
$\widetilde{Q} = Q$, satisfying
\[
    A_{\mathrm{LD}(C_3)}\,\widetilde{Q} \;=\; \widetilde{Q}\,A(C_3),
\]
and therefore $\mathrm{LD}(C_3)/\sigma \cong C_3 \cong C_6/\pi$. 
\end{ex}

\begin{cor}
The shift matrix $\widetilde{S} = \mathrm{diag}(\widetilde{P}_1, \widetilde{P}_2)$ 
of the cycle graph $C_n$ is always symmetric.
\end{cor}

\begin{proof}
Since $C_n$ has degree $2$, its adjacency matrix satisfies 
$A(C_n) = \widetilde{P}_1 + \widetilde{P}_2$, where $\widetilde{P}_1$ 
and $\widetilde{P}_2$ are the two permutation matrices corresponding 
to the two arc directions. Since each $\widetilde{P}_j$ is an involution, 
i.e.\ $\widetilde{P}_j^2 = I$, it follows by 
Corollary~\ref{involutary} that $\widetilde{S}^2 = I$. 
Since $\widetilde{S}$ is unitary and satisfies $\widetilde{S}^2 = I$, 
we have $\widetilde{S}^{-1} = \widetilde{S}$. Combined with the 
unitarity condition $\widetilde{S}^{-1} = \widetilde{S}^{\top}$, 
this gives $\widetilde{S} = \widetilde{S}^{\top}$, and hence 
$\widetilde{S}$ is symmetric.
\end{proof}
Since $\widetilde{S}$ is symmetric by Equation~\eqref{symmetric}, 
the discriminant $\widetilde{D} = \widetilde{Q}^{\top}\widetilde{S}\widetilde{Q}$ 
is also symmetric. We now apply Theorem~\ref{thm:pst-necessary} to establish 
perfect state transfer in the quotient graph $C_{2n}$, as demonstrated in 
Lemma~\ref{lem:pst-cycle-quotient} below with an explicit example. 
Subsequently, we lift this result to the parent graph $C_{2n}$ in Theorem~\ref{thm:pst-cycle} via 
Theorem~\ref{PST theorem 1}.

\begin{lemma}\label{lem:pst-cycle-quotient} 
Let $C_{2n}$ be the cycle graph on $2n$ vertices, where $n \geq 2$
is \emph{even}, and let $\pi = \{C_i\}_{i=1}^n$ with
$C_i = \{i,\,i+n\}$ be the equitable vertex partition of $C_{2n}$.  Let
$\widetilde  D $ denotes the discriminant of the quotient graph $C_{2n}/\pi$.
Then PST occurs on the quotient graph from
$\widetilde{Q}e_{C_i}$ to $\widetilde{Q}e_{C_j}$ at time
\(
  k = \frac{n}{2}.
\)
\end{lemma}

\begin{proof}
By Lemma~\ref{lem:discriminant} and
Corollary~\ref{cor:cycle_tilde_Q}, the quotient graph is
$C_{2n}/\pi \cong C_n$ and the discriminant satisfies
$\widetilde{D} = \tfrac{1}{2}A(C_n)$.
The eigenvalues of $A(C_n)$ are $2\cos\!\bigl(\tfrac{2\pi t}{n}\bigr)$
for $t = 0,1,\ldots,n-1$, so the eigenvalues of $\widetilde{D}$ are
\begin{equation}\label{eq:disc-evals}
  \mu_t = \cos\!\left(\frac{2\pi t}{n}\right),
  \qquad t = 0, 1, \ldots, n-1.
\end{equation}

By Theorem~\ref{thm:pst-necessary}, PST from
$\widetilde{Q}e_{C_i}$ to $\widetilde{Q}e_{C_j}$ at time $k$
requires
\[
  T_k(\mu_t) = \pm 1
  \qquad
  \text{for all } t = 0, 1, \ldots, n-1
  \text{ and }
  \mu_t \in \Sigma_{C_i}.
\]
where $T_k$ is the Chebyshev polynomial of degree $k$.
Using $T_k(\cos\theta) = \cos(k\theta)$, this condition becomes
\[
  \cos\!\left(k \cdot \frac{2\pi t}{n}\right) = \pm 1
  \quad\iff\quad
  k \cdot \frac{2t}{n} \in \mathbb{Z}
  \quad \text{for all } t = 0,1,\ldots,n-1.
\]
To satisfy this for all $t$, it suffices to consider $t=1$, which yields
\[
  \frac{2k}{n} \in \mathbb{Z}
  \quad\Longrightarrow\quad
  k = m\,\frac{n}{2}, \quad m \in \mathbb{Z}_{>0}.
\]
The smallest positive value is $k = \dfrac{n}{2}$, which is a
positive integer since $n$ is even.
Let $E_t$ denote the spectral projector of $\widetilde{D}$ onto the
eigenspace of $\mu_t$.  At $k = n/2$, the sign pattern is
\[
  T_{n/2}(\mu_t)
  = \cos\!\left(\frac{n}{2}\cdot\frac{2\pi t}{n}\right)
  = \cos(\pi t )
  = (-1)^t,
\]
so the transfer operator is
\begin{equation}
  T_{n/2}(\widetilde{D})
  = \sum_{t=0}^{n-1}(-1)^t E_t.
\end{equation}
Since $n$ is even and $C_n$ is vertex-transitive, the antipodal
symmetry of $C_n$ (sending cell $C_i$ to cell $C_{i+n/2 \bmod n}$)
implies
\[
  T_{n/2}(\widetilde{D})\,e_{C_i}
  = e_{C_{i+n/2 \bmod n}},
\]
that is, the transfer operator maps each basis vector to the
diametrically opposite basis vector of $C_n$.  Hence
\begin{equation}
  \bigl\langle
    \widetilde{U}^{n/2}\,\widetilde{Q}e_{C_i},\;
    \widetilde{Q}e_{C_j}
  \bigr\rangle = 1,
  \qquad
  j = i + \tfrac{n}{2} \bmod n,
\end{equation}
establishing PST at time $k = n/2$.
\end{proof}

\begin{ex}
\label{ex:pst-C8-C4}

Take $2n = 8$, so $n = 4$ (even, satisfying the lemma).
The vertex partition $C_i = \{i,\,i+4\}$ gives
\[
  C_1=\{1,5\},\quad C_2=\{2,6\},\quad C_3=\{3,7\},\quad C_4=\{4,8\},
\]
with quotient $C_8/\pi \cong C_4$.
The discriminant is $\widetilde{D} = \tfrac{1}{2}A(C_4)$, with
eigenvalues
\[
  \mu_t = \cos\!\left(\frac{2\pi t}{4}\right)
        = \cos\!\left(\frac{\pi t}{2}\right),
  \qquad t = 0,1,2,3,
\]
namely
\[
  \mu_0 = 1,\quad \mu_1 = 0,\quad \mu_2 = -1,\quad \mu_3 = 0.
\]
By the lemma, $k = n/2 = 2$.  The Chebyshev polynomial of
degree $2$ is
\[
  T_2(x) = 2x^2 - 1.
\]
Evaluating on the spectrum:
\(
  T_2(1) = 1,\quad
  T_2(0) = -1,\quad
  T_2(-1) = 1,\quad
  T_2(0) = -1.
\)
Thus $T_2(\mu_t) = (-1)^t \in \{+1,-1\}$ for all $t$, confirming the
PST condition.
The distinct eigenvalues of $\widetilde{D}$ are
$\mu = +1, 0, -1$.  Their spectral projectors are:
\[
  E_0 = \frac{1}{4}
  \begin{pmatrix}1&1&1&1\\1&1&1&1\\1&1&1&1\\1&1&1&1\end{pmatrix},
  \qquad
  E_2 = \frac{1}{4}
  \begin{pmatrix}1&{-1}&1&{-1}\\{-1}&1&{-1}&1\\
                 1&{-1}&1&{-1}\\{-1}&1&{-1}&1\end{pmatrix},
\]
\[
  E_{1,3} = \frac{1}{2}
  \begin{pmatrix}1&0&{-1}&0\\0&1&0&{-1}\\
                 {-1}&0&1&0\\0&{-1}&0&1\end{pmatrix}.
\]
These satisfy $E_0 + E_{1,3} + E_2 = I_4$.

\[
  T_2(\widetilde{D})
  = (+1)E_0 + (-1)E_{1,3} + (+1)E_2
  = 2\widetilde{D}^2 - I_4
  =
  \begin{pmatrix}
    0 & 0 & 1 & 0 \\
    0 & 0 & 0 & 1 \\
    1 & 0 & 0 & 0 \\
    0 & 1 & 0 & 0
  \end{pmatrix}.
\]
This is the antipodal permutation matrix of $C_4$, and its action is
\[
  T_2(\widetilde{D})\,e_{C_1} = e_{C_3},\quad
  T_2(\widetilde{D})\,e_{C_2} = e_{C_4},\quad
  T_2(\widetilde{D})\,e_{C_3} = e_{C_1},\quad
  T_2(\widetilde{D})\,e_{C_4} = e_{C_2}.
\]
Hence PST occurs at $k = 2$ between the
antipodal pairs:
\[
  \widetilde{Q}e_{C_1} \;\longleftrightarrow\; \widetilde{Q}e_{C_3},
  \qquad
  \widetilde{Q}e_{C_2} \;\longleftrightarrow\; \widetilde{Q}e_{C_4}.
\]
\end{ex}

\begin{Theorem}\label{thm:pst-cycle}
Let $G = C_{2n}$ be the even cycle on $2n$ vertices, where
$n \geq 2$ is even, and let $\pi = \{C_i\}_{i=1}^{n}$ with
$C_i = \{i,\, i+n\}$ be the equitable vertex partition of $G$. Then $G$ exhibits PST at time
\(
  k = \frac{n}{2}
\)
from the state $Q_\tau\, x_{C_i}$ to the state
$Q_\tau\, x_{C_j}$, where
$j \equiv i + \tfrac{n}{2} \pmod{n}$.
\end{Theorem}

\begin{proof}

By Lemma~\ref{cycle_Quotient}, the partition $\pi$ is
equitable with $n$ cells each of size  $2$, and the quotient
satisfies $C_{2n}/\pi \cong C_n$.
By Corollary~\ref{cor:cycle_tilde_Q}, the vertex and arc
characteristic matrices coincide:
\[
  Q
  \;=\;
  \widetilde{Q}
  \;=\;
  \frac{1}{\sqrt{2}}
  \begin{pmatrix} I_n \\ I_n \end{pmatrix}
  \;\in\; M_{2n \times n}(\mathbb{R}),
  \qquad
  \widetilde{Q}^\top \widetilde{Q} = I_n.
\]
The arc characteristic matrix of $\mathrm{LD}(C_{2n})$ takes the
form
\[
  Q_\tau
  \;=\;
   I_2\otimes \widetilde{Q} 
  \;=\;
  \frac{1}{\sqrt{2}}
  \begin{pmatrix} I_{2n} \\ I_{2n} \end{pmatrix}
  \;\in\; M_{4n \times 2n}(\mathbb{R}),
  \qquad
  Q_\tau^\top Q_\tau = I_{2n}.
\]
By construction, $Q_\tau$ satisfies the intertwining identity
\[
  A_{\mathrm{LD}(C_{2n})}\, Q_\tau
  \;=\;
  Q_\tau\, A_{C_{2n}},
\]
where $A_{\mathrm{LD}(C_{2n})} \in M_{4n \times 4n}(\mathbb{R})$
and $A_{C_{2n}} \in M_{2n \times 2n}(\mathbb{R})$. By Lemma~\ref{lem:pst-cycle-quotient}, the quotient $C_n$
exhibits PST at time $k = \tfrac{n}{2}$
(a positive integer since $n$ is even) between the antipodal
states $\widetilde{Q}\,e_{C_i}$ and $\widetilde{Q}\,e_{C_j}$,
that is,
\[
  \left|
    e_{C_j}^\top\,
    \widetilde{Q}^\top\,
    \widetilde{U}^{\,k}\,
    \widetilde{Q}\,
    e_{C_i}
  \right| = 1.
\]
Applying Theorem~\ref{PST theorem 1} with
$Q_\tau = I_2 \otimes \widetilde{Q}$, PST on the quotient lifts
to the original graph, giving
\[
  \left|
    x_{C_j}^\top\, Q_\tau^\top\, U^{k}\, Q_\tau\, x_{C_i}
  \right|
  \;=\;
  \left|
    e_{C_j}^\top\,
    \widetilde{Q}^\top\,
    \widetilde{U}^{\,k}\,
    \widetilde{Q}\,
    e_{C_i}
  \right|
  \;=\; 1,
\]
where $U$ and $\widetilde{U}$ are the transition operators on
$C_{2n}$ and $C_n$, respectively. Hence $C_{2n}$ exhibits PST at time
$k = \tfrac{n}{2}$ from the state $Q_\tau\, x_{C_i}$ to the
state $Q_\tau\, x_{C_j}$, with
$j \equiv i + \tfrac{n}{2} \pmod{n}$.
\end{proof}

\section{Perfect State Transfer on the Quotient Graph \texorpdfstring{$K_{n}^{\circlearrowleft}$}{Kn} for \texorpdfstring{$n=2^s$}{n=2s} for some integer \texorpdfstring{$s\ge 1$}{s>=1}}
\label{section 8}
Complete graphs with self-loops are among the most symmetric graph
structures and have received considerable attention in the study of
quantum walks. Szegedy quantum walks on complete graphs with
self-loops have been investigated, where the success probability
approaches unity as the number of vertices $N$ becomes large~\cite{PhysRevA.94.022301}.
Moreover, for a complete graph with self-loops on $N = 30$ vertices,
the first maximum of the fidelity is attained after $12$ steps under
the Szegedy quantum walk model.

Motivated by these results, we investigate perfect state transfer
(PST) on the complete graph with self-loops $K_n^{\circlearrowleft}$
on $n$ vertices, where $n$ is a power of $2$. Specifically, we
consider a $d$-regular directed graph $G$ on $2n$ vertices whose
quotient graph is $K_n^{\circlearrowleft}$ of degree $d$, and we
establish conditions under which PST occurs in the corresponding
quotient quantum walk. The results of this section serve as a
tool for proving PST on broader families of quotient graphs,
as developed in Sections~\ref{section 9} and~\ref{section 10}.
Let the shift matrix of $K_n^{\circlearrowleft}$ be 
\[
\widetilde{S}
=
\sum_{j=1}^{d} E_{jj}\otimes \widetilde   P_j,
\]
where \(\widetilde  P_1,\widetilde  P_2,\ldots,\widetilde  P_d\) are the shunts of the adjacency matrix
\(A(K_{n}^{\circlearrowleft})\) satisfying
\[
A(K_{n}^{\circlearrowleft})
=
\widetilde  P_1+ \widetilde  P_2+\cdots+ \widetilde  P_d=I_n+\widetilde  P+\widetilde   P^2 \cdots \widetilde  P^{d-1}.
\]

For the specific choice
\[
\widetilde{Q}
=
\frac{1}{\sqrt d}
\begin{bmatrix}
I_n\\
\widetilde  P\\
\widetilde  P^2\\
\vdots\\
\widetilde  P^{d-1}
\end{bmatrix}
\]
where \(\widetilde  P\) is the cyclic permutation matrix, and with reflction operator \(\widetilde  R\) as defined in
Equation~\ref{quotient reflection}, we obtain
\[
\widetilde{S}
=
\bigoplus_{k=0}^{n-1} \widetilde  P^k,
\qquad
\widetilde{U}
=
\widetilde{S}\widetilde{R}.
\]

The following lemma determines the eigenvalues of
\(\widetilde{U}\), and the subsequent theorem establishes PST on the quotient graph \(K_{n}^{\circlearrowleft}\).

\begin{lemma} \label{eigen_values_}
Let $n \geq 1$ and set $d = n$. Given
\[
  \widetilde{Q} = \frac{1}{\sqrt{n}}
  \begin{bmatrix} I_n \\ \widetilde{P} \\ \widetilde{P}^2 \\ \vdots \\ 
  \widetilde{P}^{n-1} \end{bmatrix},
  \qquad
  R = 2\widetilde{Q}\widetilde{Q}^\top - I_{n^2},
  \qquad
  S = \bigoplus_{k=0}^{n-1} \widetilde{P}^k,
  \qquad
  \widetilde{U} = \widetilde S \widetilde R,
\]
where $\widetilde{P}$ is the cyclic shift on $\mathbb{C}^n$,
the eigenvalues of $\widetilde{U} \in \mathbb{C}^{n^2 \times n^2}$
are the roots of the $n$th-degree characteristic polynomial
\[
  p_l(\lambda)
  = \prod_{m=0}^{n-1}(\lambda + \omega^{lm})
  - \frac{2}{n}\sum_{m=0}^{n-1} \omega^{-ml}
    \prod_{\substack{j=0\\j\neq m}}^{n-1}(\lambda + \omega^{lj}),
  \qquad l \in \mathbb{Z}_n,
\]
where $\omega = e^{2\pi i/n}$.
\end{lemma}

\begin{proof}
Since $d = n$, the same root of unity $\omega = e^{2\pi i/n}$ governs
both the cyclic shift $\widetilde{P}$ on $\mathbb{C}^n$ and the
DFT structure on $\mathbb{C}^n$.
The DFT basis vectors
\[
  f_l = \frac{1}{\sqrt{n}}\bigl[1,\,\omega^{l},\,\omega^{2l},\,\ldots,\,
  \omega^{(n-1)l}\bigr]^\top, \qquad l \in \mathbb{Z}_n,
\]
satisfy $\widetilde{P} f_l = \omega^l f_l$.
The subspaces
\(
  V_l = \operatorname{span}\{f_m \otimes f_l : m \in \mathbb{Z}_n\},
  \qquad l \in \mathbb{Z}_n,
\)
are $\widetilde{U}$-invariant and
$\mathbb{C}^{n^2} = \bigoplus_{l=0}^{n-1} V_l$.
Indeed, since $d = n$, the action of $S$ gives
$S(f_m \otimes f_l) = f_{m+l \bmod n} \otimes f_l \in V_l$,
and the action of $R$ gives
$R(f_m \otimes f_l) = d_m^{(l)}(f_m \otimes f_l) \in V_l$,
both of which keep $V_l$ invariant.
It therefore suffices to find the eigenvalues of
$\widetilde{U}_l = \widetilde{U}|_{V_l}$ for each $l \in \mathbb{Z}_n$. In the basis $\{f_m \otimes f_l\}_{m \in \mathbb{Z}_n}$ one has
\begin{equation}
  \widetilde{U}_l = \Sigma_l D_l,
\end{equation}
where $\Sigma_l$ is the cyclic shift by $l$ on $\mathbb{C}^n$
(that is\ $(\Sigma_l)_{m,m'} = \delta_{m,\,m'+l \bmod n}$),
arising from $S|_{V_l}$, and
$D_l = \operatorname{diag}(d_0^{(l)},\ldots,d_{n-1}^{(l)})$
with
\[
  d_m^{(l)} =
  \begin{cases} +1 & m = l,\\ -1 & m \neq l,\end{cases}
\]
arising from $R|_{V_l}$. Write $D_l$ as a rank-one update of $-I_n$,
\(
  D_l = -I_n + 2\,e_l e_l^\top,
\)
since this matrix has $+1$ in position $(l,l)$ and $-1$ elsewhere on
the diagonal. Hence
\[
  \widetilde{U}_l = \Sigma_l D_l
  = \Sigma_l(-I_n + 2\,e_l e_l^\top)
  = -\Sigma_l + 2\,(\Sigma_l e_l)\,e_l^\top.
\]
This expresses $\lambda I - \widetilde{U}_l$ as a rank-one perturbation
of $\lambda I + \Sigma_l$
\begin{equation}
  \lambda I - \widetilde{U}_l
  = \underbrace{(\lambda I + \Sigma_l)}_{=\,A}
    + \underbrace{(-2\,\Sigma_l e_l)}_{=\,u}
      \,\underbrace{e_l^\top}_{=\,v^\top}.
  \label{eq:rankone}
\end{equation}
For an invertible matrix $A$ and vectors $u,v$,
\[
  \det(A + uv^\top) = \det(A)\,(1 + v^\top A^{-1} u).
\]
Applying this to \eqref{eq:rankone} with
$A = \lambda I + \Sigma_l$, $u = -2\Sigma_l e_l$, $v = e_l$
\begin{align}
  \det(\lambda I - \widetilde{U}_l)
  &= \det(\lambda I + \Sigma_l)
     \Bigl(1 - 2\,e_l^\top(\lambda I + \Sigma_l)^{-1}\Sigma_l e_l\Bigr).
  \label{eq:MDL}
\end{align}

The matrix $\Sigma_l$ is the cyclic permutation
$m \mapsto m + l \bmod n$ on $\mathbb{C}^n$,
whose eigenvalues are $\{\omega^{lm} : m \in \mathbb{Z}_n\}$
with $\omega = e^{2\pi i/n}$. Therefore
\(
  \det(\lambda I + \Sigma_l) = \prod_{m=0}^{n-1}(\lambda + \omega^{lm}).
\)
Now we compute $e_l^\top(\lambda I + \Sigma_l)^{-1}\Sigma_l e_l$.
The shift $\Sigma_l$ maps the $k$-th standard basis vector to
$e_{k+l \bmod n}$, so in particular
\(
  \Sigma_l e_l = e_{2l \bmod n}.
\)
Hence the resolvent entry becomes
\begin{equation}\label{eq:res}
  e_l^\top(\lambda I + \Sigma_l)^{-1}\Sigma_l e_l
  = e_l^\top(\lambda I + \Sigma_l)^{-1} e_{2l \bmod n}
  = \bigl[(\lambda I + \Sigma_l)^{-1}\bigr]_{l,\,2l \bmod n}.
\end{equation}
Since $\Sigma_l$ is a circulant matrix on $\mathbb{C}^n$,
so is $(\lambda I + \Sigma_l)^{-1}$. Its $(j,k)$ entry is
\[
  \bigl[(\lambda I + \Sigma_l)^{-1}\bigr]_{j,k}
  = \frac{1}{n}\sum_{m=0}^{n-1}\frac{\omega^{m(j-k)}}{\lambda + \omega^{lm}}.
\]
Setting $j = l$ and $k = 2l \bmod n$ gives $j - k \equiv -l \pmod{n}$,
so
\begin{equation}\label{eq:entry}
  \bigl[(\lambda I + \Sigma_l)^{-1}\bigr]_{l,\,2l \bmod n}
  = \frac{1}{n}\sum_{m=0}^{n-1}\frac{\omega^{-ml}}{\lambda + \omega^{lm}}.
\end{equation}
Substituting \eqref{eq:res} and \eqref{eq:entry} into \eqref{eq:MDL}:
\[
  \det(\lambda I - \widetilde{U}_l)
  = \prod_{m=0}^{n-1}(\lambda + \omega^{lm})
    \left(1 - \frac{2}{n}\sum_{m=0}^{n-1}
    \frac{\omega^{-ml}}{\lambda + \omega^{lm}}\right).
\]
Clearing denominators by multiplying through gives
\begin{equation}\label{det}
  \det(\lambda I - \widetilde{U}_l)
  = \prod_{m=0}^{n-1}(\lambda + \omega^{lm})
    - \frac{2}{n}\sum_{m=0}^{n-1}
      \omega^{-ml}\prod_{\substack{j=0\\j\neq m}}^{n-1}(\lambda + \omega^{lj})
  = p_l(\lambda).
\end{equation}
This is a degree-$n$ polynomial in $\lambda$, valid for each
$l \in \mathbb{Z}_n$. Thus the eigenvalues of $\widetilde{U}$ are
$\bigcup_{l=0}^{n-1}\{\lambda : p_l(\lambda) = 0\}$.
\end{proof}
\begin{lemma}\label{lem:Uln_general}
Let $n = 2^s$ for some integer $s \geq 1$, and let
$\widetilde U_l = \Sigma_l D_l$ be the restriction of
$\widetilde{U} = \widetilde S\widetilde R$ to the invariant subspace
$V_l$, for $l \in \mathbb{Z}_n$. Then the following hold.

\noindent\textup{(i)}
For every $l \in \mathbb{Z}_n$,
\[
  \widetilde U_l^n = (-1)^{l \bmod n}\,I_n, \qquad \widetilde U_l^{2n} = I_n.
\]
That is, the order of $\widetilde U_l$ is exactly $2n$ for every
$l \in \mathbb{Z}_n$ with $l \bmod n$ odd, and divides $2n$ for
every $l \in \mathbb{Z}_n$ with $l \bmod n$ even.

\noindent\textup{(ii)}
Since $d = n$, the identity
\(
  \widetilde{U}^{\,n} = I_n \otimes \widetilde{P}^{n/2}
\)
holds.
\end{lemma}

\begin{proof}

\noindent\textit{Proof of part~(i).} Recall from Lemma~\ref{eigen_values_} that $\widetilde{U}_l = \Sigma_l D_l$,
where $\Sigma_l$ is the cyclic shift by $l$ on $\mathbb{C}^n$ and
$D_l = \operatorname{diag}(d_0^{(l)}, \ldots, d_{n-1}^{(l)})$ with
$d_m^{(l)} = +1$ if $m = l \bmod n$ and $d_m^{(l)} = -1$ otherwise.
We first establish that for all $k \geq 1$,
\begin{equation}\label{eq:ind}
  \widetilde U_l^k = \Sigma_{kl \bmod n}\,E_k,
\end{equation}
where $\Sigma_{kl}$ denotes the cyclic shift by $kl \bmod n$ on
$\mathbb{C}^n$ and $E_k$ is the $n \times n$ diagonal matrix with
entries
\begin{equation}\label{eq:Ek}
  (E_k)_{mm} = \prod_{r=0}^{k-1} d_{m-rl \bmod n}^{(l)},
  \qquad m \in \mathbb{Z}_n,
\end{equation}

\noindent\textit{Base case $k = 1$.}
We have $\widetilde U_l^1 = \Sigma_l D_l = \Sigma_l E_1$ since
$(E_1)_{mm} = d_m^{(l)}$.

\noindent\textit{Inductive step.}
Assume $\widetilde U_l^k = \Sigma_{kl} E_k$. Then
\(
  \widetilde U_l^{k+1} = \Sigma_{kl} E_k \cdot \Sigma_l D_l.
\)
Since $E_k$ is diagonal and $\Sigma_l$ is the shift by $l$ on
$\mathbb{C}^n$, commuting $E_k$ past $\Sigma_l$ relabels the
diagonal indices:
\[
  E_k \Sigma_l = \Sigma_l \widehat{E}_k,
  \qquad
  (\widehat{E}_k)_{mm} = (E_k)_{m-l \bmod n,\; m-l \bmod n}.
\]
Hence
\[
  \widetilde U_l^{k+1} = \Sigma_{(k+1)l}\,\widehat{E}_k D_l
  = \Sigma_{(k+1)l}\,E_{k+1},
\]
where
\[
  (E_{k+1})_{mm}
  = (\widehat{E}_k)_{mm}\cdot d_m^{(l)}
  = \prod_{r=0}^{k-1} d_{m-(r+1)l}^{(l)} \cdot d_m^{(l)}
  = \prod_{r=0}^{k} d_{m-rl}^{(l)}.
\]
This completes the induction. At $k = n$ the shift satisfies
$\Sigma_{nl} = \Sigma_0 = I_n$ since $nl \equiv 0 \pmod{n}$ for
every $l$, so
\begin{equation}
  \widetilde U_l^n = E_n,
\end{equation}
and it remains to determine the diagonal entries of $E_n$.
Fix $m \in \mathbb{Z}_n$ and consider the sequence of indices
\[
  m,\; m-l,\; m-2l,\; \ldots,\; m-(n-1)l \pmod{n}.
\]
Let $g = \gcd(l \bmod n,\,n)$. The map $r \mapsto m - rl \bmod n$
has period $n/g$, so as $r$ runs over $\{0,1,\ldots,n-1\}$ each
element of the orbit
\[
  \mathcal{O}(m,l) = \{m - rl \bmod n : r \in \mathbb{Z}_n\}
\]
is visited exactly $g$ times. The factor $d_{m-rl}^{(l)}$ equals
$+1$ only when $m - rl \equiv l \pmod{n}$. The number of such
$r \in \{0,\ldots,n-1\}$ is
\[
  N_+(m,l)
  = \#\{r \in \mathbb{Z}_n : rl \equiv m-l \pmod{n}\}
  = \begin{cases}
      g & \text{if } g \mid (m-l),\\
      0 & \text{if } g \nmid (m-l).
    \end{cases}
\]
The remaining $n - N_+(m,l)$ factors each equal $-1$, so
\begin{equation}\label{eq:Enentry}
  (E_n)_{mm} = (+1)^{N_+}\cdot(-1)^{n-N_+} = (-1)^{n-N_+}.
\end{equation}

\paragraph{Case 1: $g \nmid (m-l)$.}
Then $N_+ = 0$, so $(E_n)_{mm} = (-1)^n = +1$ since $n = 2^s$ is
even.

\paragraph{Case 2: $g \mid (m-l)$.}
Then $N_+ = g$, so $(E_n)_{mm} = (-1)^{n-g}$. Write
$g = \gcd(l \bmod n,\,n) = 2^a$ with $0 \leq a \leq s$, so that
$n - g = 2^s - 2^a = 2^a(2^{s-a}-1)$.
\begin{itemize}
  \item If $a \geq 1$ (that is\ $l \bmod n$ is \emph{even}):
    $n - g$ is even, so $(-1)^{n-g} = +1$.
  \item If $a = 0$ (that is\ $l \bmod n$ is \emph{odd}, $g = 1$):
    $n - g = 2^s - 1$ is odd for $s \geq 1$, so
    $(-1)^{n-g} = -1$.
\end{itemize}

\noindent Combining both cases, for every $m \in \mathbb{Z}_n$,
\[
  (E_n)_{mm} =
  \begin{cases}
    +1 = (-1)^{l \bmod n} & \text{if } l \bmod n \text{ is even},\\
    -1 = (-1)^{l \bmod n} & \text{if } l \bmod n \text{ is odd}.
  \end{cases}
\]
Hence $E_n = (-1)^{l \bmod n}\,I_n$ independently of $m$, and
therefore
\begin{equation}
  \widetilde U_l^n = (-1)^{l \bmod n}\,I_n.
\end{equation}
Squaring gives $\widetilde U_l^{2n} = I_n$.
To see that the order is exactly $2n$ when $l \bmod n$ is odd,
note that $\widetilde U_l^n = -I_n \neq I_n$, so no divisor of $n$
is the order. For $n < T < 2n$, write $T = n + q$ with
$1 \leq q < n$; then
\begin{equation}
  \widetilde U_l^T = \widetilde U_l^n \cdot \widetilde U_l^q
  = -\widetilde U_l^q.
\end{equation}
For $\widetilde U_l^T = I_n$ we would need $\widetilde U_l^q = -I_n$,
hence $\Sigma_{ql \bmod n} = I_n$, that is\ $ql \equiv 0 \pmod{n}$.
Since $\gcd(l \bmod n,\,n) = 1$ this forces $n \mid q$,
contradicting $1 \leq q < n$. Hence the order of $\widetilde U_l$
is exactly $2n$ for every $l$ with $l \bmod n$ odd.
This proves part~(i).

\noindent\textit{Proof of part~(ii).}
Since $d = n$, we verify directly that $\widetilde{U}^{\,n}$ and
$I_n \otimes \widetilde{P}^{n/2}$ agree on every basis vector
$f_m \otimes f_l$ of $\mathbb{C}^{n^2}$.
From part~(i),
\begin{equation}
  \widetilde{U}^{\,n}(f_m \otimes f_l)
  = \widetilde U_l^n\,(f_m \otimes f_l)
  = (-1)^{l \bmod n}\,(f_m \otimes f_l)
  = (-1)^l\,(f_m \otimes f_l),
\end{equation}
where the last equality uses $l \in \mathbb{Z}_n$ so that
$l \bmod n = l$.
Using the mixed-product property and the eigenrelation
$\widetilde{P}\,f_l = \omega^l f_l$ with $\omega = e^{2\pi i/n}$, we have
\[
  (I_n \otimes \widetilde{P}^{n/2})(f_m \otimes f_l)
  = f_m \otimes \widetilde{P}^{n/2} f_l
  = f_m \otimes \omega^{ln/2}\,f_l
  = e^{i\pi l}\,(f_m \otimes f_l)
  = (-1)^l\,(f_m \otimes f_l).
\]
Both sides produce the same scalar $(-1)^l$ on every basis vector
$f_m \otimes f_l$ of $\mathbb{C}^{n^2}$, so
\begin{equation}
  \widetilde{U}^{\,n} = I_n \otimes \widetilde{P}^{n/2},
\end{equation}
which is part~(ii).
\end{proof}
These two Lemmas together yield the PST and periodicity values of $K_n^{\circlearrowleft}$, 
as established below with example.
\begin{Theorem}
\label{thm:PST}
Let \(\pi\) be an equitable partition of a \(d\)-regular directed graph \(G\) on \(2n\) vertices, and suppose that the corresponding quotient graph is \(K_n^{\circlearrowleft}\), where \(n=2^s\) for some integer \(s\ge 1\). Then the following hold.
\noindent\textup{(i)}
At time \(k=n\), PST occurs between each basis state
\[
e_{v_i}\otimes e_{v_j}
\in
\operatorname{Im}( I_d \otimes \widetilde Q),
\qquad (v_i,v_j)\in\mathcal A(K_n^{\circlearrowleft}),
\]
representing the walker on the arc \((v_i,v_j)\), and the basis state corresponding to its antipodal arc
\(
\bigl(v_i,\;v_{j+n/2 \bmod n}\bigr).
\)
That is, for every \((v_i,v_j)\in\mathcal A(K_n^{\circlearrowleft})\),
\[
\Bigl|
\bigl(\widetilde U^{\,n}\bigr)_{
(v_i,\;v_{j+n/2 \bmod n}),
(v_i,\;v_j)}
\Bigr|^2
=1,
\]
and hence the transfer fidelity is \(\mathcal F=1\).

\noindent\textup{(ii)} 
\(
  \widetilde{U}^{\,2n} = I_{n^2},
\)
and the period of $\widetilde{U}$ is exactly $k = 2n$.
\end{Theorem}

\begin{proof}
Let the vertices of \(K_n^{\circlearrowleft}\) be labeled by
\(
V(K_n^{\circlearrowleft})
=
\{v_0,v_1,\ldots,v_{n-1}\}.
\)
Let $\widetilde{U}$ be transition matrix of
$K_n^{\circlearrowleft}$ defined by $\widetilde{U} = \widetilde  S \widetilde  R$, where
$S = \bigoplus_{k=0}^{n-1} \widetilde  P^k$ and
$R = 2\widetilde{Q}\widetilde{Q}^\top - I_{n^2}$, \(
  \widetilde{Q} = \frac{1}{\sqrt{d}}
  \begin{bmatrix} I_n \\ \widetilde{P} \\ \widetilde{P}^2 \\ \vdots \\ 
  \widetilde{P}^{d-1} \end{bmatrix}.\)
Thus the arc set of $K_n^{\circlearrowleft}$ is
\(
  \mathcal A(K_n^{\circlearrowleft})
  = \bigl\{(v_i, v_j) : i, j \in \mathbb{Z}_n\bigr\},
\)
so the state space of the walk is $\mathbb{C}^{n^2}$ with
orthonormal basis $\{e_{v_i} \otimes e_{v_j}\}_{(v_i,v_j) \in \mathcal A(K_n^{\circlearrowleft})}\in \operatorname{Im}(I_d\otimes \widetilde{Q})$. Since \(d=n\) for $K_n^{\circlearrowleft}$, we have by Lemma~\ref{lem:Uln_general},
\begin{equation}
\label{eq:Un-thm}
\widetilde U^{\,n}
=
I_n\otimes \widetilde{P}^{n/2},
\end{equation}
where \(\widetilde{P}\) is the cyclic shift matrix satisfying
\(\widetilde{P} e_{v_j}=e_{v_{j+1\bmod n}}\).
The basis vector \(e_{v_i}\otimes e_{v_j}\) represents the walker
on the arc \((v_i,v_j)\). Using the mixed-product property,
\begin{equation}
(I_n\otimes \widetilde{P}^{n/2})(e_{v_i}\otimes e_{v_j})
=
(I_n e_{v_i})\otimes(\widetilde{P}^{n/2}e_{v_j})
=
e_{v_i}\otimes e_{v_{j+n/2\bmod n}}.
\end{equation}
Hence, by \eqref{eq:Un-thm},
\begin{equation}
\widetilde U^{\,n}(e_{v_i}\otimes e_{v_j})
=
e_{v_i}\otimes e_{v_{j+n/2\bmod n}}.
\end{equation}
Therefore every arc \((v_i,v_j)\) is mapped to its antipodal arc
\(
(v_i,v_j)
\longmapsto
(v_i,v_{j+n/2\bmod n})
\)
with amplitude
\(
\bigl(\widetilde U^{\,n}\bigr)_{
(v_i,v_{j+n/2\bmod n}),
(v_i,v_j)}
=
1.
\)
Since \(\widetilde U^{\,n}\) is unitary, all other entries in the same
column are zero, and thus the transfer fidelity is
\begin{equation}
\mathcal F
=
\Bigl|
\bigl(\widetilde U^{\,n}\bigr)_{
(v_i,v_{j+n/2\bmod n}),
(v_i,v_j)}
\Bigr|^2
=
1.
\end{equation}
This proves (i).

Applying \eqref{eq:Un-thm} twice gives
\begin{equation}
\widetilde U^{\,2n}
=
(\widetilde U^{\,n})^2
=
(I_n\otimes C^{n/2})^2
=
I_n\otimes C^n
=
I_{n^2}.
\end{equation}
Thus the period divides \(2n\). Since \(n\ge2\),
\(
\widetilde{P}^{n/2}e_{v_j}
=
e_{v_{j+n/2}}
\neq
e_{v_j},
\)
so \(\widetilde{P}^{n/2}\neq I_n\), and hence
\begin{equation}
\widetilde U^{\,n}
=
I_n\otimes \widetilde{P}^{n/2}
\neq
I_{n^2}.
\end{equation}
Therefore the period does not divide \(n\). Since it divides \(2n\) but
not \(n\), the period is exactly \(2n\). This proves (ii).
\end{proof}

\begin{ex}{ $K_4^{\circlearrowleft}$ with Loops, $n = 4 = 2^2=d$}. The cyclic shift matrix $\widetilde{P} \in M_{4\times4}(\mathbb{R})$ is
defined by $(\widetilde{P})_{ij} = \delta_{i,\,j+1 \bmod 4}$
\[
  \widetilde{P}=
  \begin{bmatrix}
    0 & 1 & 0 & 0 \\
    0 & 0 & 1 & 0 \\
    0 & 0 & 0 & 1 \\
    1 & 0 & 0 & 0
  \end{bmatrix}.
\]
Its powers are
\[
  \widetilde{P}^0 = I_4, \qquad
  \widetilde{P}^1 = \widetilde P, \qquad
  \widetilde{P}^2 =
  \begin{bmatrix}
    0&0&1&0\\ 0&0&0&1\\ 1&0&0&0\\ 0&1&0&0
  \end{bmatrix}, \qquad
  \widetilde{P}^3 =
  \begin{bmatrix}
    0&0&0&1\\ 1&0&0&0\\ 0&1&0&0\\ 0&0&1&0
  \end{bmatrix}.
\]
\[
  \widetilde{Q}
  = \frac{1}{\sqrt{4}}
    \begin{bmatrix} I_4 \\ \widetilde{P} \\ \widetilde{P}^2 \\ \widetilde{P}^3 \end{bmatrix}
  = \frac{1}{2}
  \scalebox{0.5}{$\begin{bmatrix}
      1&0&0&0\\ 0&1&0&0\\ 0&0&1&0\\ 0&0&0&1\\
      0&1&0&0\\ 0&0&1&0\\ 0&0&0&1\\ 1&0&0&0\\
      0&0&1&0\\ 0&0&0&1\\ 1&0&0&0\\ 0&1&0&0\\
      0&0&0&1\\ 1&0&0&0\\ 0&1&0&0\\ 0&0&1&0
    \end{bmatrix}$}
  \in M_{16\times4}(\mathbb{R}).
\]

Each column of $\widetilde{Q}$ contains exactly four entries equal to
$\tfrac{1}{2}$ and is otherwise zero, so every column has unit norm.
Distinct columns have disjoint support, hence are orthogonal. Therefore
$\widetilde{Q}^{\top}\widetilde{Q} = I_4$.
Set $P = \widetilde{Q}\widetilde{Q}^{\top} \in M_{16\times16}(\mathbb{R})$,
the orthogonal projection onto the column space of $\widetilde{Q}$, and
\[
 \widetilde R = 2P - I_{16}.
\]
Hence $\widetilde R$ is a symmetric orthogonal matrix (a Householder-type reflection).

\[
  \widetilde S = \bigoplus_{k=0}^{3} \widetilde{P}^k
  =
  \begin{bmatrix}
    I_4 & 0   & 0   & 0   \\
    0   & \widetilde{P}   & 0   & 0   \\
    0   & 0   & \widetilde{P}^2 & 0   \\
    0   & 0   & 0   & \widetilde{P}^3
  \end{bmatrix}
  \in M_{16\times16}(\mathbb{R}).
\]
Since each $\widetilde P^k$ is a permutation matrix, $\widetilde S$ is also a permutation matrix and
hence orthogonal $ \widetilde S^{\top} \widetilde S = I_{16}$.
We have 
\(
  \widetilde{U} = \widetilde S \widetilde R \in M_{16\times16}(\mathbb{R}).
\)
Since $\widetilde S$ and $\widetilde R$ are both orthogonal,
$\widetilde{U}^{\top}\widetilde{U} = \widetilde R^{\top}\widetilde    S^{\top} \widetilde S \widetilde R =  \widetilde R^{\top} \widetilde R = I_{16}$,
so $\widetilde{U}$ is orthogonal and all its eigenvalues lie on the unit circle. Let $\omega = e^{2\pi i/4} = i$ and
$f_l = \tfrac{1}{2}[1, \omega^l, \omega^{2l}, \omega^{3l}]^{\top}$ for
$l \in \mathbb{Z}_4$.  By Lemma~\ref{eigen_values_}, $\widetilde{U}$ preserves each subspace
\(
  V_l = \operatorname{span}\{f_m \otimes f_l : m \in \mathbb{Z}_4\},
  \qquad l \in \mathbb{Z}_4,
\)
and its restriction to $V_l$ (in the basis $\{f_m \otimes f_l\}$) is
\[
  \widetilde U_l = \Sigma_l D_l,
\]
where $\Sigma_l$ is the cyclic shift by $l$ and
$D_l = \operatorname{diag}(d_0^{(l)}, \ldots, d_3^{(l)})$ with
\[
  d_m^{(l)} =
  \begin{cases}
    +1 & m = l, \\
    -1 & m \neq l.
  \end{cases}
\]

Explicitly
\[
  \Sigma_0 = I_4, \qquad
  \Sigma_1 =
  \begin{bmatrix}0&0&0&1\\1&0&0&0\\0&1&0&0\\0&0&1&0\end{bmatrix},
  \qquad
  \Sigma_2 =
  \begin{bmatrix}0&0&1&0\\0&0&0&1\\1&0&0&0\\0&1&0&0\end{bmatrix},
  \qquad
  \Sigma_3 =
  \begin{bmatrix}0&1&0&0\\0&0&1&0\\0&0&0&1\\1&0&0&0\end{bmatrix},
\]
\[
  D_0 = \operatorname{diag}(+1,-1,-1,-1), \quad
  D_1 = \operatorname{diag}(-1,+1,-1,-1), \quad
  D_2 = \operatorname{diag}(-1,-1,+1,-1),\]
  \[
  D_3 = \operatorname{diag}(-1,-1,-1,+1).
\]

From Lemma~\ref{eigen_values_}, Equation~\ref{det}, we have for each $l \in \mathbb{Z}_4$
\begin{equation}\label{eq:charpoly}
  \det(\lambda I - \widetilde U_l)
  = \prod_{m=0}^{3}(\lambda + \omega^{lm})
    - \frac{1}{2}\sum_{m=0}^{3}
      \omega^{-ml}
      \prod_{\substack{j=0\\j\neq m}}^{3}(\lambda + \omega^{lj}).
\end{equation}

\textit{Case $l = 0$.}
All $\omega^{0\cdot m} = 1$, so $\omega^{-m\cdot 0} = 1$ and
\[
  \prod_{m=0}^{3}(\lambda+1) = (\lambda+1)^4,
  \qquad
  \frac{1}{2}\sum_{m=0}^{3}(\lambda+1)^3
  = \frac{1}{2}\cdot 4(\lambda+1)^3 = 2(\lambda+1)^3.
\]
Hence
\[
  \det(\lambda I - \widetilde U_0) = (\lambda+1)^4 - 2(\lambda+1)^3
  = (\lambda+1)^3\bigl[(\lambda+1)-2\bigr]
  = (\lambda+1)^3(\lambda-1).
\]
\textit{Eigenvalues of $ \widetilde U_0$:} $\{-1,-1,-1,+1\}$.
\textit{Cases $l = 1$ and $l = 3$ (odd $l$).}
Since $l$ is odd and $n = 2^2$, Lemma~\ref{lem:Uln_general} gives $\widetilde U_l^4 = -I_4$.  If $\lambda$
is an eigenvalue with eigenvector $v$, then
\(
  -v = \widetilde U_l^4 v = \lambda^4 v,
\)
so $\lambda^4 = -1$.  The four solutions are the primitive $8$th roots of
unity, all distinct.  The $4\times 4$ matrix $\widetilde U_l$ therefore has characteristic
polynomial
\[
  \det(\lambda I - \widetilde U_l) = \lambda^4 + 1.
\]
\textit{Eigenvalues of $\widetilde U_1$ and $ \widetilde U_3$:}
$e^{\pm i\pi/4},\, e^{\pm 3i\pi/4}$,
that is\ $\dfrac{\pm 1 \pm i}{\sqrt{2}}$.

\textit{Case $l = 2$.}
Here $\omega^2 = -1$, so $\omega^{2m}$ cycles as $\{1,-1,1,-1\}$.  We apply
\eqref{eq:charpoly} directly.

\noindent\textit{First term.}
\[
  \prod_{m=0}^{3}(\lambda+\omega^{2m})
  = (\lambda+1)^2(\lambda-1)^2
  = (\lambda^2-1)^2.
\]

\noindent\textit{Second term.}
The weights $\omega^{-2m} = (-1)^{-m}$ cycle as $\{1,-1,1,-1\}$.
The four omitted products are:
\[
\begin{array}{rcl}
  m=0: & \omega^{0} = 1, &
    \displaystyle\prod_{j\neq0}(\lambda+\omega^{2j})
    = (\lambda-1)(\lambda+1)(\lambda-1) = (\lambda^2-1)(\lambda-1),\\
  m=1: & \omega^{-2} = -1, &
    \displaystyle\prod_{j\neq1}(\lambda+\omega^{2j})
    = (\lambda+1)(\lambda+1)(\lambda-1) = (\lambda+1)^2(\lambda-1),\\[6pt]
  m=2: & \omega^{-4} = 1, &
    \displaystyle\prod_{j\neq2}(\lambda+\omega^{2j})
    = (\lambda+1)(\lambda-1)(\lambda-1) = (\lambda^2-1)(\lambda-1),\\[6pt]
  m=3: & \omega^{-6} = -1, &
    \displaystyle\prod_{j\neq3}(\lambda+\omega^{2j})
    = (\lambda+1)(\lambda+1)(\lambda-1) = (\lambda+1)^2(\lambda-1).
\end{array}
\]
Collecting with their weights and the factor $\tfrac{1}{2}$
\begin{align*}
  &\frac{1}{2}\Bigl[
      1\cdot(\lambda^2-1)(\lambda-1)
    + (-1)\cdot(\lambda+1)^2(\lambda-1)
    + 1\cdot(\lambda^2-1)(\lambda-1)
    + (-1)\cdot(\lambda+1)^2(\lambda-1)
    \Bigr]\\[4pt]
  &= (\lambda-1)\Bigl[(\lambda^2-1) - (\lambda+1)^2\Bigr]
   = -2(\lambda^2-1).
\end{align*}
\begin{align*}
  \det(\lambda I - \widetilde U_2)
  &= (\lambda^2-1)^2 + 2(\lambda^2-1)
   = (\lambda^2-1)(\lambda^2+1)
   = \lambda^4 - 1.
\end{align*}
\textit{Eigenvalues of $\widetilde U_2$:} $\{+1,-1,+i,-i\}$.

Collecting eigenvalues from all four blocks:

\begin{center}
\begin{tabular}{ccl}
\hline
$l$ & Characteristic polynomial & Eigenvalues (with multiplicity)\\
\hline
$0$ & $(\lambda+1)^3(\lambda-1)$
    & $-1$ (mult.\ 3),\; $+1$ (mult.\ 1)\\
$1$ & $\lambda^4+1$
    & $e^{\pm i\pi/4},\; e^{\pm 3i\pi/4}$ (each mult.\ 1)\\
$2$ & $\lambda^4-1 = (\lambda^2-1)(\lambda^2+1)$
    & $+1,\;-1,\;+i,\;-i$ (each mult.\ 1)\\
$3$ & $\lambda^4+1$
    & $e^{\pm i\pi/4},\; e^{\pm 3i\pi/4}$ (each mult.\ 1)\\
\hline
\end{tabular}
\end{center}

Summing multiplicities across all blocks gives
\[
  \operatorname{spec}(\widetilde{U})
  =
  \left\{
    +1^{(2)},\;
    -1^{(4)},\;
    (+i)^{(1)},\;
    (-i)^{(1)},\;
    \Bigl(\tfrac{1+i}{\sqrt{2}}\Bigr)^{(2)},\;
    \Bigl(\tfrac{1-i}{\sqrt{2}}\Bigr)^{(2)},\;
    \Bigl(\tfrac{-1+i}{\sqrt{2}}\Bigr)^{(2)},\;
    \Bigl(\tfrac{-1-i}{\sqrt{2}}\Bigr)^{(2)}
  \right\},
\]
accounting for all $16$ eigenvalues, all on the unit circle.

\begin{itemize}
  \item $l = 1, 3$ (odd): $\widetilde U_l^4 = -I_4$, so $\widetilde U_l^8 = I_4$.
  \item $l = 0$: eigenvalues $\pm 1$ give $\widetilde U_0^2 = I_4$, so $\widetilde U_0^8 = I_4$.
  \item $l = 2$: eigenvalues $+1,-1,+i,-i$ give $\widetilde U_2^4 = I_4$, so $\widetilde U_2^8 = I_4$.
\end{itemize}
Hence $\widetilde{U}^8 = \bigoplus_{l=0}^{3} \widetilde U_l^8 = I_{16}$. The blocks $l = 1,3$ have primitive $8$th roots of unity as eigenvalues,
forcing $8 \mid T$. Therefore,
\(
  {T = 2n = 8.}
\)
By Theorem~\ref{thm:PST}(i), PST occurs at time $t = 4$
via the operator identity
\[
  \widetilde{U}^{\,4} = I_4 \otimes  \widetilde P^2.
\]
We verify this directly from the $16\times16$ matrix $\widetilde{U}^4$,
whose rows and columns are indexed by arcs $i\to j$ with
$i,j\in\mathbb{Z}_4$. The matrix has a single $+1$ in each row and
column and is zero elsewhere:

\begin{center}
\renewcommand{\arraystretch}{0.5}
\resizebox{\textwidth}{!}{%
\begin{tabular}{c|cccccccccccccccc}
 & \scriptsize$0{\to}0$ & \scriptsize$0{\to}1$ & \scriptsize$0{\to}2$
 & \scriptsize$0{\to}3$ & \scriptsize$1{\to}0$ & \scriptsize$1{\to}1$
 & \scriptsize$1{\to}2$ & \scriptsize$1{\to}3$ & \scriptsize$2{\to}0$
 & \scriptsize$2{\to}1$ & \scriptsize$2{\to}2$ & \scriptsize$2{\to}3$
 & \scriptsize$3{\to}0$ & \scriptsize$3{\to}1$ & \scriptsize$3{\to}2$
 & \scriptsize$3{\to}3$ \\
\hline
\scriptsize$0{\to}0$ & $\cdot$ & $\cdot$ & $+1$ & $\cdot$ & $\cdot$ & $\cdot$ & $\cdot$ & $\cdot$ & $\cdot$ & $\cdot$ & $\cdot$ & $\cdot$ & $\cdot$ & $\cdot$ & $\cdot$ & $\cdot$ \\
\scriptsize$0{\to}1$ & $\cdot$ & $\cdot$ & $\cdot$ & $+1$ & $\cdot$ & $\cdot$ & $\cdot$ & $\cdot$ & $\cdot$ & $\cdot$ & $\cdot$ & $\cdot$ & $\cdot$ & $\cdot$ & $\cdot$ & $\cdot$ \\
\scriptsize$0{\to}2$ & $+1$ & $\cdot$ & $\cdot$ & $\cdot$ & $\cdot$ & $\cdot$ & $\cdot$ & $\cdot$ & $\cdot$ & $\cdot$ & $\cdot$ & $\cdot$ & $\cdot$ & $\cdot$ & $\cdot$ & $\cdot$ \\
\scriptsize$0{\to}3$ & $\cdot$ & $+1$ & $\cdot$ & $\cdot$ & $\cdot$ & $\cdot$ & $\cdot$ & $\cdot$ & $\cdot$ & $\cdot$ & $\cdot$ & $\cdot$ & $\cdot$ & $\cdot$ & $\cdot$ & $\cdot$ \\
\scriptsize$1{\to}0$ & $\cdot$ & $\cdot$ & $\cdot$ & $\cdot$ & $\cdot$ & $\cdot$ & $+1$ & $\cdot$ & $\cdot$ & $\cdot$ & $\cdot$ & $\cdot$ & $\cdot$ & $\cdot$ & $\cdot$ & $\cdot$ \\
\scriptsize$1{\to}1$ & $\cdot$ & $\cdot$ & $\cdot$ & $\cdot$ & $\cdot$ & $\cdot$ & $\cdot$ & $+1$ & $\cdot$ & $\cdot$ & $\cdot$ & $\cdot$ & $\cdot$ & $\cdot$ & $\cdot$ & $\cdot$ \\
\scriptsize$1{\to}2$ & $\cdot$ & $\cdot$ & $\cdot$ & $\cdot$ & $+1$ & $\cdot$ & $\cdot$ & $\cdot$ & $\cdot$ & $\cdot$ & $\cdot$ & $\cdot$ & $\cdot$ & $\cdot$ & $\cdot$ & $\cdot$ \\
\scriptsize$1{\to}3$ & $\cdot$ & $\cdot$ & $\cdot$ & $\cdot$ & $\cdot$ & $+1$ & $\cdot$ & $\cdot$ & $\cdot$ & $\cdot$ & $\cdot$ & $\cdot$ & $\cdot$ & $\cdot$ & $\cdot$ & $\cdot$ \\
\scriptsize$2{\to}0$ & $\cdot$ & $\cdot$ & $\cdot$ & $\cdot$ & $\cdot$ & $\cdot$ & $\cdot$ & $\cdot$ & $\cdot$ & $\cdot$ & $+1$ & $\cdot$ & $\cdot$ & $\cdot$ & $\cdot$ & $\cdot$ \\
\scriptsize$2{\to}1$ & $\cdot$ & $\cdot$ & $\cdot$ & $\cdot$ & $\cdot$ & $\cdot$ & $\cdot$ & $\cdot$ & $\cdot$ & $\cdot$ & $\cdot$ & $+1$ & $\cdot$ & $\cdot$ & $\cdot$ & $\cdot$ \\
\scriptsize$2{\to}2$ & $\cdot$ & $\cdot$ & $\cdot$ & $\cdot$ & $\cdot$ & $\cdot$ & $\cdot$ & $\cdot$ & $+1$ & $\cdot$ & $\cdot$ & $\cdot$ & $\cdot$ & $\cdot$ & $\cdot$ & $\cdot$ \\
\scriptsize$2{\to}3$ & $\cdot$ & $\cdot$ & $\cdot$ & $\cdot$ & $\cdot$ & $\cdot$ & $\cdot$ & $\cdot$ & $\cdot$ & $+1$ & $\cdot$ & $\cdot$ & $\cdot$ & $\cdot$ & $\cdot$ & $\cdot$ \\
\scriptsize$3{\to}0$ & $\cdot$ & $\cdot$ & $\cdot$ & $\cdot$ & $\cdot$ & $\cdot$ & $\cdot$ & $\cdot$ & $\cdot$ & $\cdot$ & $\cdot$ & $\cdot$ & $\cdot$ & $\cdot$ & $+1$ & $\cdot$ \\
\scriptsize$3{\to}1$ & $\cdot$ & $\cdot$ & $\cdot$ & $\cdot$ & $\cdot$ & $\cdot$ & $\cdot$ & $\cdot$ & $\cdot$ & $\cdot$ & $\cdot$ & $\cdot$ & $\cdot$ & $\cdot$ & $\cdot$ & $+1$ \\
\scriptsize$3{\to}2$ & $\cdot$ & $\cdot$ & $\cdot$ & $\cdot$ & $\cdot$ & $\cdot$ & $\cdot$ & $\cdot$ & $\cdot$ & $\cdot$ & $\cdot$ & $\cdot$ & $+1$ & $\cdot$ & $\cdot$ & $\cdot$ \\
\scriptsize$3{\to}3$ & $\cdot$ & $\cdot$ & $\cdot$ & $\cdot$ & $\cdot$ & $\cdot$ & $\cdot$ & $\cdot$ & $\cdot$ & $\cdot$ & $\cdot$ & $\cdot$ & $\cdot$ & $+1$ & $\cdot$ & $\cdot$ \\
\end{tabular}%
}
\end{center}
Each row $i\to j$ has its unique $+1$ in column $i\to j+2\bmod 4$,
confirming $\widetilde{U}^4 = I_4\otimes\widetilde  P^2$ entry by entry. The image of $I_4\otimes\widetilde{Q}$ consists of all vectors of the form
\[
\operatorname{Im}(I_4\otimes\widetilde{Q})
= \left\{
\frac{1}{2}
\begin{bmatrix} \widetilde{Q}z_0 \\ \widetilde{Q}z_1 \\ \widetilde{Q}z_2 \\ \widetilde{Q}z_3 \end{bmatrix}
: z_0,z_1,z_2,z_3\in\mathbb{C}^4
\right\}.
\]
Taking $z_i = e_{v_j}$ for each pair $(i,j)\in\mathbb{Z}_4\times\mathbb{Z}_4$
and $z_k=0$ for $k\neq i$, the 16 distinguished basis states in
$\operatorname{Im}(I_4\otimes\widetilde{Q})$ are
\[
x_{ij}
= (I_4\otimes\widetilde{Q})(e_{v_i}\otimes e_{v_j})
= e_{v_i}\otimes\widetilde{Q}e_{v_j}
= \frac{1}{2}\,e_{v_i}\otimes
\begin{bmatrix}
e_{v_j}\\e_{v_{j+1\bmod4}}\\e_{v_{j+2\bmod4}}\\e_{v_{j+3\bmod4}}
\end{bmatrix},
\qquad i,j\in\mathbb{Z}_4.
\]
Explicitly, for tail block $i=0$:
\[
x_{0,0} = \frac{1}{2}
\begin{bmatrix}e_{v_0}\\e_{v_1}\\e_{v_2}\\e_{v_3}\\0\\\vdots\\0\end{bmatrix},\quad
x_{0,1} = \frac{1}{2}
\begin{bmatrix}e_{v_1}\\e_{v_2}\\e_{v_3}\\e_{v_0}\\0\\\vdots\\0\end{bmatrix},\quad
x_{0,2} = \frac{1}{2}
\begin{bmatrix}e_{v_2}\\e_{v_3}\\e_{v_0}\\e_{v_1}\\0\\\vdots\\0\end{bmatrix},\quad
x_{0,3} = \frac{1}{2}
\begin{bmatrix}e_{v_3}\\e_{v_0}\\e_{v_1}\\e_{v_2}\\0\\\vdots\\0\end{bmatrix},
\]
and analogously for tail blocks $i=1,2,3$, with the nonzero entries
shifted to the corresponding block. Applying $\widetilde{U}^{4} = I_4\otimes\widetilde{P}^{2}$.
Since $I_4\otimes\widetilde{P}^2$ acts on each tail-block $i$ independently
by $\widetilde{P}^2$, and $x_{ij}$ is supported only in tail-block $i$:
\[
\widetilde{U}^{\,4}\,x_{ij}
= (I_4\otimes\widetilde{P}^2)(e_{v_i}\otimes\widetilde{Q}e_{v_j})
= e_{v_i}\otimes\widetilde{P}^2\widetilde{Q}e_{v_j}
= \frac{1}{2}\,e_{v_i}\otimes
\begin{bmatrix}
e_{v_{j+2\bmod4}}\\e_{v_{j+3\bmod4}}\\e_{v_{j+0\bmod4}}\\e_{v_{j+1\bmod4}}
\end{bmatrix}
= e_{v_i}\otimes\widetilde{Q}e_{v_{j+2\bmod4}}
= x_{i,\,j+2\bmod4}.
\]
Hence $\widetilde{U}^{\,4}\,x_{ij} = x_{i,\,j+2\bmod 4}$, with fidelity
\[
\mathcal{F}
= \bigl|\langle x_{i,\,j+2\bmod4},\,\widetilde{U}^{\,4}\,x_{ij}\rangle\bigr|^2
= 1.
\]
The 16 PST pairs, grouped by tail block, are:

\begin{center}
\begin{tabular}{ccc}
\hline
Tail block & Initial state $x$ & Target state $y$ \\
\hline
$i=0$ & $x_{0,0}$ & $x_{0,2}$ \\
$i=0$ & $x_{0,1}$ & $x_{0,3}$ \\
$i=0$ & $x_{0,2}$ & $x_{0,0}$ \\
$i=0$ & $x_{0,3}$ & $x_{0,1}$ \\
$i=1$ & $x_{1,0}$ & $x_{1,2}$ \\
$i=1$ & $x_{1,1}$ & $x_{1,3}$ \\
$i=1$ & $x_{1,2}$ & $x_{1,0}$ \\
$i=1$ & $x_{1,3}$ & $x_{1,1}$ \\
$i=2$ & $x_{2,0}$ & $x_{2,2}$ \\
$i=2$ & $x_{2,1}$ & $x_{2,3}$ \\
$i=2$ & $x_{2,2}$ & $x_{2,0}$ \\
$i=2$ & $x_{2,3}$ & $x_{2,1}$ \\
$i=3$ & $x_{3,0}$ & $x_{3,2}$ \\
$i=3$ & $x_{3,1}$ & $x_{3,3}$ \\
$i=3$ & $x_{3,2}$ & $x_{3,0}$ \\
$i=3$ & $x_{3,3}$ & $x_{3,1}$ \\
\hline
\end{tabular}
\end{center}
Applying the transfer twice confirms the period
\[
\widetilde{U}^{\,8}\,x_{ij}
= \widetilde{U}^{\,4}\,x_{i,\,j+2\bmod4}
= x_{i,\,j+4\bmod4}
= x_{ij},
\qquad
\widetilde{U}^{\,8} = I_{16}.
\]
Since $\widetilde{P}^2\neq I_4$ (as $e_{v_0}\mapsto e_{v_2}\neq e_{v_0}$),
we have $\widetilde{U}^{\,4} = I_4\otimes\widetilde{P}^2\neq I_{16}$,
so the period does not divide~$4$. Therefore the period of $\widetilde{U}$
is exactly $2n = 8$.

\end{ex}
Thus the derived model exhibits exact PST on
\(K_{n}^{\circlearrowleft}\) whenever \(n=d\) is a power of \(2\).
Specifically, PST occurs after \(n\) steps and the evolution is periodic
with period \(2n\). Thus, for every complete graph with loops on
\(n=2^s\) vertices, the corresponding quotient quantum walk exhibits
PST together with a well-defined periodic structure.

\section{Construction of 
\texorpdfstring{$LD(K_{n,n}^{\rightleftharpoons})$}{LD(Knn)} 
with PST for 
\texorpdfstring{$n = 2^s$}{n = 2\^{}s}, 
\texorpdfstring{$s \geq 0$}{s >= 0}}
\label{section 9}
\begin{figure}[h]
\centering

\begin{subfigure}[b]{0.62\textwidth}
\centering
\begin{tikzpicture}[
    scale=0.8,
    transform shape,
    every node/.style={
        circle, draw, fill=white,
        inner sep=1.5pt,
        minimum size=18pt,
        font=\small
    },
    Gedge/.style={thick, gray, <->, >=stealth},
    col1/.style={green!70!black, line width=1.6pt, ->, >=stealth},
    col2/.style={orange,         line width=1.6pt, ->, >=stealth},
    col3/.style={blue,           line width=1.6pt, ->, >=stealth},
    col4/.style={violet,         line width=1.6pt, ->, >=stealth},
]

\node[
    draw=green!60!black,
    dashed,
    rectangle,
    rounded corners=8pt,
    minimum width=2.2cm,
    minimum height=5.0cm,
    fill=green!8,
    label=above:{\small $C_1=\{a_1,b_1\}$}
] at (0,1.8) {};

\node[
    draw=blue!60!black,
    dashed,
    rectangle,
    rounded corners=8pt,
    minimum width=2.2cm,
    minimum height=5.0cm,
    fill=blue!8,
    label=below:{\small $C_3=\{a_2,b_2\}$}
] at (0,-3.0) {};

\node[
    draw=orange!80!black,
    dashed,
    rectangle,
    rounded corners=8pt,
    minimum width=2.2cm,
    minimum height=5.0cm,
    fill=orange!8,
    label=above:{\small $C_2=\{a_3,b_3\}$}
] at (6.5,1.8) {};

\node[
    draw=violet!80!black,
    dashed,
    rectangle,
    rounded corners=8pt,
    minimum width=2.2cm,
    minimum height=5.0cm,
    fill=violet!5,
    label=below:{\small $C_4=\{a_4,b_4\}$}
] at (6.5,-3.0) {};

\node (a1) at (0,3.2) {$a_1$};
\node (b1) at (0,0.4) {$b_1$};
\node (a2) at (0,-1.8) {$a_2$};
\node (b2) at (0,-4.2) {$b_2$};

\node (a3) at (6.5,3.2) {$a_3$};
\node (b3) at (6.5,0) {$b_3$};
\node (a4) at (6.5,-1.8) {$a_4$};
\node (b4) at (6.5,-4.2) {$b_4$};

\foreach \j in {1,2,3,4}{
    \draw[col1,bend left=8] (a1) to (b\j);
    \draw[col1,bend left=8] (b1) to (a\j);
}

\foreach \j in {1,2,3,4}{
    \draw[col3,bend left=10] (a2) to (b\j);
    \draw[col3,bend left=10] (b2) to (a\j);
}

\foreach \j in {1,2,3,4}{
    \draw[col2,bend left=12] (a3) to (b\j);
    \draw[col2,bend left=12] (b3) to (a\j);
}

\foreach \j in {1,2,3,4}{
    \draw[col4,bend left=14] (a4) to (b\j);
    \draw[col4,bend left=14] (b4) to (a\j);
}

\end{tikzpicture}
\caption{Vertex partition of \(K^{\rightleftharpoons}_{4,4}\)}
\label{fig:partition-colored}
\end{subfigure}
\hfill
\begin{subfigure}[b]{0.34\textwidth}
\centering

\begin{tikzpicture}[
    scale=1.0,
    transform shape,
    cnode/.style={
        circle,
        draw,
        fill=white,
        inner sep=2pt,
        minimum size=22pt,
        font=\small\bfseries
    },
    Kloop/.style={thick, gray, ->, >=stealth},
    Kedge/.style={thick, gray, <->, >=stealth},
    col1/.style={green!70!black, line width=1.6pt, ->, >=stealth},
    col2/.style={orange,         line width=1.6pt, ->, >=stealth},
    col3/.style={blue,           line width=1.6pt, ->, >=stealth},
    col4/.style={violet,         line width=1.6pt, ->, >=stealth},
]

\node[cnode, fill=green!8,  draw=green!60!black]  (C1) at (0,2.8) {$C_1$};
\node[cnode, fill=orange!8, draw=orange!80!black] (C2) at (2.8,2.8) {$C_2$};
\node[cnode, fill=blue!8,   draw=blue!60!black]   (C3) at (0,0) {$C_3$};
\node[cnode, fill=violet!5, draw=violet!80!black] (C4) at (2.8,0) {$C_4$};

\draw[Kloop] (C1) edge[loop left,looseness=6] (C1);
\draw[Kloop] (C2) edge[loop right,looseness=6] (C2);
\draw[Kloop] (C3) edge[loop left,looseness=6] (C3);
\draw[Kloop] (C4) edge[loop right,looseness=6] (C4);

\draw[col1,bend left=18]  (C1) to (C2);
\draw[col2,bend left=18]  (C2) to (C1);

\draw[col1,bend right=18] (C1) to (C3);
\draw[col3,bend right=18] (C3) to (C1);

\draw[col1,bend left=12]  (C1) to (C4);
\draw[col4,bend left=12]  (C4) to (C1);

\draw[col2,bend right=12] (C2) to (C3);
\draw[col3,bend right=12] (C3) to (C2);

\draw[col2,bend left=18]  (C2) to (C4);
\draw[col4,bend left=18]  (C4) to (C2);

\draw[col3,bend left=18]  (C3) to (C4);
\draw[col4,bend left=18]  (C4) to (C3);

\end{tikzpicture}

\caption{Quotient graph  $K^{\rightleftharpoons}_{4,4}/\pi\cong K_4^\circlearrowleft$ .}
\label{fig:K4-quotient}
\end{subfigure}

\caption{
(a) The vertex-equitable partition $\pi = \{C_1, C_2, C_3, C_4\}$ 
of $K^{\rightleftharpoons}_{4,4}$ with colored inter-cell arcs. 
(b) The corresponding quotient digraph 
$K^{\rightleftharpoons}_{4,4}/\pi \cong K_4^{\circlearrowleft}$, 
the complete digraph on four vertices with a loop at each cell.
}
\label{fig:partition-full}
\end{figure}
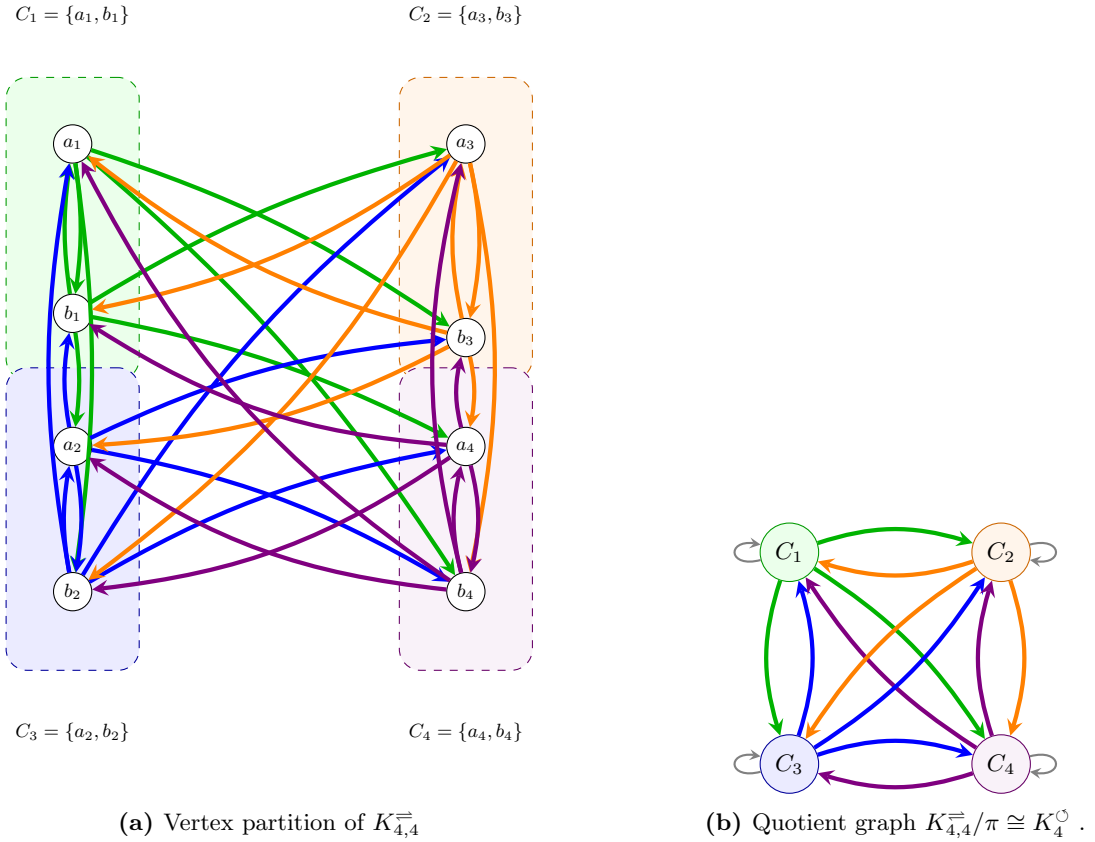

In this section, we consider the line digraph of the complete bipartite
digraph $LD(K^{\rightleftharpoons}_{n,n})$, where $n = d$ holds for
this graph. We perform a vertex-equitable partition $\pi$ of
$K^{\rightleftharpoons}_{n,n}$, and show that
\begin{equation}
    K^{\rightleftharpoons}_{n,n}/\pi \cong K^{\circlearrowleft}_{n},
\end{equation}
that is, the complete bipartite digraph with loops at all vertices. The
relation between $LD(K^{\rightleftharpoons}_{2n,2n})/\tau$ and
$LD(K^{\rightleftharpoons}_{n,n}/\pi)$ will be further established in
subsequent theorems with illustrative examples, where $\tau$ is an
equitable arc partition of $K^{\rightleftharpoons}_{n,n}$. Moreover,
the PST in $LD(K^{\rightleftharpoons}_{n,n})$
 also showed in the subsequent results.

\begin{lemma}\label{Bipartite Graph}
Let $K^{\rightleftharpoons}_{n,n}$ be the complete bipartite 
digraph with bipartition $A=\{a_1,\dots,a_n\}$ and $B=\{b_1,\dots,b_n\}$ and 
arc set
\[
E=\{(a_i,b_j)\}\cup\{(b_i,a_j)\},\quad 1\le i,j\le n,
\]
so that $|E|=2n^2$. For each pair $(i,j)\in[n]^2$ define the cell
\[
C_{ij}=\{(a_i,b_j),(b_i,a_j)\}\subset E.
\]
Then $\tau=\{C_{ij}:1\le i,j\le n\}$ is an equitable partition of the arc set 
of $LD(K^{\rightleftharpoons}_{n,n})$ into $n^2$ cells of size $2$, with quotient 
matrix $B\in\{0,1\}^{n^2\times n^2}$ whose $\big((i,j),(k,l)\big)$-entry is
\[
B_{(i,j),(k,l)}=\begin{cases}1 & \text{if }j=k,\\0 & \text{if }j\neq k.\end{cases}
\] 
Furthermore,
\[
LD(K^{\rightleftharpoons}_{n,n})/\tau\;\cong\; LD(K^{\circlearrowleft}_n).
\]
\end{lemma}

\begin{proof}

Every arc of $E$ is uniquely of the form $(a_i,b_j)$ or $(b_i,a_j)$ for some 
$(i,j)\in[n]^2$, so each arc belongs to exactly one cell $C_{ij}$. The cells 
are pairwise disjoint and cover $E$, so $\tau$ partitions $E$ into $n^2$ cells 
each of size $2$. For each cell $C_{ij}$
\begin{equation}\label{tail head}
\mathrm{tail}(C_{ij})=\{a_i,b_i\},\quad \mathrm{head}(C_{ij})=\{a_j,b_j\}.
\end{equation}

By definition of the line digraph, $C_{ij}\to C_{kl}$ in 
$LD(K^{\rightleftharpoons}_{n,n})$ if and only if
\[
\mathrm{head}(C_{ij})\cap\mathrm{tail}(C_{kl})\neq\emptyset,
\]
that is, $\{a_j,b_j\}\cap\{a_k,b_k\}\neq\emptyset$, which holds if and only 
if $j=k$. Hence out neighbours of 
\begin{equation}\label{head tail}
(C_{ij})=\{C_{jl}:1\le l\le n\},
\end{equation}
and every cell has out-degree exactly $n$. We must show that for every pair of cells $C_{ij}$ and $C_{kl}$, every arc 
$e\in C_{ij}$ points to the same number of arcs in $C_{kl}$, and that this 
number equals $B_{(i,j),(k,l)}$.

\textit{Case 1: $j\neq k$.} 
By Equation \ref{head tail}, no arc of $C_{ij}$ points to any arc of $C_{kl}$, so the count 
is $0$.

\textit{Case 2: $j=k$.} 
Take any arc $e\in C_{ij}$. From Equation ~\ref{tail head}, $e$ has a unique head vertex 
$v\in\mathrm{head}(C_{ij})=\{a_j,b_j\}$, namely:
\[
v=\begin{cases}b_j & \text{if }e=(a_i,b_j),\\ 
               a_j & \text{if }e=(b_i,a_j).\end{cases}
\]
Now consider cell $C_{jl}$ (since $k=j$). Its two arcs are $(a_j,b_l)$ and 
$(b_j,a_l)$, with tails $a_j$ and $b_j$ respectively. Exactly one of these 
tails equals $v$:
\[
\begin{cases}
(a_j,b_l)\text{ has tail }a_j=v & \text{if }e=(b_i,a_j),\\ 
(b_j,a_l)\text{ has tail }b_j=v & \text{if }e=(a_i,b_j).
\end{cases}
\]
Hence exactly $1$ arc in $C_{jl}=C_{kl}$ follows $e$, and this count is 
independent of which arc $e\in C_{ij}$ we chose and independent of $i$ 
and $l$. Combining both cases, the number of arcs in $C_{kl}$ that any arc 
$e\in C_{ij}$ points to is
\[
B_{(i,j),(k,l)}=\begin{cases}1 & \text{if }j=k,\\0 & \text{if }j\neq k,\end{cases}
\]
which depends only on $j$ and $k$, not on $i$, $l$, or the choice of 
$e\in C_{ij}$. Therefore $\tau$ is an equitable partition with quotient 
matrix $B$ as stated.
If $Q_\tau$ is the normalized characteristic matrix of $\tau$, then
\begin{equation}
A_{LD(K^{\rightleftharpoons}_{n,n})}\,Q_\tau=Q_\tau\,B,
\end{equation}
where $A_{LD(K^{\rightleftharpoons}_{n,n})}$ is the adjacency matrix of $LD(K^{\rightleftharpoons}_{n,n})$.
Under the identification $C_{ij}\mapsto(i\to j)$, the adjacency rule
\[
C_{ij}\to C_{kl}\iff j=k
\]
coincides exactly with the arc-composition rule $(i\to j)\to(j\to l)$ 
in $LD(K^{\circlearrowleft}_n)$. Hence
\begin{equation}
LD(K^{\rightleftharpoons}_{n,n})/\tau\cong LD(K^{\circlearrowleft}_n).
\end{equation}
\end{proof}
Next, we consider the vertex equitable partition of 
$K_{n,n}^{\rightleftharpoons}$ and show that its quotient graph 
is isomorphic to the complete digraph with loops $K_n^{\circlearrowleft}$.
\begin{lemma}\label{thm:Knn_quotient}
Let $K_{n,n}^{\rightleftharpoons}$ be the complete bipartite digraph with
bipartition $A = \{a_1, \dots, a_n\}$ and $B = \{b_1, \dots, b_n\}$.
Define the partition
\(
\pi \;=\; \{C_1, \dots, C_n\}, \qquad C_i \;=\; \{a_i,\, b_i\}, 
\quad i = 1, \dots, n.
\)
Then $\pi$ is an equitable partition of $K_{n,n}^{\rightleftharpoons}$,
and the corresponding quotient digraph satisfies
\[
K_{n,n}^{\rightleftharpoons}/\pi \;\cong\; K_n^{\circlearrowleft},
\]
where $K_n^{\circlearrowleft}$ denotes the complete digraph on $n$
vertices with a loop at every vertex. Moreover, the normalized
characteristic matrix of $\pi$ is given by
\[
Q \;=\; \frac{1}{\sqrt{2}}
\begin{pmatrix} I_n \\ I_n \end{pmatrix},
\]
where $I_n$ is the identity matrix of order $n$.
\end{lemma}

\begin{proof}
  
Fix any two cells $C_i, C_j \in \pi$. Since $a_i$ has out-neighbours
$\{b_1,\dots,b_n\}$ and $b_j$ is the unique member of $B\cap C_j$, vertex
$a_i$ has exactly one out-neighbour in $C_j$. Likewise $b_i$ has
out-neighbours $\{a_1,\dots,a_n\}$ and $a_j$ is the unique member of
$A\cap C_j$, so $b_i$ also has exactly one out-neighbour in $C_j$. This
count is $1$ for every pair $(i,j)$, including $i=j$, so $\pi$ is equitable
with quotient matrix
\[
B_{ij} \;=\; 1 \quad\text{for all } i,j, \qquad\text{that is, } B = J_n.
\]
Each cell $C_i = \{a_i, b_i\}$ has size $d = 2$, so the normalized
characteristic matrix is
\[
Q \;=\; \frac{1}{\sqrt{2}}
\begin{pmatrix} I_n \\ I_n \end{pmatrix} \;\in\; \mathbb{R}^{2n\times n},
\]
where the $i$-th column has entry $\tfrac{1}{\sqrt{2}}$ at rows $a_i$
and $b_i$, and zero elsewhere. Fix any $i \in [n]$. The $i$-th column of $A_G Q$ is
\[
A_{K_{n,n}^{\rightleftharpoons}} Q\, e_i
\;=\; \frac{1}{\sqrt{2}}\,A_{K_{n,n}^{\rightleftharpoons}}\,\mathbf{1}_{C_i},
\]
where $\mathbf{1}_{C_i} \in \mathbb{R}^{2n}$ denotes the indicator vector
of the cell $C_i$, with entry $1$ at vertices $a_i, b_i$ and $0$ elsewhere.
Since $a_i$ sends an arc to every $b_j$ and $b_i$ sends an arc to every
$a_j$, every vertex of $K_{n,n}^{\rightleftharpoons}$ receives exactly one arc from $C_i$, so
\(
A_{K_{n,n}^{\rightleftharpoons}}\,\mathbf{1}_{C_i} \;=\; \mathbf{1}_{V(K_{n,n}^{\rightleftharpoons})}.
\)
The $i$-th column of $Q J_n$ is
\[
Q J_n\, e_i
\;=\; \frac{1}{\sqrt{2}}
\begin{pmatrix} I_n \\ I_n \end{pmatrix}
\mathbf{1}_n
\;=\; \frac{1}{\sqrt{2}}\,\mathbf{1}_{V(K_{n,n}^{\rightleftharpoons})},
\]
since $J_n e_i = \mathbf{1}_n$ and each vertex of $G$ belongs to exactly
one cell, so $\sum_{j=1}^n \mathbf{1}_{C_j} = \mathbf{1}_{V(K_{n,n}^{\rightleftharpoons})}$.
Both columns equal $\tfrac{1}{\sqrt{2}}\,\mathbf{1}_{V(K_{n,n}^{\rightleftharpoons})}$ for every
$i \in [n]$, hence $A_{K_{n,n}^{\rightleftharpoons}}\,Q = Q\,J_n$. Thus, the quotient graph $K_{n,n}^{\rightleftharpoons}/\pi$ has adjacency matrix $B = J_n$, which has a $1$
in every position including the diagonal. Hence $K_{n,n}^{\rightleftharpoons}/\pi$ has an arc from
$C_i$ to $C_j$ for every $(i,j)\in[n]^2$, including $i=j$, giving
\begin{equation}
K_{n,n}^{\rightleftharpoons}/\pi \cong K_n^{\circlearrowleft}.
\end{equation}
\end{proof}

\begin{ex}\label{ex:K44-full}
Consider the complete bipartite digraph $G = {K}^\rightleftharpoons_{4,4}$
with vertex set $A\cup B$, where $A=\{a_1,a_2,a_3,a_4\}$ and
$B=\{b_1,b_2,b_3,b_4\}$, and arc set
\[
\mathcal{A}(G)
= \{(a_i,b_j)\mid 1\le i,j\le 4\}
\cup \{(b_i,a_j)\mid 1\le i,j\le 4\},
\qquad |\mathcal{A}(G)|=32.
\]

Define the arc partition
\[
\tau=\{C_{ij} : 1\le i,j\le 4\},\quad
C_{ij}=\{(a_i,b_j),(b_i,a_j)\},
\]
so $|\tau|=16$ cells each of size $2$.
The cells are displayed in the following table.

\[
\renewcommand{\arraystretch}{1.3}
\begin{array}{c|c|c|c|c}
 & j=1 & j=2 & j=3 & j=4 \\
\hline
i=1 &
T_1:\,(a_1,b_1),(b_1,a_1) &
T_2:\,(a_1,b_2),(b_1,a_2) &
T_3:\,(a_1,b_3),(b_1,a_3) &
T_4:\,(a_1,b_4),(b_1,a_4)\\
\hline
i=2 &
T_5:\,(a_2,b_2),(b_2,a_2) &
T_6:\,(a_2,b_3),(b_2,a_3) &
T_7:\,(a_2,b_4),(b_2,a_4)&
T_8:\,(a_2,b_1),(b_2,a_1) \\

\hline
i=3 &
T_{9}:\,(a_3,b_3),(b_3,a_3) &
T_{10}:\,(a_3,b_4),(b_3,a_4)&
T_{11}:\,(a_3,b_1),(b_3,a_1) &
T_{12}:\,(a_3,b_2),(b_3,a_2) \\

\hline
i=4 &
T_{13}:\,(a_4,b_4),(b_4,a_4)&
T_{14}:\,(a_4,b_1),(b_4,a_1) &
T_{15}:\,(a_4,b_2),(b_4,a_2) &
T_{16}:\,(a_4,b_3),(b_4,a_3) 
\end{array}
\]

Since $|\mathcal{A}(G)|=32$ and each cell has size $2$, the normalized
characteristic matrix is
\[
Q_\tau\in\; \mathbb{R}^{32\times 16},
\]
One verifies $Q_\tau^\top Q_\tau = I_{16}$. By Lemma~\ref{line digraph 1}, $\tau$ is an equitable partition of $LD(G)$
and the quotient adjacency matrix
\[
A' = Q_\tau^\top\, A_{LD(G)}\, Q_\tau \;\in\;\mathbb{R}^{16\times 16}
\]
satisfies $A_{LD(G)}\,Q_\tau = Q_\tau\,A'$.
Ordering the cells as
$T_1,\dots,T_{16}$ (row-major in $i$, then $j$), a direct computation gives
\[
A'=
\renewcommand{\arraystretch}{0.8}
\left[
\begin{array}{@{}cccc|cccc|cccc|cccc@{}}
1&1&1&1&0&0&0&0&0&0&0&0&0&0&0&0\\
0&0&0&0&1&1&1&1&0&0&0&0&0&0&0&0\\
0&0&0&0&0&0&0&0&1&1&1&1&0&0&0&0\\
0&0&0&0&0&0&0&0&0&0&0&0&1&1&1&1\\
\hline
0&0&0&0&1&1&1&1&0&0&0&0&0&0&0&0\\
0&0&0&0&0&0&0&0&1&1&1&1&0&0&0&0\\
0&0&0&0&0&0&0&0&0&0&0&0&1&1&1&1\\
1&1&1&1&0&0&0&0&0&0&0&0&0&0&0&0\\
\hline
0&0&0&0&0&0&0&0&1&1&1&1&0&0&0&0\\
0&0&0&0&0&0&0&0&0&0&0&0&1&1&1&1\\
1&1&1&1&0&0&0&0&0&0&0&0&0&0&0&0\\
0&0&0&0&1&1&1&1&0&0&0&0&0&0&0&0\\
\hline
0&0&0&0&0&0&0&0&0&0&0&0&1&1&1&1\\
1&1&1&1&0&0&0&0&0&0&0&0&0&0&0&0\\
0&0&0&0&1&1&1&1&0&0&0&0&0&0&0&0\\
0&0&0&0&0&0&0&0&1&1&1&1&0&0&0&0
\end{array}
\right].
\]

Let $G'$ denote the quotient graph with adjacency matrix $A'$.
Note that $LD(G)/\tau\ncong G$ in general; indeed
$G'=LD(G)/\tau\cong LD({K}_4^\circlearrowleft)$.
Define the equitable vertex partition
\[
\pi=\{C_1,C_2,C_3,C_4\},\qquad C_k=\{a_k,b_k\},
\]
with normalized characteristic matrix $Q\in\mathbb{R}^{8\times 4}$.
Because every vertex in $C_k$ has out-arcs to all vertices in every
$C_l$ (including $l=k$), the quotient is
\[
G/\pi \;\cong\; {K}_4^\circlearrowleft
\]
with adjacency matrix $A(G/\pi)=J_4$ (the $4\times 4$ all-ones matrix)(see Figure~\ref{fig:K4-quotient}). The complete digraph ${K}_4^\circlearrowleft$ has $4\times 4=16$ arcs
(including self-loops), with arc set
$\mathcal{A}(G/\pi)=\{(i,j):1\le i,j\le 4\}$. The line digraph $LD(G/\pi)$ has these 16 arcs as vertices, with adjacency
matrix (arcs ordered $(1,1),(1,2),(1,3),(1,4),(2,2),(2,3),(2,4),(2,1),(3,3),(3,4),(3,1),(3,2),(4,4),(4,1),(4,2),(4,3)$), shown in
Figure~\ref{fig:partition-full}, we have 
\[
A_{LD(G/\pi)}
= A'\]
 
Now define the  arc partition $\sigma=\{\sigma_1,\sigma_2,\sigma_3,\sigma_4\}$
of $LD(G/\pi)$ by grouping arcs by their terminal vertex
\[
\sigma_j = \{(i,j) : 1\le i\le 4\},\quad j=1,2,3,4.
\]
Each $\sigma_j$ has size $4$, so the normalized characteristic matrix is
\[
 \widetilde{Q}
\;\in\;\mathbb{R}^{16\times 4},
\]
Explicitly,
\[
\widetilde{Q}=
\begin{bmatrix}
\tfrac{1}{2}&0&0&0\\
0&\tfrac{1}{2}&0&0\\
0&0&\tfrac{1}{2}&0\\
0&0&0&\tfrac{1}{2}\\[2pt]
0&\tfrac{1}{2}&0&0\\
0&0&\tfrac{1}{2}&0\\
0&0&0&\tfrac{1}{2}\\
\tfrac{1}{2}&0&0&0\\[2pt]
0&0&\tfrac{1}{2}&0\\
0&0&0&\tfrac{1}{2}\\
\tfrac{1}{2}&0&0&0\\
0&\tfrac{1}{2}&0&0\\[2pt]
0&0&0&\tfrac{1}{2}\\
\tfrac{1}{2}&0&0&0\\
0&\tfrac{1}{2}&0&0\\
0&0&\tfrac{1}{2}&0
\end{bmatrix}.
\]
A direct computation confirms:
\[
\widetilde{A}
\;=\;
\widetilde{Q}^\top\,A_{LD(G/\pi)}\,\widetilde{Q}
\;=\; J_4,
\]
and $A_{LD(G/\pi)}\,\widetilde{Q} = \widetilde{Q}\,\widetilde{A}$,
so $\sigma$ is equitable for $LD(G/\pi)$ with quotient
\[
LD(G/\pi)/\sigma \;\cong\; {K}_4 ^\circlearrowleft\;\cong\; G/\pi.
\]
This verifies $LD(G/\pi)/\sigma\cong G/\pi$. Comparing the matrices above, we see
\[
{A_{LD(G/\pi)} \;=\; A'},
\]
that is,\ the line-digraph adjacency matrix of the quotient graph $G/\pi$
equals the quotient adjacency matrix of $LD(G)$ under $\tau$.
Consequently, by Theorem~\ref{lemma_auxiliary_pst}, PST on  $LD(G/\pi)$ is
equivalent to PST on $LD(G)$.

\noindent\textbf{Remark.}
Note that $\widetilde{Q}\neq Q_\pi$: the former lives in
$\mathbb{R}^{16\times 4}$ (arcs of $G/\pi$ vs.\ cells of $\sigma$),
while the latter lives in $\mathbb{R}^{8\times 4}$ (vertices of $G$
vs.\ cells of $\pi$). Hence one cannot directly transfer PST results from $G/\pi$ to $G$ or vice versa via $\widetilde{Q}$
alone; the equivalence is instead mediated by the identity
$A_{LD(G/\pi)}=A'$ and Theorem~\ref{lemma_auxiliary_pst}.

\end{ex}
Following Lemma~\ref{thm:Knn_quotient} and Lemma~\ref{Bipartite Graph}, 
we establish PST in the lifted line digraph 
$LD(K_{n,n}^{\rightleftharpoons})$ by applying 
Theorem~\ref{thm:PST} and Theorem~\ref{lemma_auxiliary_pst}, 
as shown in Theorem~\ref{thm:main_PST} below.
\begin{Theorem}
\label{thm:main_PST}
Let $K^{\rightleftharpoons}_{n,n}$ be the complete bipartite digraph, where \(n=2^s\) for some integer \(s\geq 1\). Let
$\tau = \{T_1, \ldots, T_{2m}\}$, $m = 2n$, be an equitable arc partition of
$\mathrm{LD}(K^{\rightleftharpoons}_{n,n})$, with normalized characteristic
matrix $Q_\tau$. Let $\pi = \{C_1, \ldots, C_n\}$ be an equitable vertex
partition of $K^{\rightleftharpoons}_{n,n}$ with $C_i = \{a_i, b_i\}$, and let
$\sigma = \{\sigma_1, \ldots, \sigma_n\}$ be an equitable arc partition of
$\mathrm{LD}(K^{\rightleftharpoons}_{n,n}/\pi)$ with normalized characteristic
matrix $\widetilde{Q}$. Let $U_\tau$ and $U_\sigma$ be the transition matrices
of $\mathrm{LD}(K^{\rightleftharpoons}_{n,n})$ and $\mathrm{LD}(K^{\rightleftharpoons}_{n,n}/\pi)$, respectively. Then:
\begin{enumerate}
    \item $U_\sigma$ exhibits PST in $\mathrm{LD}(K^{\rightleftharpoons}_{n,n}/\pi)$
    from $(I_d \otimes \widetilde{Q} )\,x$ to $(I_d \otimes \widetilde{Q} )\,y$,
    for all $x, y \in \mathrm{Im}(I_d \otimes \widetilde{Q})$, at step $k = n$
    and is periodic at step $2n$.

    \item $U_\tau$ exhibits PST in $\mathrm{LD}(K^{\rightleftharpoons}_{n,n})$ from
    $(I_d \otimes Q_\tau )( I_d \otimes\widetilde{Q} )\,x$ to
    $(I_d \otimes Q_\tau )( I_d \otimes \widetilde{Q})\,y$,
    for all $x, y \in \mathrm{Im}(I_d\otimes \widetilde{Q} )$,
    at step $k = n$ and is periodic at step $k=2n$.
\end{enumerate}
\end{Theorem}

\begin{proof}
Let $\pi = \{C_1, C_2, \ldots, C_n\}$ be an equitable vertex partition where
$C_i = \{a_i, b_i\}$. By Lemma~\ref{thm:Knn_quotient}, we have
\[
    K^{\rightleftharpoons}_{n,n}/\pi \;\cong\; K_n^{\circlearrowleft},
    \label{eq:quotient_iso}
\]
the complete graph with a loop at every vertex. Let
$\sigma = \{\sigma_1, \ldots, \sigma_n\}$ be an equitable arc partition of
$\mathrm{LD}(K^{\rightleftharpoons}_{n,n}/\pi)$. Then by Theorem~\ref{Isomorphism},
\[
    \mathrm{LD}(K^{\rightleftharpoons}_{n,n}/\pi)/\sigma \;\cong\; K^{\rightleftharpoons}_{n,n}/\pi \;\cong\; K_n^{\circlearrowleft}.
\]
Since $\sigma$ is an equitable arc partition, by Lemma~\ref{line digraph 2} we have
\[
    A_{\mathrm{LD}(K^{\rightleftharpoons}_{n,n}/\pi)}\,\widetilde{Q} \;=\; \widetilde{Q}\,\widetilde{A},
\]
where $\widetilde{A}$ is the adjacency matrix of $K_n^{\circlearrowleft}$.
By relabeling the arcs of $K^{\rightleftharpoons}_{n,n}/\pi$, the normalized characteristic matrix takes
the form
\[
    \widetilde{Q} \;=\; \frac{1}{\sqrt{d}}
    \begin{bmatrix}
        I_n \\
        \widetilde{P} \\
        \widetilde{P}^2 \\
        \vdots \\
        \widetilde{P}^{d-1}
    \end{bmatrix},
\]
where
\[
    I_n + \widetilde{P} + \widetilde{P}^2 + \cdots + \widetilde{P}^{d-1}
    \;=\; A(K_n^{\circlearrowleft}),
\]
that is, $\widetilde{P}, \ldots, \widetilde{P}^{d-1}$ are the cyclic shifts
(permutation matrices) of the adjacency matrix of $K_n^{\circlearrowleft}$.

By Theorem~\ref{thm:PST}, $K_n^{\circlearrowleft}$ exhibits PST at step $k = n$ and is periodic at step $2n$. Moreover, by
Lemma~\ref{Bipartite Graph},
\[
    \mathrm{LD}(K^{\rightleftharpoons}_{n,n})/\tau
    \;\cong\;
    \mathrm{LD}(K^{\circlearrowleft}_n).
\]
Therefore, by Theorem~\ref{lemma_auxiliary_pst}, the transition matrix
$U_\sigma$ of $\mathrm{LD}(K^{\rightleftharpoons}_{n,n}/\pi)$ exhibits PST from
$(I_d\otimes \widetilde{Q} )\,x$ to $(I_d \otimes \widetilde{Q})\,y$, for all
$x, y \in \mathrm{Im}(I_d\otimes \widetilde{Q} )$, at step $k = n$ and is
periodic at step $2n$. Furthermore, the transition matrix $U_\tau$ of
$\mathrm{LD}(K^{\rightleftharpoons}_{n,n})$ exhibits PST from
\[
    (I_d\otimes Q_\tau )(I_d \otimes \widetilde{Q})\,x
    \quad \text{to} \quad
    (I_d \otimes Q_\tau )(I_d \otimes \widetilde{Q} )\,y,
\]
for all $x, y \in \mathrm{Im}(I_d \otimes \widetilde{Q} )$, at step $k = n$
and is periodic at step $2n$.
\end{proof}
\section{Construction of \texorpdfstring{$LD(\operatorname{Circ}(2n, S))$}{LD(Circ(2n,S))}, 
where \texorpdfstring{$S$}{S} consists of all odd residues modulo 
\texorpdfstring{$2n$}{2n}, with PST for 
\texorpdfstring{$n = 2^s,\, s \geq 0$}{n = 2\^{}s, s >= 0}}
\label{section 10}
\begin{figure}[t]
\centering

\begin{minipage}[ht]{0.45\textwidth}
\centering
\begin{tikzpicture}[
  scale=0.6,
  >=stealth,
  every node/.style={circle, draw, font=\scriptsize, inner sep=2pt,
                     minimum size=0.55cm}]

\node[fill=blue!20,   draw=blue!70!black]   (v0) at ( 90:2.6) {$0$};
\node[fill=red!20,    draw=red!70!black]    (v1) at ( 45:2.6) {$1$};
\node[fill=violet!20, draw=violet!80!black] (v2) at (  0:2.6) {$2$};
\node[fill=blue!20,   draw=blue!70!black]   (v3) at (-45:2.6) {$3$};
\node[fill=green!20,  draw=green!60!black]  (v4) at (-90:2.6) {$4$};
\node[fill=violet!20, draw=violet!80!black] (v5) at (-135:2.6){$5$};
\node[fill=red!20,    draw=red!70!black]    (v6) at (180:2.6) {$6$};
\node[fill=green!20,  draw=green!60!black]  (v7) at (135:2.6) {$7$};

\foreach \u/\v in {0/1,0/3,0/5,0/7,
                   1/2,1/4,1/6,
                   2/3,2/5,2/7,
                   3/4,3/6,
                   4/5,4/7,
                   5/6,6/7}
  \draw[black, thin] (v\u) -- (v\v);

\end{tikzpicture}
\subcaption{$\mathrm{Circ}(8,\{1,3,5,7\})$  
with vertices coloured by partition class.}
\label{fig:circ8-circular}
\end{minipage}
\hfill
\begin{minipage}[t]{0.45\textwidth}
\centering
\begin{tikzpicture}[
  scale=0.5,
  > = {Stealth[length=3pt,bend]},
  every node/.style = {circle, draw, font=\scriptsize,
                       inner sep=1.5pt, minimum size=0.55cm}]

\draw[blue!60,        dashed, rounded corners=8pt, line width=0.8pt]
      (-5.4, 0.4) rectangle (-2.4, 4.0);
\draw[red!60,         dashed, rounded corners=8pt, line width=0.8pt]
      ( 2.4, 0.4) rectangle ( 5.4, 4.0);
\draw[green!60!black, dashed, rounded corners=8pt, line width=0.8pt]
      (-5.4,-4.0) rectangle (-2.4,-0.4);
\draw[violet!80!black,dashed, rounded corners=8pt, line width=0.8pt]
      ( 2.4,-4.0) rectangle ( 5.4,-0.4);

\node[draw=none,font=\scriptsize\bfseries,text=blue!70!black]
      at (-3.9, 3.72) {$C_1$};
\node[draw=none,font=\scriptsize\bfseries,text=red!70!black]
      at ( 3.9, 3.72) {$C_2$};
\node[draw=none,font=\scriptsize\bfseries,text=green!60!black]
      at (-3.9,-3.72) {$C_3$};
\node[draw=none,font=\scriptsize\bfseries,text=violet!80!black]
      at ( 3.9,-3.72) {$C_4$};

\node[fill=blue!20,   draw=blue!70!black]   (v0) at (-4.6, 3.0) {$0$};
\node[fill=blue!20,   draw=blue!70!black]   (v3) at (-3.2, 1.2) {$3$};
\node[fill=red!20,    draw=red!70!black]    (v1) at ( 4.6, 3.0) {$1$};
\node[fill=red!20,    draw=red!70!black]    (v6) at ( 3.2, 1.2) {$6$};
\node[fill=green!20,  draw=green!60!black]  (v4) at (-3.2,-1.2) {$4$};
\node[fill=green!20,  draw=green!60!black]  (v7) at (-4.6,-3.0) {$7$};
\node[fill=violet!20, draw=violet!80!black] (v2) at ( 4.6,-3.0) {$2$};
\node[fill=violet!20, draw=violet!80!black] (v5) at ( 3.2,-1.2) {$5$};

\draw[->,blue!70,thin] (v0) to[bend left=10]  (v1);
\draw[->,blue!70,thin] (v0) to[bend left=10]  (v3);
\draw[->,blue!70,thin] (v0) to[bend left=10]  (v5);
\draw[->,blue!70,thin] (v0) to[bend left=10]  (v7);
\draw[->,blue!70,thin] (v3) to[bend right=10] (v4);
\draw[->,blue!70,thin] (v3) to[bend left=10]  (v6);
\draw[->,blue!70,thin] (v3) to[bend left=10]  (v0);
\draw[->,blue!70,thin] (v3) to[bend right=10] (v2);

\draw[->,red!60,thin]  (v1) to[bend left=10]  (v2);
\draw[->,red!60,thin]  (v1) to[bend left=5]   (v4);
\draw[->,red!60,thin]  (v1) to[bend left=10]  (v6);
\draw[->,red!60,thin]  (v1) to[bend left=10]  (v0);
\draw[->,red!60,thin]  (v6) to[bend left=5]   (v7);
\draw[->,red!60,thin]  (v6) to[bend left=10]  (v1);
\draw[->,red!60,thin]  (v6) to[bend left=10]  (v3);
\draw[->,red!60,thin]  (v6) to[bend left=10]  (v5);

\draw[->,green!60!black,thin] (v4) to[bend right=10] (v5);
\draw[->,green!60!black,thin] (v4) to[bend right=10] (v7);
\draw[->,green!60!black,thin] (v4) to[bend left=5]   (v1);
\draw[->,green!60!black,thin] (v4) to[bend right=10] (v3);
\draw[->,green!60!black,thin] (v7) to[bend left=10]  (v0);
\draw[->,green!60!black,thin] (v7) to[bend left=10]  (v2);
\draw[->,green!60!black,thin] (v7) to[bend right=10] (v4);
\draw[->,green!60!black,thin] (v7) to[bend left=5]   (v6);

\draw[->,violet!70,thin] (v2) to[bend right=10] (v3);
\draw[->,violet!70,thin] (v2) to[bend right=10] (v5);
\draw[->,violet!70,thin] (v2) to[bend left=10]  (v7);
\draw[->,violet!70,thin] (v2) to[bend left=10]  (v1);
\draw[->,violet!70,thin] (v5) to[bend left=10]  (v6);
\draw[->,violet!70,thin] (v5) to[bend left=5]   (v0);
\draw[->,violet!70,thin] (v5) to[bend right=10] (v2);
\draw[->,violet!70,thin] (v5) to[bend right=10] (v4);

\end{tikzpicture}
\subcaption{Partition layout of $\mathrm{Circ}(8,\{1,3,5,7\})$: 
$C_1 = \{0,3\}$ (blue), $C_2 = \{1,6\}$ (red), 
$C_3 = \{4,7\}$ (green), $C_4 = \{2,5\}$ (violet).}
\label{fig:circ8-partition}
\end{minipage}

\vspace{4pt}
\begin{minipage}[t]{\textwidth}
\centering
\begin{tikzpicture}[
  scale=0.55,
  >=stealth,
  every node/.style={circle, draw, font=\scriptsize, inner sep=3pt,
                     minimum size=0.85cm}]

\node[fill=blue!20,   draw=blue!70!black]   (C1) at (-2.2,  2.2) {$C_1$};
\node[fill=red!20,    draw=red!70!black]    (C2) at ( 2.2,  2.2) {$C_2$};
\node[fill=green!20,  draw=green!60!black]  (C3) at (-2.2, -2.2) {$C_3$};
\node[fill=violet!20, draw=violet!80!black] (C4) at ( 2.2, -2.2) {$C_4$};

\draw[->, blue!70!black,  thick] (C1) to[out=120, in=60,  looseness=8] (C1);
\draw[->, red!70!black,   thick] (C2) to[out=60,  in=120, looseness=8] (C2);
\draw[->, green!60!black, thick] (C3) to[out=240, in=180, looseness=8] (C3);
\draw[->, violet!80!black,thick] (C4) to[out=300, in=0,   looseness=8] (C4);

\draw[->, blue!70, thick] (C1) to[bend left=18]  (C2);
\draw[->, blue!70, thick] (C1) to[bend right=18] (C4);
\draw[->, blue!70, thick] (C1) to[bend right=18] (C3);
\draw[->, red!70,  thick] (C2) to[bend left=18]  (C3);
\draw[->, red!70,  thick] (C2) to[bend left=18]  (C1);
\draw[->,red!70,  thick] (C2) to[bend right=18] (C4);
\draw[->,green!60!black, thick] (C3) to[bend left=18]  (C4);
\draw[->,green!60!black, thick] (C3) to[bend right=18] (C1);
\draw[->,green!60!black, thick] (C3) to[bend left=18]  (C2);
\draw[->,violet!70, thick] (C4) to[bend right=18] (C1);
\draw[->,violet!70, thick] (C4) to[bend right=18] (C2);
\draw[->,violet!70, thick] (C4) to[bend left=18]  (C3);

\node[draw=none, font=\scriptsize, text=gray!80] at ( 0.0,  2.9) {$$};
\node[draw=none, font=\scriptsize, text=gray!80] at ( 2.9,  0.0) {$$};
\node[draw=none, font=\scriptsize, text=gray!80] at ( 0.0, -2.9) {$$};
\node[draw=none, font=\scriptsize, text=gray!80] at (-2.9,  0.0) {$$};
\node[draw=none, font=\scriptsize, text=gray!80] at ( 0.0,  0.5) {$$};
\node[draw=none, font=\scriptsize, text=gray!80] at ( 0.0, -0.5) {$$};

\end{tikzpicture}
\subcaption{Quotient digraph $\mathrm{Circ}(8,\{1,3,5,7\})/\pi 
\cong K_4^{\circlearrowleft}$, obtained by collapsing each cell to a 
single vertex. Every cell carries a loop and all four cells are 
mutually adjacent with multiplicity~$2$.}
\label{fig:circ8-quotient}
\end{minipage}

\caption{Three representations of $\mathrm{Circ}(8,\{1,3,5,7\})$ 
under the equitable partition $\pi = \{C_1, C_2, C_3, C_4\}$, where 
$C_1 = \{0,3\}$, $C_2 = \{1,6\}$, $C_3 = \{4,7\}$, and 
$C_4 = \{2,5\}$. The corresponding quotient digraph is 
$K_4^{\circlearrowleft}$.}
\label{fig:circ8-three-layouts}
\end{figure}
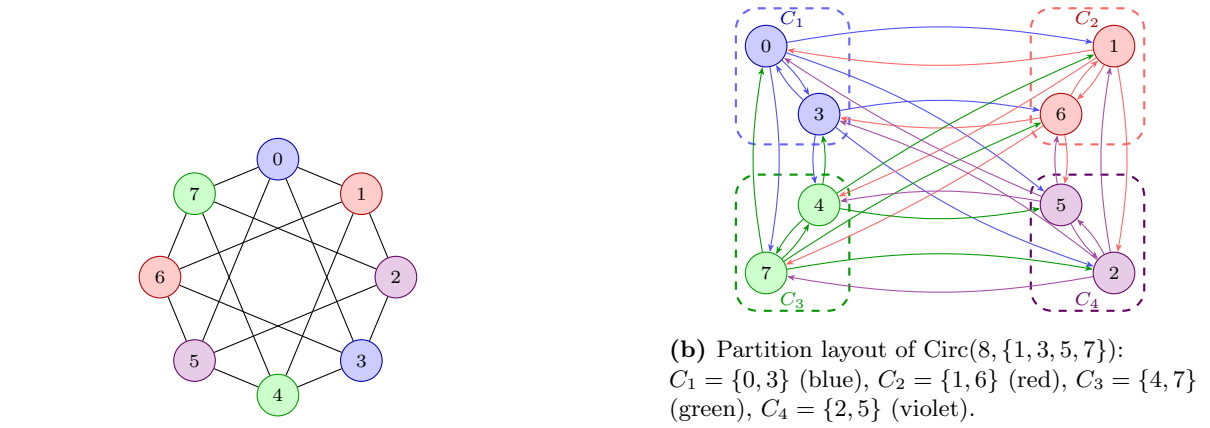
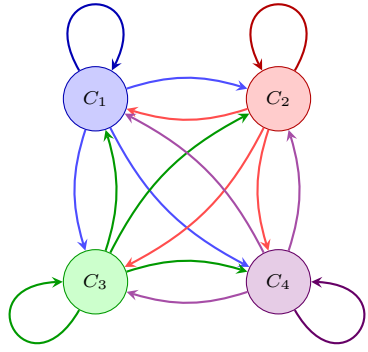

A \emph{circulant graph} $\mathrm{Circ}(n, S)$ is an undirected graph
on vertex set $\mathbb{Z}_n = \{0, 1, \dots, n-1\}$ in which each
vertex $v$ is adjacent to $v + s \pmod{n}$ for every $s \in S$, where
$S \subseteq \mathbb{Z}_n \setminus \{0\}$ is symmetric, that is,
$s \in S \Rightarrow -s \in S$. The graph is $|S|$-regular. When $n$
is even and $S$ consists of all odd residues modulo $n$, we write
$\mathrm{Circ}(2n, n)$; this yields an $n$-regular graph in which
every vertex is adjacent to all vertices at odd distance modulo $n$.

In this section we partition the vertices of $\mathrm{Circ}(2n, S)$ into $ n$
cells $\pi = \{C_1, \dots, C_n\}$, each containing exactly one even
and one odd element of $\mathbb{Z}_{2n}$. Hence \(C_{2n}/\pi\) is 2-regular quotient graph by Corollary~\ref{cor:3.1}. Since every $s \in S$ is odd,
adding $s$ to any vertex flips its parity, so each vertex has exactly
one neighbour in every cell. This uniform one-per-cell property makes
$\pi$ an equitable partition, whose quotient adjacency matrix is
determined in the lemma below.

\begin{lemma}\label{lem:circulant-complete-quotient}
Let $G = \mathrm{Circ}(2n, S)$ where
$S = \{1, 3, \dots, 2n-1\}$ is the set of all odd residues
modulo $2n$. Partition $\mathbb{Z}_{2n}$ into
$n$ cells $\pi = \{C_1, \dots, C_n\}$ where each cell contains
exactly one even and one odd element of $\mathbb{Z}_{2n}$. Then
\[
A(G/\pi) = J_{n \times n},
\]
so the quotient multigraph $G/\pi$ is the complete graph $K_n^\circlearrowleft$
augmented with a loop at every vertex.
\end{lemma}

\begin{proof}
Write $\mathbb{Z}_{2n} = \mathcal{E} \sqcup \mathcal{O}$ where
$\mathcal{E} = \{0,2,\dots,2n-2\}$ and $\mathcal{O} = \{1,3,\dots,2n-1\}$
are the even and odd elements respectively, each of cardinality $n$.
By hypothesis, each cell $C_{r'}$ satisfies
$|C_{r'} \cap \mathcal{E}| = |C_{r'} \cap \mathcal{O}| = 1$,
so we may write $C_{r'} = \{e_{r'}, o_{r'}\}$ with
$e_{r'} \in \mathcal{E}$ and $o_{r'} \in \mathcal{O}$.
Fix $r \in \{1,\dots,n\}$ and let $v \in C_r$ be arbitrary.
Since every $s \in S = \mathcal{O}$ is odd,
\[
v + s \equiv v + 1 \pmod{2},
\]
so all $n$ neighbours $N(v) = \{v + s \pmod{2n} : s \in S\}$
lie entirely in $\mathcal{O}$ if $v \in \mathcal{E}$, and
entirely in $\mathcal{E}$ if $v \in \mathcal{O}$. Moreover,
since $0 \notin S$, the map $s \mapsto v + s \pmod{2n}$ is
injective on $S$, so $|N(v)| = |S| = n$. Now fix $r' \in \{1,\dots,n\}$. We count $|N(v) \cap C_{r'}|$
by considering two cases.

\textit{Case 1:} $v \in \mathcal{E}$. Then $N(v) \subseteq
\mathcal{O}$, so $N(v) \cap C_{r'} = N(v) \cap \{o_{r'}\}$.
The element $o_{r'} \in N(v)$ if and only if $o_{r'} - v
\pmod{2n} \in S$. Since $v \in \mathcal{E}$ and $o_{r'} \in
\mathcal{O}$, we have $v \neq o_{r'}$, so $o_{r'} - v
\not\equiv 0 \pmod{2n}$. Furthermore $o_{r'} - v$ is odd
(odd minus even), hence $o_{r'} - v \pmod{2n} \in \mathcal{O}
= S$. Thus $|N(v) \cap C_{r'}| = 1$.

\textit{Case 2:} $v \in \mathcal{O}$. Then $N(v) \subseteq
\mathcal{E}$, so $N(v) \cap C_{r'} = N(v) \cap \{e_{r'}\}$.
Since $v \in \mathcal{O}$ and $e_{r'} \in \mathcal{E}$,
we have $v \neq e_{r'}$, so $e_{r'} - v \not\equiv 0
\pmod{2n}$. Furthermore $e_{r'} - v$ is odd (even minus odd),
hence $e_{r'} - v \pmod{2n} \in \mathcal{O} = S$. Thus
$|N(v) \cap C_{r'}| = 1$.

In both cases $|N(v) \cap C_{r'}| = 1$, and this count
depends only on $r$ and $r'$, not on the particular choice
of $v \in C_r$. Hence by definition $\pi$ is an equitable
partition of $G$ with
\[
A(G/\pi)_{r,r'} = 1 \qquad \text{for all } r,r'
\in \{1,\dots,n\},
\]
giving \begin{equation} A(G/\pi) = J_{n\times n}\end{equation}. The diagonal entries
$A(G/\pi)_{r,r} = 1$ correspond to loops at each vertex
of $G/\pi$, so the quotient multigraph is $K_n^{\circlearrowleft}$.
\end{proof}
\begin{ex}
For $G=\mathrm{Circ}(8,\{1,3,5,7\})$ with $n=4$, the partition
$C_1=\{0,3\}$, $C_2=\{1,6\}$, $C_3=\{4,7\}$, $C_4=\{2,5\}$
satisfies the hypothesis with $e_1=0,o_1=3$; $e_2=6,o_2=1$;
$e_3=4,o_3=7$; $e_4=2,o_4=5$. Vertex $0\in\mathcal{E}$
has neighbours $\{1,3,5,7\}=\mathcal{O}$, hitting $o_2=1$,
$o_1=3$, $o_4=5$, $o_3=7$: exactly one per cell. Vertex
$3\in\mathcal{O}$ has neighbours $\{0,2,4,6\}=\mathcal{E}$,
hitting $e_1=0$, $e_4=2$, $e_3=4$, $e_2=6$: exactly one
per cell. Hence $A(G/\pi)=J_{4\times 4}$ ( see Figure~\ref{fig:circ8-three-layouts}). 
\end{ex}
We take the arc partition $\tau$ of $LD(\mathrm{Circ}(2n,S))$ and 
show that the quotient $LD(\mathrm{Circ}(2n,S))/\tau$ is isomorphic 
to the line digraph $LD(K_n^{\circlearrowleft})$ of the complete 
digraph with loops $K_n^{\circlearrowleft}$, as established in 
Lemma~\ref{ciculant:PST} below.
\begin{lemma}\label{ciculant:PST}
Let $G=\mathrm{Circ}(2n,S)$ with $S=\{1,3,5,\dots,2n-1\}$ all odd residues 
modulo $2n$, so that $|\mathcal{A}(G)|=2n^2$. Partition $\mathbb{Z}_{2n}$ 
into $n$ consecutive even--odd pairs
\[
m_i=\{e_i,o_i\}=\{2i-2,\;2i-1\},\quad 1\le i\le n.
\]
For each pair $(i,j)\in[n]^2$ define the cell
\[
C_{ij}=\{(e_i,\;o_j),\;(o_i,\;e_j)\}\subset\mathcal{A}(G).
\]
Then $\tau=\{C_{ij}:1\le i,j\le n\}$ is an equitable partition of 
$\mathcal{A}(G)$ into $n^2$ cells of size $2$, with quotient matrix 
$B\in\{0,1\}^{n^2\times n^2}$ whose $\big((i,j),(k,l)\big)$-entry is
\[
B_{(i,j),(k,l)}=\begin{cases}1 & \text{if }j=k,\\0 & \text{if }j\neq k.\end{cases}
\]
Furthermore,
\[
LD(\mathrm{Circ}(2n,S))/\tau\;\cong\;LD(K^{\circlearrowleft}_n).
\]
\end{lemma}

\begin{proof}

Observe that
\(
o_j - e_i = (2j-1)-(2i-2) = 2(j-i)+1 \in S,
\)
so $(e_i,o_j)\in\mathcal{A}(G)$, and
\(
e_j - o_i = (2j-2)-(2i-1) = 2(j-i)-1 \equiv 2(j-i)-1 \pmod{2n},
\)
which is odd and hence lies in $S$, so $(o_i,e_j)\in\mathcal{A}(G)$. 
Thus $|C_{ij}|=2$ for all $(i,j)\in[n]^2$. Since
\[
\mathcal{A}(G) 
= \{(v,w)\in\mathbb{Z}_{2n}^2 : w-v\in S\}
= \{(e_i,o_j):1\le i,j\le n\}
  \cup\{(o_i,e_j):1\le i,j\le n\},
\]
and the maps
\(
(e_i,o_j)\mapsto (i,j),\quad (o_i,e_j)\mapsto (i,j)
\)
are both bijections from their respective domains to $[n]^2$, each arc 
of $G$ belongs to exactly one cell $C_{ij}$. Hence
\(
\bigsqcup_{(i,j)\in[n]^2} C_{ij} = \mathcal{A}(G),
\)
so $\tau$ partitions $\mathcal{A}(G)$ into $n^2$ cells each of size $2$. For each cell $C_{ij}$,
\begin{equation}\label{circ:tailhead}
\mathrm{tail}(C_{ij})=\{e_i,o_i\}=m_i,\qquad
\mathrm{head}(C_{ij})=\{o_j,e_j\}=m_j.
\end{equation}
By definition of the line digraph, $C_{ij}\to C_{kl}$ in $LD(G)$ if and 
only if
\[
\mathrm{head}(C_{ij})\cap\mathrm{tail}(C_{kl})\neq\emptyset,
\]
that is, $m_j\cap m_k\neq\emptyset$, which holds if and only if $j=k$. 
Hence the out-neighbours of $C_{ij}$ are
\begin{equation}\label{circ:outnbr}
\{C_{jl}:1\le l\le n\},
\end{equation}
and every cell has out-degree exactly $n$. We must show that for every 
pair of cells $C_{ij}$ and $C_{kl}$, every arc $e\in C_{ij}$ points to 
the same number of arcs in $C_{kl}$, and that this number equals 
$B_{(i,j),(k,l)}$.

\textit{Case 1: $j\neq k$.}
By Equation ~\eqref{circ:outnbr}, no arc of $C_{ij}$ points to any arc of 
$C_{kl}$, so the count is $0$.

\textit{Case 2: $j=k$.}
Take any arc $e\in C_{ij}$. From Equation~\eqref{circ:tailhead}, $e$ has a 
unique head vertex $v\in\mathrm{head}(C_{ij})=m_j$, namely
\[
v=\begin{cases}o_j & \text{if }e=(e_i,o_j),\\ 
               e_j & \text{if }e=(o_i,e_j).\end{cases}
\]
Now consider cell $C_{jl}$ (since $k=j$). Its two arcs are $(e_j,o_l)$ 
and $(o_j,e_l)$, with tails $e_j$ and $o_j$ respectively. Exactly one 
of these tails equals $v$
\[
\begin{cases}
(e_j,o_l)\text{ has tail }e_j=v & \text{if }e=(o_i,e_j),\\
(o_j,e_l)\text{ has tail }o_j=v & \text{if }e=(e_i,o_j).
\end{cases}
\]
Hence exactly $1$ arc in $C_{jl}=C_{kl}$ follows $e$, and this count 
is independent of which arc $e\in C_{ij}$ we chose and independent of 
$i$ and $l$. Combining both cases, the number of arcs in $C_{kl}$ that any arc 
$e\in C_{ij}$ points to is
\[
B_{(i,j),(k,l)}=\begin{cases}1 & \text{if }j=k,\\0 & \text{if }j\neq k,\end{cases}
\]
which depends only on $j$ and $k$, not on $i$, $l$, or the choice of 
$e\in C_{ij}$. Therefore $\tau$ is an equitable partition with quotient 
matrix $B$ as stated. If $Q_\tau$ is the normalised characteristic matrix of $\tau$, then
\[
A_{LD(G)}\,Q_\tau=Q_\tau\,B,
\]
where $A_{LD(G)}$ is the adjacency matrix of $L(\mathrm{Circ}(2n,S))$.
Under the identification $C_{ij}\mapsto(i\to j)$, the adjacency rule
\[
C_{ij}\to C_{kl}\iff j=k
\]
coincides exactly with the arc-composition rule $(i\to j)\to(j\to l)$ 
in $L(K^{\circlearrowleft}_n)$. Hence
\begin{equation}
LD(\mathrm{Circ}(2n,S))/\tau\cong LD(K^{\circlearrowleft}_n).\qquad\square
\end{equation}
\end{proof}

\begin{ex}
Consider $G=\mathrm{Circ}(8,\{1,3,5,7\})$ on $\mathbb{Z}_8=\{0,1,\dots,7\}$,
so $n=4$ and $|\mathcal{A}(G)|=32$. The four even--odd pairs are
\[
m_1=\{0,1\},\quad m_2=\{2,3\},\quad m_3=\{4,5\},\quad m_4=\{6,7\}.
\]
The $16$ cells $C_{ij}=\{(e_i,o_j),(o_i,e_j)\}$ partitioning all $32$ 
arcs are displayed in the following table.
\[
\renewcommand{\arraystretch}{1.3}
\begin{array}{c|c|c|c|c}
 & j=1 & j=2 & j=3 & j=4 \\
\hline
i=1 &
T_1:\,(0,1),(1,0) &
T_2:\,(0,3),(1,2) &
T_3:\,(0,5),(1,4) &
T_4:\,(0,7),(1,6)\\
\hline
i=2 &
T_5:\,(2,3),(3,2) &
T_6:\,(2,5),(3,4) &
T_7:\,(2,7),(3,6)&
T_8:\,(2,1),(3,0) \\
\hline
i=3 &
T_9:\,(4,5),(5,4) &
T_{10}:\,(4,7),(5,6)&
T_{11}:\,(4,1),(5,0) &
T_{12}:\,(4,3),(5,2) \\

\hline
i=4 &
T_{13}:\,(6,7),(7,6)&
T_{14}:\,(6,1),(7,0) &
T_{15}:\,(6,3),(7,2) &
T_{16}:\,(6,5),(7,4) 

\end{array}
\]
Each cell $C_{ij}$ has $\mathrm{tail}(C_{ij})=m_i$ and 
$\mathrm{head}(C_{ij})=m_j$. For example,
\[
\mathrm{tail}(C_{23})=m_2=\{2,3\},\quad\mathrm{head}(C_{23})=m_3=\{4,5\},
\]
and the out-neighbours of $C_{23}$ in $LD(G)$ are $\{C_{31},C_{32},
C_{33},C_{34}\}=\{T_{11},T_{12},T_9,T_{10}\}$, exactly the four cells 
in row $i=3$. Every arc of $C_{23}$ points to exactly one arc of each 
of these cells, confirming $B_{(2,3),(3,l)}=1$ for $l=1,2,3,4$ and 
$B_{(2,3),(k,l)}=0$ for $k\neq 3$.
\end{ex}
Following Lemma~\ref{lem:circulant-complete-quotient} and Lemma~\ref{ciculant:PST}, 
we establish PST in the lifted line digraph 
$LD(\mathrm{Circ}(2n,\, S))$ by applying 
Theorem~\ref{thm:PST} and Theorem~\ref{lemma_auxiliary_pst}, 
as shown in Theorem~\ref{thm:main_PST_2} below.
\begin{Theorem}
\label{thm:main_PST_2}
Let $G = \mathrm{Circ}(2n,\, S)$ where $S = \{1, 3, \ldots, 2n-1\}$ is the
set of all odd residues modulo $2n$, be the circulant graph of degree $d$,
where $n = 2^s$ for some integer $s \geq 1$. Let
$\tau = \{T_1, \ldots, T_{2m}\}$, $m = 2n$, be an equitable arc partition of
$\mathrm{LD}(G)$, with normalized characteristic matrix $Q_\tau$. Let
$\pi = \{C_1, \ldots, C_n\}$ be an equitable vertex partition of $G$, where
each cell $C_i$ contains exactly one even and one odd element of
$\mathbb{Z}_{2n}$, and let $\sigma = \{\sigma_1, \ldots, \sigma_n\}$ be a  
equitable arc partition of $\mathrm{LD}(G/\pi)$ with normalized characteristic
matrix $\widetilde{Q}$. Let $U_\tau$ and $U_\sigma$ be the transition matrices
of $\mathrm{LD}(G)$ and $\mathrm{LD}(G/\pi)$, respectively. Then:
\begin{enumerate}
    \item $U_\sigma$ exhibits PST  in $\mathrm{LD}(G/\pi)$
    from $(I_d\otimes \widetilde{Q} )\,x$ to $(I_d \otimes \widetilde{Q} )\,y$,
    for all $x,\, y \in \mathrm{Im}(I_d \otimes \widetilde{Q} )$,
    at step $k = n$ and is periodic at step $2n$.

    \item $U_\tau$ exhibits PST in $\mathrm{LD}(G)$ from
    $(I_d \otimes Q_\tau )(I_d \otimes \widetilde{Q} )\,x$ to
    $(I_d \otimes Q_\tau )(I_d \otimes \widetilde{Q})\,y$,
    for all $x,\, y \in \mathrm{Im}(I_d \otimes \widetilde{Q} )$,
    at step $k = n$ and is periodic at step $2n$.
\end{enumerate}
\end{Theorem}

\begin{proof}
Let $\pi = \{C_1, C_2, \ldots, C_n\}$ be an equitable vertex partition of $G$
where each cell $C_i$ contains exactly one even and one odd element of
$\mathbb{Z}_{2n}$. By Lemma~\ref{lem:circulant-complete-quotient}, we have
\begin{equation}
    G/\pi \;\cong\; K_n^{\circlearrowleft},
    \label{eq:circ_quotient_iso}
\end{equation}
the complete graph with a loop at every vertex. Let
$\sigma = \{\sigma_1, \ldots, \sigma_n\}$ be an equitable arc partition of
$\mathrm{LD}(G/\pi)$. Then by Theorem~\ref{Isomorphism},
\[
    \mathrm{LD}(G/\pi)/\sigma \;\cong\; G/\pi \;\cong\; K_n^{\circlearrowleft}.
\]
Since $\sigma$ is an equitable arc partition, by Lemma~\ref{line digraph 2} we have
\[
    A_{\mathrm{LD}(G/\pi)}\,\widetilde{Q} \;=\; \widetilde{Q}\,\widetilde{A},
\]
where $\widetilde{A}$ is the adjacency matrix of $K_n^{\circlearrowleft}$.
By relabeling the arcs of $G/\pi$, the normalized characteristic matrix takes
the form
\[
    \widetilde{Q} \;=\; \frac{1}{\sqrt{d}}
    \begin{bmatrix}
        I_n \\[4pt]
        \widetilde{P} \\[4pt]
        \widetilde{P}^2 \\[4pt]
        \vdots \\[4pt]
        \widetilde{P}^{d-1}
    \end{bmatrix},
\]
where
\[
    I_n + \widetilde{P} + \widetilde{P}^2 + \cdots + \widetilde{P}^{d-1}
    \;=\; A\!\left(K_n^{\circlearrowleft}\right),
\]
that is, $\widetilde{P}, \widetilde{P}^2, \ldots, \widetilde{P}^{d-1}$ are the
cyclic shifts (permutation matrices) of the adjacency matrix of
$K_n^{\circlearrowleft}$. By Theorem~\ref{thm:PST}, $K_n^{\circlearrowleft}$ exhibits PST  at step $k = n$ and is periodic at step $2n$. Moreover, by
Lemma~\ref{ciculant:PST},
\[
    \mathrm{LD}(G)/\tau \;\cong\; \mathrm{LD}\!\left(K_n^{\circlearrowleft}\right).
\]
Therefore, by Theorem~\ref{lemma_auxiliary_pst}, the transition matrix
$U_\sigma$ of $\mathrm{LD}(G/\pi)$ exhibits PST from
$(I_d \otimes \widetilde{Q})\,x$ to $(I_d \otimes \widetilde{Q} )\,y$, for all
$x,\, y \in \mathrm{Im}(I_d \otimes \widetilde{Q})$, at step $k = n$ and is
periodic at step $2n$. Furthermore, the transition matrix $U_\tau$ of
$\mathrm{LD}(G)$ exhibits PST from
\[
    (I_d \otimes Q_\tau)(I_d \otimes \widetilde{Q} )\,x
    \quad\text{to}\quad
    (I_d \otimes Q_\tau )(I_d \otimes \widetilde{Q} )\,y,
\]
for all $x,\, y \in \mathrm{Im}(I_d \otimes \widetilde{Q} )$, at step $k = n$
and is periodic at step $2n$.
\end{proof}

\section{Conclusion}\label{section 11}
In this work, we developed a framework for studying perfect state 
transfer (PST) in quotient graphs based on the equitable partition 
and the shunt decomposition of a graph $G$. The main results are 
summarized in two cases.

In the first case, we defined the unitary evolution and established 
conditions under which PST on the quotient graph $G/\pi$ is equivalent 
to PST on $G$, using the Chebyshev representation of the unitary 
evolution for quotient graphs. We verified that the cycle graph 
$C_{2n}$ with $n = 2^s$, $s \geq 1$, satisfies this condition, and 
provided illustrative examples. In the second case, we addressed the unitary 
evolution in settings where PST in the quotient graph $G/\pi$ is not 
necessarily equivalent to PST in $G$. In this setting, we established 
conditions under which, graph \(G\) can be reduced to a quotient graph 
isomorphic to $K_n^{\circlearrowleft}$, then PST can be found in the 
line digraph $LD(G)$. To this end, we constructed two families of line 
digraphs with PST by ordering the arcs of $K_{n,n}^{\rightleftharpoons}$ 
and $\operatorname{Circ}(2n, S)$, where $S$ consists of all odd residues 
modulo $2n$ and $n = 2^s$ for some $s \geq 1$. 

Moreover, several promising directions remain for future research, 
including the exploration of PST in broader families of Cayley graphs, and to develop systematic methods for constructing larger graphs 
admitting PST from smaller ones for which PST is already established. It would also be of interest to 
investigate how the interplay between the shunt-decomposition 
configurations $S$ and $\widetilde{S}$ affects the entanglement 
entropy and average mixing matrices of the quantum walk. Finally, 
analysing the spectral gap of the reduced operators, which  may provide 
deeper insights into hitting times and the algorithmic speedups 
achievable in quantum walks on symmetric network topologies.
\printbibliography
\end{document}